\documentclass[10pt]{article}

\usepackage{mathrsfs}
\usepackage[dvips]{graphicx}
\usepackage{amsmath,amsthm, amscd}
\usepackage[psamsfonts]{amssymb}
\usepackage[all]{xy}
\usepackage{wrapfig}
\usepackage{bbm}
\usepackage{here}

\newtheorem{Thm}{Theorem}[section]
\newtheorem{Prop}[Thm]{Proposition}
\newtheorem{Lem}[Thm]{Lemma}
\newtheorem{Cor}[Thm]{Corollary}

\newtheorem{Claim}[Thm]{Claim}

\theoremstyle{remark}
\newtheorem{Rem}[Thm]{Remark}

\theoremstyle{definition}
\newtheorem{Def}[Thm]{Definition}
\newtheorem{Assum}[Thm]{Assumption}
\newtheorem{Exa}[Thm]{Example}

\newcommand{\mysection}[2]{%
\vspace{2mm}\section{\bf #1}\label{#2}
}

\def\Z{{\mathbb Z}}
\def\R{{\mathbb R}}
\def\Q{{\mathbb Q}}
\def\C{{\mathbb C}}
\def\calA{\mathscr{A}}

\def\calF{\mathscr{F}}

\def\calO{\mathscr{O}}

\def\calS{\mathscr{S}}

\def\g{\mathfrak{g}}
\def\End{\mathrm{End}}

\def\Sym{\mathrm{Sym}}
\def\Conf{\mathrm{Conf}}
\def\bConf{\overline{\Conf}}

\def\Tr{\mathrm{Tr}}
\def\Hol{\mathrm{Hol}}
\def\ve{\varepsilon}
\def\bvec#1{\mbox{\boldmath{$#1$}}}

\def\Diff{\mathrm{Diff}}
\def\Emb{\mathrm{Emb}}
\def\bcalA{\overline{\calA}}
\def\Map{\mathrm{Map}}

\def\fib{\mathrm{fib}}
\def\fr{\mathrm{fr}}
\def\tcoprod{\textstyle\coprod}
\def\tbigwedge{\textstyle\bigwedge}

\def\even{\mathrm{even}}

\def\Lk{\mathrm{Lk}}
\def\2vec#1#2{\begin{pmatrix}#1\\#2\end{pmatrix}}

\title{Theta-graph and diffeomorphisms of some 4-manifolds}
\author{Tadayuki Watanabe\thanks{Department of Mathematics, Kyoto University.}}
\date{\today}

\begin{document}

\maketitle
\setcounter{tocdepth}{2}
\numberwithin{equation}{section}
\renewcommand{\thefootnote}{\fnsymbol{footnote}}

\begin{abstract}
In this article, we construct countably many mutually non-isotopic diffeomorphisms of some closed non simply-connected 4-manifolds that are homotopic to but not isotopic to the identity, by surgery along $\Theta$-graphs. As corollaries of this, we obtain some new results on codimension 1 embeddings and pseudo-isotopies of 4-manifolds. In the proof of the non-triviality of the diffeomorphisms, we utilize a twisted analogue of Kontsevich's characteristic class for smooth bundles, which is obtained by extending a higher dimensional analogue of March\'{e}--Lescop's ``equivariant triple intersection'' in configuration spaces of 3-manifolds to allow Lie algebraic local coefficient system.
\end{abstract}
\par\vspace{5mm}
\tableofcontents
%\clearpage

\def\baselinestretch{1.06}\small\normalsize
%%%%%%%%%%%%%%%%%%%%%%%%%%%%%%%%%%%%%%%%%%%%
%%%%%%%%%%%%%%%%%%%%%%%%%%%%%%%%%%%%%%%%%%%%
\mysection{Introduction}{s:intro}

In \cite{Ga}, David Gabai proposed a remarkable viewpoint that several important problems in 4-dimensional topology can be interpreted by 4-dimensional light bulb problem. There is a version of 4-dimensional light bulb problem, which asks whether spanning $k$-disk of the unknotted $S^{k-1}$ in $S^4$ is unique up to isotopy fixing the boundary. Gabai gave a positive resolution of this problem for $k=2$ (``the 4-dimensional light bulb theorem") in \cite{Ga}, and pointed out that a positive resolution for $k=3$ implies the smooth Schoenflies conjecture and that the converse is true up to taking a lift in some finite branched covering of $S^4$ over the unknotted $S^2$. In \cite{BG}, Ryan Budney and Gabai constructed infinitely many counterexamples to the 4-dimensional light bulb problem for $k=3$ by studying the group $\pi_0\Diff(D^3\times S^1,\partial)$ in detail, utilizing an embedding calculus method by Gregory Arone and Markus Szymik in \cite{ArSz}, and gave a framework for approaching the smooth Schoenflies conjecture.

In this paper, we attempt to extend some of the results of \cite{BG} to 4-manifolds that may not be diffeomorphic to $D^3\times S^1$. More precisely, we use the technique of graph surgery given in \cite{Wa1,Wa2} to construct 4-manifold bundles over $S^1$, and show their nontriviality by using invariants. Here, the invariant we use is a twisted analogue of Kontsevich's characteristic class of smooth bundles (\cite{Kon}), and is defined in a quite different way from that of \cite{BG}. The invariant in this paper is defined by extending the 3-manifold invariants of Julien March\'{e} and Christine Lescop that count ``equivariant triple intersection'' in configuration space (\cite{Mar, Les1, Les2}) to 4-manifolds with more general local coefficient system including those of Lie algebra. Main ideas in the definition and computation of the invariant of this paper are included in \cite{Les1, Les2}, which is analogous to those by Greg Kuperberg and Dylan Thurston \cite{KT} for $\Z$ homology 3-spheres. Similar invariants of 3-manifolds with Lie algebra local coefficient system were given in \cite{AxSi, Kon, Fuk, BC, CS}. 

Our approach can be considered different from that of \cite{BG} in the sense that ours can give nontrivial elements that cannot be detected by a direct application of their method. It should be mentioned that our construction is obtained by the ``barbell implantation'' given in \cite[\S{6}]{BG}, as written there, which looks quite simpler than our construction. 

%%%%%%%%%%%%%%%%%%%%%%%%%%%%%%%%
\subsection{Main results}\label{ss:main-result}

For a parallelizable closed manifold $X$ with a basepoint $x_0$, let $X^\bullet$ be the complement of a small open ball around $x_0$. For a compact manifold $Z$ and its submanifold $Y$, let $\Diff(Z,Y)$ denote the group of self-diffeomorphisms of $Z$ which fix a neighborhood of $Y$ pointwise, equipped with the $C^\infty$-topology, let $\Diff(Z,\partial)=\Diff(Z,\partial Z)$  and $\Diff(Z)=\Diff(Z,\emptyset)$. Let $\Diff_0(Z,Y)$ denote the subgroup of $\Diff(Z,Y)$ consisting of diffeomorphisms homotopic to the identity fixing the boundary. Let $\calF_*(X)$ denote the space of all framings on $X$ that are standard near $x_0$. Let $\Map_*^\mathrm{deg}(X,X)$ denote the space of all degree 1 maps $(X,x_0)\to (X,x_0)$ that maps a neighborhood of $x_0$ onto $x_0$ and that are homotopic to the identity. The groups $\Diff(X^\bullet,\partial)$, $\Diff_0(X^\bullet,\partial)$ act on these spaces respectively. Let
\[ \begin{split}
&\widetilde{B\Diff}(X^\bullet,\partial)=E\Diff(X^\bullet,\partial)\times_{\Diff(X^\bullet,\partial)}\calF_*(X),\\ 
&\widetilde{B\Diff}_{\mathrm{deg}}(X^\bullet,\partial)=E\Diff_0(X^\bullet,\partial)\times_{\Diff_0(X^\bullet,\partial)}\bigl(\calF_*(X)\times \Map_*^\mathrm{deg}(X,X)\bigr).
\end{split}\]
The space $\widetilde{B\Diff}(X^\bullet,\partial)$ is the classifying space for framed $X$-bundles $\pi\colon E\to B$ that are standard near $x_0$, and $\widetilde{B\Diff}_{\mathrm{deg}}(X^\bullet,\partial)$ is the classifying space for such $X$-bundles equipped with a map $E\to X$ whose restriction to each fiber is a pointed homotopy equivalence. For $p\geq 1$, we will see that the groups $\pi_p\calF_*(X)$ and $\pi_p\Map_*^\mathrm{deg}(X,X)$ are both abelian groups of finite ranks (Proposition~\ref{prop:pi-F-M}).

\begin{Thm}[Theorems~\ref{thm:Z-inv-O=0}, \ref{thm:surgery}, Corollary~\ref{cor:image-Z}]\label{thm:main1-odd}
Let $\Sigma^3$ be the Poincar\'{e} homology 3-sphere $\Sigma(2,3,5)$ and $X=\Sigma^3\times S^1$. Suppose that the local coefficient system of the adjoint action of a representation $\rho\colon \pi'=\pi_1(\Sigma^3)\to SU(2)$ on $\mathfrak{g}=\mathfrak{sl}_2$ is acyclic, that is, $H_*(X;\mathfrak{g})=0$. Then a homomorphism
\[ 
Z_\Theta^\even\colon \pi_1\widetilde{B\Diff}_{\mathrm{deg}}(X^\bullet,\partial)\to \calA_\Theta^\even(\mathfrak{g}^{\otimes 2}[t^{\pm 1}];\rho(\pi')\times \Z)
\]
to some space of anti-symmetric tensors is defined, and the image of $Z_\Theta^\even$ has countable infinite rank. 
\end{Thm}

The nontrivial elements detected in Theorem~\ref{thm:main1-odd} are constructed by surgery along $\Theta$-graphs as in \cite{Wa1,Wa2}, which are analogous to Goussarov and Habiro's graph surgery in 3-manifolds (\cite{Gu,Hab}).
\begin{Cor}[Theorem~\ref{thm:infinite-rank-restate}]\label{cor:infinite-rank1}
Let $\Sigma^3=\Sigma(2,3,5)$. Let $\Emb_0(\Sigma^3,\Sigma^3\times S^1)$ denote the space of embeddings $\Sigma^3\to \Sigma^3\times S^1$ homotopic to the standard inclusion $\Sigma^3=\Sigma^\times\{1\}\subset \Sigma^3\times S^1$.
\begin{enumerate}
\item The abelianization of the group $\pi_0\Diff_0(\Sigma^3\times S^1)$ has countable infinite rank.
\item The set $\pi_0\Emb_0(\Sigma^3,\Sigma^3\times S^1)$ is infinite.
\end{enumerate}
\end{Cor}
An explicit subgroup of infinite rank detected in both Theorem~\ref{thm:main1-odd} and Corollary~\ref{cor:infinite-rank1}-1 is given in Proposition~\ref{prop:A-incl}.

Peter Teichner pointed out the following (\cite{Tei}). See also \cite{BW} for a detail.
\begin{Thm}[Teichner]\label{thm:teichner}
Let $C(Y)=\Diff(Y\times I,Y\times\{0\})$. The nontrivial elements of the group $\pi_0\Diff_0(\Sigma^3\times S^1)$ found in Corollary~\ref{cor:infinite-rank1} are included in the image of the natural map 
\[ \pi_0 C(\Sigma^3\times S^1)\to \pi_0\Diff_0(\Sigma^3\times S^1). \]
\end{Thm}

\begin{Cor} 
If $\Sigma^3=\Sigma(2,3,5)$, we have $\pi_0 C(\Sigma^3\times S^1)\neq 0$. 
\end{Cor}

The nontrivial elements of $\pi_0 C(\Sigma^3\times S^1)$ found here give examples of diffeomorphisms of $\Sigma^3\times S^1$ that are pseudo-isotopic to but not isotopic to the identity. Nontriviality results for $\pi_0 C(X)$ for other 4-manifolds $X$ are obtained more recently in \cite{Igu,Sin,Igu2}.

Our $\Theta$-graph surgery construction gives diffeomorphisms of $\Sigma^3\times S^1$ that are homotopic to the identity but smoothly nontrivial (Proposition~\ref{prop:homotopically-trivial}). We do not know whether these elements are also topologically nontrivial, though it seems likely to be so, considering the state in higher even dimensions (\cite[\S{7.7}]{BRW}). It would be worth mentioning that there are several works (\cite{Ru, BK, KM}) giving diffeomorphisms of many 4-manifolds $X$ whose isotopy classes are in the kernel of $\pi_0\Diff(X)\to \pi_0\mathrm{Top}(X)$. 

One can consider similar problems for $X=D^3\times S^1$.
The following result of \cite{BG} can be obtained as a corollary of Corollary~\ref{cor:infinite-rank1}.
\begin{Cor}[Budney--Gabai]\label{cor:infinite-rank2}
The abelian group
$\pi_0\Diff(D^3\times S^1,\partial)$ has a direct summand of countable infinite rank, which is included in the image of the natural map $\pi_0 C_\partial(D^3\times S^1)\to \pi_0\Diff(D^3\times S^1,\partial)$, where $C_\partial(Z)=\Diff(Z\times I,Z\times\{0\}\cup \partial Z\times I)$.
\end{Cor}
\begin{proof}
In Corollary~\ref{cor:infinite-rank1}, we still have a subgroup of countable infinite rank, if we replace $\pi_0\Diff_0(\Sigma^3\times S^1)$ with the image of a map
\[ \pi_0\Diff_\partial(D^3\times S^1)\to \pi_0\Diff_0(\Sigma^3\times S^1) \]
induced by the inclusion $D^3\times S^1\to \Sigma^3\times S^1$ to a small tubular neighborhood of $\{*\}\times S^1$. 
\end{proof}
What is proved in this paper is the infiniteness of the group $\pi_0\Diff(D^3\times S^1,\partial)$ and it is not proved in this paper that our group of infinite rank agrees with that of \cite{BG}. According to \cite{BG}, it follows from Corollary~\ref{cor:infinite-rank2} that the group $\pi_0\Emb_0(D^3,D^3\times S^1)$, which they proved to be abelian, has countable infinite rank. Hence there are many distinct spanning 3-disks of the unknot in $S^4$ that are not relatively isotopic to the standard one. A statement analogous to Theorem~\ref{thm:teichner} holds in this case as well, as also shown in \cite{BG} for their construction. We mention that the results for $\Sigma^3\times S^1$ as above in Corollary~\ref{cor:infinite-rank1} cannot be obtained by applying the method of \cite{BG}, at least in a direct way. It should be mentioned that there is also a remarkable result in \cite{BG} which says that the nontrivial subgroup found in \cite{BG} survives in $\pi_0\Diff_0(S^3\times S^1)$. 

For higher dimensions, it is known that $\pi_1\Sigma(2,3,5)\cong SL_2(\mathbb{F}_5)$ acts freely on $S^{4k-1}$ (\cite{LM}). 
\begin{Cor}[Corollary~\ref{cor:high-dim}]
Let $n\geq 7$ be an integer of the form $4k-1$ and let $\Sigma^n=S^n/\pi'$, where $\pi'=\pi_1\Sigma(2,3,5)$. Then the abelian group $\pi_{n-2}B\Diff_0(\Sigma^n\times S^1)$ has countable infinite rank.
\end{Cor}

%%%%%%%%%%%%%%%%%%%%%%%%%%%%%%%%
\subsection{Plan of the paper}

In \S\ref{s:graph}, we describe the spaces of (anti-)symmetric tensors, which are the target spaces of the invariants considered in this paper.

In \S\ref{s:manifold}, we make assumptions on manifolds with local coefficient systems and define their configuration spaces, and propagator.

In \S\ref{s:bundle}, we make assumptions on manifold bundles with local coefficient systems and define their fiberwise configuration spaces, and propagator for families. Also, we consider an intersection form between chains of configuration space with local coefficients.

In \S\ref{s:invariant}, we define the invariant and prove its bordism invariance.

In \S\ref{s:sformula}, we construct fiber bundles by $\Theta$-graph surgery, following \cite{Wa1,Wa2}, and give a formula of the invariant for the bundles obtained by $\Theta$-graph surgery. The proof of the surgery formula boils down to Proposition~\ref{prop:normalization}. The outline of its proof is almost the same as given in \cite{Les1}. 

In \S\ref{s:support-I}, we describe a restriction on the value of the invariant for $\Sigma^n\times S^1$-bundles that has support in $\Sigma^n\times I$. Then we prove Corollary~\ref{cor:infinite-rank1}.

In \S\ref{s:pseudo-isotopy}, We show that the $X$-bundle obtained by $\Theta$-graph surgery extends to a bundle of pseudo-isotopy.

%%%%%%%%%%%%%%%%%%%%%%%%%%%%5
\subsection{Notations and conventions}

For a sequence of submanifolds $A_1,A_2,\ldots,A_r\subset W$ of a smooth Riemannian manifold $W$, we say that the intersection $A_1\cap A_2\cap \cdots\cap A_r$ is {\it transversal} if for each point $x$ in the intersection, the subspace $N_xA_1+N_xA_2+\cdots+N_xA_r\subset T_xW$ is the direct sum $N_xA_1\oplus N_xA_2\oplus\cdots\oplus N_xA_r$, where $N_xA_i$ is the orthogonal complement of $T_xA_i$ in $T_xW$ with respect to the Riemannian metric. 

For manifolds with corners and their (strata) transversality, we follow \cite[Appendix]{BT}.

For orientation convention for (sub)manifolds, we follow \cite[\S{1.5} and Appendix~D]{Wa2}. As in \cite{Wa2}, we represent an orientation of a manifold $M$ by a nowhere-zero section of $\bigwedge^{\dim{M}} TM$ and use the symbol $o(M)$ for orientation of $M$. When $\dim{M}=0$, we give an orientation of $M$ by a choice of sign $\pm 1$ at each point, as usual.
We orient the boundary of a manifold by the outward-normal-first convention. 

As chains in a manifold $X$, we consider $\C$-linear combinations of finitely many smooth maps from compact oriented manifolds with corners to $X$. We say that two chains $\sum n_i\sigma_i$ and $\sum m_j\tau_j$ ($n_i,m_j\in\C$, $\sigma_i,\tau_j$: smooth maps from compact manifolds with corners) are strata transversal if for every pair $i,j$, the terms $\sigma_i$ and $\tau_j$ are strata transversal. Strata transversality among two or more chains can be defined similarly. The intersection number $\langle \sigma,\tau\rangle_X$ of strata transversal two chains $\sigma=\sum n_i\sigma_i$ and $\tau=\sum m_j\tau_j$ with $\dim{\sigma_i}+\dim{\tau_j}=\dim{X}$ is defined by $\sum_{i,j}n_im_j(\sigma_i\cdot \tau_j)$, where we impose the orientation on the intersection by the wedge product of coorientations of the two pieces. We also consider intersection $\langle \sigma_1,\ldots,\sigma_n\rangle_X$ of strata transversal chains $\sigma_1,\ldots,\sigma_n$ for $n\geq 2$, which is defined similarly.

The diagonal $\{(x,x)\in X\times X\mid x\in X\}$ is denoted by $\Delta_X$. 

For a fiber bundle $\pi\colon E\to B$, we denote by $T^vE$ the (vertical) tangent bundle along the fiber $\mathrm{Ker}\,d\pi\subset TE$. Let $ST^vE$ denote the subbundle of $T^vE$ of unit spheres. Let $\partial^\fib E$ denote the fiberwise boundaries: $\bigcup_{b\in B}\partial(\pi^{-1}\{b\})$. 

%%%%%%%%%%%%%%%%%%%%%%%%%%%%%%%%
\subsubsection*{Acknowledgments}

I am grateful to G.~Arone, B.~Botvinnik, R.~Budney, Y.~Eliashberg, D.~Gabai, H.~Konno, D.~Kosanovi\'{c}, A.~Kupers, F.~Laudenbach, C.~Lescop, A.~Lobb, M.~Powell, O.~Randal-Williams, K.~Sakai, T.~Sakasai, M.~Sato, T.~Satoh, T.~Shimizu, M.~Taniguchi, P.~Teichner, Y.~Yamaguchi, D.~Yuasa for helpful comments.

Part of this work was done through my stay at OIST, Laboratory of Mathematics Jean Leray, Institut Fourier, Max Planck Institute for Mathematics, University of Oregon, Stanford University in 2019. I thank the institutes for their generous support. Also, I thank the organizers of the workshops ``Four Manifolds: Confluence of High and Low Dimensions (Fields Institute, 2019)", ``HCM Workshop: Automorphisms of Manifolds (Hausdorff Center, 2019)", ``Workshop on 4-manifolds (MPIM, 2019)", which enabled me to communicate with many people about this work. This work was partially supported by JSPS Grant-in-Aid for Scientific Research 17K05252, 15K04880 and 26400089, and by RIMS, Kyoto University.

%%%%%%%%%%%%%%%%%%%%%%%%%%%%%%%
%%%%%%%%%%%%%%%%%%%%%%%%%%%%%%%
\mysection{Spaces of (anti-)symmetric tensors}{s:graph}

In this section, we define the target spaces of the invariants and give some ways to extract information from these spaces. 

%%%%%%%%%%%%%%%%%%%%%%%%%%%%%%%
\subsection{Anti-symmetric tensors for even dimensional manifolds}

Let $G=SU(m)$, $\mathfrak{g}=\mathrm{Lie}(G)\otimes \C=\mathfrak{sl}_m$. Generally, a complex semisimple Lie algebra $\mathfrak{g}$ has an $\mathrm{Ad}(G)$-invariant symmetric nondegenerate bilinear form $B\colon \mathfrak{g}^{\otimes 2}\to\C$. If $\mathfrak{g}=\mathfrak{sl}_m$, then $B$ can be given by $B(X,Y)=\Tr(XY)$. This is $\mathrm{Ad}(G)$-invariant in the sense that $B(\mathrm{Ad}(g)X,\mathrm{Ad}(g)Y)=\Tr(gXg^{-1} gYg^{-1})=\Tr(XY)=B(X,Y)$. We equip $\mathfrak{g}^{\otimes 2}$ with an algebra structure over $\C$ by the following product.
\[ (X_1\otimes Y_1)\cdot(X_2\otimes Y_2)=B(Y_1,X_2)\,X_1\otimes Y_2\]
The multiplicative unit for this product is the Casimir element 
$c_\mathfrak{g}=\sum_i e_i\otimes e_i^*$, where $\bvec{e}=\{e_i\}$ a basis of $\mathfrak{g}$, $\bvec{e}^*=\{e_i^*\}$ is the dual basis for $\bvec{e}$ with respect to $B(\cdot,\cdot)$: $B(e_i,e_j^*)=\delta_{ij}$. The two algebras $\g^{\otimes 2}$ and $\End(\g)$ are identified by $u\otimes v^*\mapsto (x\mapsto B(x,v^*)u)$. 
Then the product in $\mathfrak{g}^{\otimes 2}$ corresponds to the composition of endomorphisms, and $B$ corresponds to the trace. The element $\mathrm{Ad}(g)\in\mathrm{End}(\mathfrak{g})$ corresponds to $(1\otimes\mathrm{Ad}(g^*))(c_\mathfrak{g})=(\mathrm{Ad}(g)\otimes 1)(c_\mathfrak{g})\in\mathfrak{g}^{\otimes 2}$, which we denote by the same symbol $\mathrm{Ad}(g)$. An $\mathrm{Ad}(G)$-invariant skew-symmetric 3-form
$\Tr\colon  \mathfrak{g}^{\otimes 3}\to \C$
is defined by the following formula.
\[ \Tr(X\otimes Y\otimes Z)=B([X,Y],Z) \]
We refer the reader to \cite[Ch.XVII.1]{Kas} for the details.
The external tensor product $\mathfrak{g}^{\boxtimes 2}=\mathfrak{g}\boxtimes \mathfrak{g}$ is the left $G\times G$-module $\mathfrak{g}^{\otimes 2}$ defined by the action
\[ (g,h)(X\otimes Y)=\mathrm{Ad}(g)X\otimes \mathrm{Ad}(h)Y=gXg^{-1}\otimes hYh^{-1}. \]

\begin{Def}\label{def:A}
\begin{enumerate}
\item Let $\calA_\Theta^\even(\C[t^{\pm 1}];\Z)=\bigwedge^3\C[t^{\pm 1}]/\Z$, where the alternating tensor product is over $\C$, and the quotient by $\Z$ is given by the relation:
\[ \begin{split}
  a\wedge b\wedge c&=t^na\wedge t^nb\wedge t^nc\quad (a,b,c\in \C[t^{\pm 1}],\,n\in \Z).
\end{split} \]

\item Let $H$ be a subgroup of $G$ and let 
\[ \calA_\Theta^\even(\mathfrak{g}^{\otimes 2}[t^{\pm 1}];H\times \Z)=\tbigwedge^3(\mathfrak{g}^{\otimes 2}[t^{\pm 1}])/(H^{\times 2}\times\Z), \]
where the alternating tensor product is over $\C$, and the quotient by $H^{\times 2}\times\Z$ is given by the relations:
\[ \begin{split}
  a\wedge b\wedge c&=t^na\wedge t^nb\wedge t^nc,\\
  a\wedge b\wedge c&=(g,g')\, a\wedge (g,g')\, b\wedge (g,g')\, c \quad (a,b,c\in\mathfrak{g}^{\otimes 2}[t^{\pm 1}],\,g,g'\in H).
\end{split} \]

\item Let $\calS_\Theta^\even=\bigwedge^3 \C[t^{\pm 1}]/{\sim}$, $t^p\wedge t^q\wedge t^r\sim t^{-p}\wedge t^{-q}\wedge t^{-r}$. This space is embedded into a $\C$ subspace of $\C[t_1^{\pm 1},t_2^{\pm 1},t_3^{\pm 1}]$ by the map
\[ 
[t^p\wedge t^q\wedge t^r]\mapsto 
\sum_{\sigma\in\mathfrak{S}_3}\mathrm{sgn}(\sigma)\bigl(t_{\sigma(1)}^p t_{\sigma(2)}^q t_{\sigma(3)}^r+
t_{\sigma(1)}^{-p} t_{\sigma(2)}^{-q} t_{\sigma(3)}^{-r}\bigr).
\]
\end{enumerate}
\end{Def}

The spaces $\calA_\Theta^\even$ can be considered as the spaces of decorated $\Theta$-graphs considered modulo the invariance relation at trivalent vertices. We will denote the element $\alpha \wedge \beta\wedge \gamma$ of the space $\bigwedge^3 \C[t^{\pm 1}]$ or $\bigwedge^3 \mathfrak{g}^{\otimes 2}[t^{\pm 1}]$ by $\Theta(\alpha,\beta,\gamma)$ to distinguish from elements of $\calS_\Theta^\even$.  Let 
$W_{\C[t^{\pm 1}]}^\even(\Theta(t^a,t^b,t^c))\in \calS_\Theta^\even$ be the element defined by the following formula.
\[ \begin{split}
W_{\C[t^{\pm 1}]}^\even(\Theta(t^a,t^b,t^c))=[t^{b-a}\wedge t^{a-c}\wedge t^{c-b}] 
\end{split}\]
For $P t^a,Q t^b,R t^c$ ($P,Q,R\in \mathfrak{g}^{\otimes 2}$), let $\Theta(Pt^a,Qt^b,Rt^c)$ denote the element $P t^a\wedge Q t^b\wedge R t^c$ in $\calA_\Theta^\even(\mathfrak{g}^{\otimes 2}[t^{\pm 1}];G\times\Z)$, and let 
$W^\even_{\mathfrak{g}[t^{\pm 1}]}(\Theta(P t^a,Q t^b,R t^c))\in \calS_\Theta^\even$
be defined by the following formula:
\[ \begin{split}
&W_{\mathfrak{g}[t^{\pm 1}]}^\even(\Theta(P t^a,Q t^b,R t^c))\\
&\hspace{5mm}=(\Tr\otimes \Tr)\,\sigma_\Theta\,(P\otimes Q\otimes R)[t^{b-a}\wedge t^{a-c}\wedge t^{c-b}], 
\end{split}\]
where $\sigma_\Theta\colon (\mathfrak{g}^{\otimes 2})^{\otimes 3}\to (\mathfrak{g}^{\otimes 3})^{\otimes 2}$ is the permutation of factors that is determined by the combinatorial structure of the $\Theta$-graph:
\[ 
(x_P\otimes y_P)\otimes (x_Q\otimes y_Q)\otimes (x_R\otimes y_R)
\stackrel{\sigma_\Theta}{\mapsto}(x_P\otimes x_Q\otimes x_R)\otimes (y_R\otimes y_Q\otimes y_P).
\]
The coefficient $(\Tr\otimes \Tr)\,\sigma_\Theta\,(P\otimes Q\otimes R)$ is precisely the $\mathfrak{sl}_m$-weight system of the $\Theta$-graph with decorations $P,Q,R$ on edges (\cite[\S{2.4}]{BN}, see also \cite[\S{6.2}]{CDM}). 

\begin{Prop}\label{prop:w0-well-defined}
The $W_{\C[t^{\pm 1}]}^\even$ and $W_{\mathfrak{g}[t^{\pm 1}]}^\even$ for $\Theta(t^a,t^b,t^c)$ and $\Theta(Pt^a,Qt^b,Rt^c)$ as above induce well-defined $\C$-linear maps
\[\begin{split}
 &W_{\C[t^{\pm 1}]}^\even\colon \calA_\Theta^\even(\C[t^{\pm 1}];\Z)\to \calS_\Theta^\even,\mbox{ and}\\
 &W_{\mathfrak{g}[t^{\pm 1}]}^\even\colon \calA_\Theta^\even(\mathfrak{g}^{\otimes 2}[t^{\pm 1}];G\times \Z)\to \calS_\Theta^\even,\mbox{ respectively.}
\end{split} \]
\end{Prop}
\begin{proof}
First, we see that $W_{\C[t^{\pm 1}]}^\even$ is well-defined. The part $[t^{b-a}\wedge t^{a-c}\wedge t^{c-b}]$ is invariant under the action of $\Z$ on $\bigwedge^3\C[t^{\pm 1}]$, and the transpositions $a\leftrightarrow b$, $b\leftrightarrow c$, $c\leftrightarrow a$ turn this into 
\[\begin{split}
 &[t^{a-b}\wedge t^{b-c}\wedge t^{c-a}]=[t^{b-a}\wedge t^{c-b}\wedge t^{a-c}]=-[t^{b-a}\wedge t^{a-c}\wedge t^{c-b}],\\
 &[t^{c-a}\wedge t^{a-b}\wedge t^{b-c}]=[t^{a-c}\wedge t^{b-a}\wedge t^{c-b}]=-[t^{b-a}\wedge t^{a-c}\wedge t^{c-b}],\\
 &[t^{b-c}\wedge t^{c-a}\wedge t^{a-b}]=[t^{c-b}\wedge t^{a-c}\wedge t^{b-a}]=-[t^{b-a}\wedge t^{a-c}\wedge t^{c-b}]
\end{split} \]
in $\calS_\Theta^\even$. Hence $W_{\C[t^{\pm 1}]}^\even$ is well-defined on $\calA_\Theta^\even(\C[t^{\pm 1}];\Z)$.

Also, it follows from the $\mathfrak{S}_3$-antisymmetry and the $\mathrm{Ad}(G)$-invariance of $\Tr$ that the coefficient part $(\Tr\otimes \Tr)\,\sigma_\Theta\,(P\otimes Q\otimes R)$ is $\mathfrak{S}_3$-symmetric and $\mathrm{Ad}(G)^{\times 2}$-invariant. This together with the $\Z$-invariance and $\mathfrak{S}_3$-antisymmetry of the part $[t^{b-a}\wedge t^{a-c}\wedge t^{c-b}]$ proves that $W_{\mathfrak{g}[t^{\pm 1}]}^\even$ is well-defined.
\end{proof}

\begin{Exa}
We have $W_{\C[t^{\pm 1}]}^\even\bigl(\Theta(1,t,t^p)\bigr)=[t\wedge t^{-p}\wedge t^{p-1}]$. This element corresponds to the following element in $\C[t_1^{\pm 1},t_2^{\pm 1},t_3^{\pm 1}]$.
\[ \begin{split}
&t_1t_2^{-p}t_3^{p-1} 
-t_1t_3^{-p}t_2^{p-1}
-t_2t_1^{-p}t_3^{p-1}
+t_2t_3^{-p}t_1^{p-1}
+t_3t_1^{-p}t_2^{p-1}
-t_3t_2^{-p}t_1^{p-1}\\
+&t_1^{-1}t_2^pt_3^{-p+1} 
-t_1^{-1}t_3^pt_2^{-p+1} 
-t_2^{-1}t_1^pt_3^{-p+1} 
+t_2^{-1}t_3^pt_1^{-p+1} 
+t_3^{-1}t_1^pt_2^{-p+1} 
-t_3^{-1}t_2^pt_1^{-p+1} 
\end{split} \]
We denote this polynomial by $f_p(t_1,t_2,t_3)$.
\begin{Prop}\label{prop:1ttp}
$[\Theta(1,t,t^p)]$ ($p\geq 3$) spans a countable infinite dimensional $\C$-subspace of $\calA_\Theta^\even(\C[t^{\pm 1}];\Z)$. 
\end{Prop}
\begin{proof}
We have $f_p(1,x,x^3)=x^{3p-1}-x^{3p-2}-x^{2p+1}+x^{2p-3}+x^{p+2}-x^{p-3}
-x^{-p+3}+x^{-p-2}+x^{-2p+3}-x^{-2p-1}+x^{-3p+1}
$, and for $p\geq 3$, its maximal degree term is $x^{3p-1}$, minimal degree term is $x^{-3p+1}$. Thus 
$\{f_p(1,x,x^3)\mid p\geq 3\}$ is linearly independent over $\C$ in $\C[x^{\pm 1}]$.
\end{proof}
\end{Exa}

\begin{Prop}\label{prop:W(Ad)}
For $x,y\in SU(2)$, $a,b\in \Z$, we have
\[ \begin{split}
&W_{\mathfrak{g}[t^{\pm 1}]}^\even\bigl(\Theta(1,\mathrm{Ad}(x)t^a,\mathrm{Ad}(y)t^b)\bigr)\\
&=2\bigl(\Tr(\mathrm{Ad}(x))\Tr(\mathrm{Ad}(y))-\Tr(\mathrm{Ad}(xy))\bigr)[t^{a}\wedge t^{-b}\wedge t^{b-a}]
\end{split} \]
\end{Prop}
\begin{proof}
For $\mathfrak{g}=\mathfrak{su}(2)\otimes\C=\mathfrak{sl}_2$, the following identity holds in $\mathrm{End}_\C(\mathfrak{g}^{\otimes 2})$ (\cite{CVa}).
\begin{equation}\label{eq:sl2-relation}
 \includegraphics{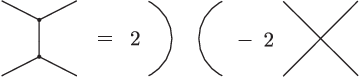} 
\end{equation}
Here, the left hand side is the composition of the Lie bracket $b\colon \mathfrak{g}^{\otimes 2}\to \mathfrak{g}$ and its dual $b^*\colon \mathfrak{g}\to \mathfrak{g}^{\otimes 2}$. The two terms in the right hand side represent the identity morphism and the transposition $x\otimes y\mapsto y\otimes x$, respectively. Then, we have
\[ \sigma_\Theta(1\otimes\mathrm{Ad}(x)\otimes\mathrm{Ad}(y))
=\Tr(\varphi\circ (\mathrm{Ad}(x^*)\otimes\mathrm{Ad}(y^*))), \]
where $\varphi\in\End_\C(\mathfrak{g}^{\otimes 2})$ is the map on the left hand side of (\ref{eq:sl2-relation}).

If we apply the relation (\ref{eq:sl2-relation}) to the $\Theta$-graph with edges decorated by $1$, $\mathrm{Ad}(x)$, $\mathrm{Ad}(y)$, then we obtain a disjoint union of circles with decorations. The first term in the right hand side of (\ref{eq:sl2-relation}) gives disjoint union of two circles decorated by $\mathrm{Ad}(x)$ and $\mathrm{Ad}(y)$, respectively. The second term in the right hand side of (\ref{eq:sl2-relation}) gives one circle decorated by $\mathrm{Ad}(x)\mathrm{Ad}(y)=\mathrm{Ad}(xy)$. 
\end{proof}

%%%%%%%%%%%%%%%%%%%%%%%%%%%%%%%
%%%%%%%%%%%%%%%%%%%%%%%%%%%%%%%
\mysection{Manifolds and configuration spaces with local coefficient system}{s:manifold}

We make an assumption on the manifold $X$ with a local coefficient system $A$, and we compute the homology of the configuration space $\bConf_2(X)$ of two points with the local coefficient system induced from $A\boxtimes A$. Based on the computation of the homology, we define propagator, which is an analogue of ``equivariant propagator" in \cite{Les1}. We will define in the next section propagator in families of configuration spaces, which plays an important role in the definition of the main perturbative invariant.

%%%%%%%%%%%%%%%%%%%%%%%%%%%%%%%
\subsection{Acyclic complex}\label{ss:acyclic-complex}

Let $X$ be a manifold with a basepoint $x_0$ and a universal cover $\widetilde{X}$. The space $\widetilde{X}$ can be considered as the set of pairs $(x,[\gamma])$, where $x$ is a point of $X$ and $[\gamma]$ is the homotopy class relative to the endpoints of a path $\gamma\colon [0,1]\to X$ from $x$ to the basepoint $x_0\in X$. We take a cellular chain complex $C_*(X)$ for a CW decomposition of $X$ and take $C_*(\widetilde{X})$ for the CW decomposition of $\widetilde{X}$ compatible with that of $X$. Let $S_*(X)$, $S_*(\widetilde{X})$ be the complexes of piecewise smooth singular chains in $X$ and $\widetilde{X}$, respectively. The complexes $C_*(\widetilde{X})$ and $S_*(\widetilde{X})$ are naturally $\C[\pi]$-modules through the covering transformations. We assume that the coefficients are in $\C$, unless otherwise noted. 

Let $\pi$ denote $\pi_1(X)$, $A$ be a left $\C[\pi]$-module, and let $\rho_A\colon \pi\to \mathrm{End}_\C(A)$ be the corresponding $\C$-linear representation. Let $R$ be $\C$ or $\C[t^{\pm 1}]$. We assume that $A$ has finite rank over a ring $R$ and has a nondegenerate $\C[\pi]$-invariant symmetric $R$-bilinear form $B(\cdot,\cdot)\colon A\otimes_R A\to R$. Let $c_A\in A\otimes_R A$ be the $\C[\pi]$-invariant 2-tensor defined by the following formula.
\[ c_A=\sum_i e_i\otimes e_i^* \]
Here, $\{e_i\}$ is an $R$-basis of $A$, $\{e_i^*\}$ is the dual basis for $\{e_i\}$ with respect to $B$.

Let $C_*(X;A)=C_*(\widetilde{X})\otimes_{\C[\pi]} A$, $S_*(X;A)=S_*(\widetilde{X})\otimes_{\C[\pi]} A$. The boundary operators of these complexes are given by $\partial_A=\partial\otimes \mathbbm{1}$ (see e.g., \cite[Ch.VI]{Wh}, \cite[II-Ch.6]{Hatt} etc. for homology of local coefficients). Especially, when $X=\Sigma^n\times S^1$, $C_*(X;A)=C_*(\widetilde{X})\otimes_{\C[\pi'\times \Z]}A$ has a structure of free $(\C[t^{\pm 1}],\C[t^{\pm 1}])$-bimodule if we let $\pi'=\pi_1(\Sigma^n)$ and take $A=A_1\otimes_\C\C[t^{\pm 1}]$ for some $\C[\pi']$-module $A_1$. In this case, we have a symmetric $\C[t^{\pm 1}]$-bilinear form $B(\cdot,\cdot)\colon A\otimes_R A\to R$, where $R=\C[t^{\pm 1}]$, and the invariant 2-tensor $c_{A_1}\in A_1^{\otimes 2}[t^{\pm 1}]=A\otimes_R A$, where we identify $R\otimes_R R$ with $R$ by $t^a\otimes t^b=t^{a-b}$. We make the following assumption.

\begin{Assum}\label{assum:acyclic}
$C_*(X;A)$ (and $S_*(X;A)$) is acyclic, i.e., $H_*(X;A)=0$.
\end{Assum}

\begin{Lem}\label{lem:tensor-acyclic}
Let $\Lambda=\C[t^{\pm 1}]$. Under Assumption~\ref{assum:acyclic}, the following hold.
\begin{enumerate}
\item $C_*(X\times X;A\boxtimes_\C A)=C_*(X;A)\otimes_\C C_*(X;A)$ is acyclic.
\item If moreover $X=\Sigma^n\times S^1$ and $A=A_1\otimes_\C\Lambda$, then $C_*(X\times X;A\boxtimes_\Lambda A)=C_*(X;A)\otimes_{\Lambda} C_*(X;A)$ is acyclic.
\end{enumerate}
\end{Lem}
\begin{proof}
The assertion 1 follows from the K\"{u}nneth formula for $\C$-modules. For 2, since $\Lambda$ is a PID, the exact sequence for $\Lambda$-modules (K\"{u}nneth formula (e.g., \cite[Theorem~VI.3.2]{CE}))
\[ 0\to \bigoplus_{p+q=n}H_p(C)\otimes_\Lambda H_q(C)\to H_n(C\otimes_\Lambda C)\to \bigoplus_{p+q=n-1}\mathrm{Tor}_1^\Lambda(H_p(C),H_q(C))\to 0\]
($C=C_*(X;A)$) holds. Then 2 is an immediate corollary of this.
\end{proof}

\begin{Rem}
\begin{enumerate}
\item To apply the K\"{u}nneth formula, certain restriction on the coefficient ring or on the chain complex $C$ is necessary, and it is not always possible to replace the tensor product in Lemma~\ref{lem:tensor-acyclic} (2) with that of $\C[\pi]$-modules, instead of $\Lambda$-modules. For example, for $\pi=\Z\times \Z$, the K\"{u}nneth formula for $\C[\pi]$-modules fails.
\item When $X$ is a closed manifold, $\chi(X)=0$ is necessary for the complex $C_*(X;\mathfrak{g})$ of the adjoint representation of a flat $G$-connection to be acyclic. Indeed, the Euler characteristic of this complex depends only on the dimensions of the modules $C_i(X;\mathfrak{g})$, and not on the twisted differential. Thus the identity
\[ \sum_i (-1)^i \dim C_*(X;\mathfrak{g}) = \chi(X)\dim{\mathfrak{g}} \]
holds. If $C_*(X;\mathfrak{g})$ is acyclic, $\chi(X)$ must be 0. For $\dim{X}=4$, $\chi(X)=0$ is satisfied if $X$ is a homology $S^3\times S^1$, but not if $X$ is a homology $S^4$. 
\item Although we identify $\mathrm{End}_R(A)$ with $A\otimes_R A$ through $B(\cdot,\cdot)$ to simplify notation, it would be more natural to think of this as $A\otimes_R A^*$. For example, the diagonal action of $\Z$ on $R\otimes_R R$ is converted into the conjugation $x\otimes y\mapsto tx\otimes yt^{-1}$ on $R\otimes_R R^*=\End_R(R)$, which is trivial. 
\end{enumerate}
\end{Rem}

\begin{Exa}\label{ex:235}
Let $\Sigma^3=\Sigma(2,3,5)$ and $\mathfrak{g}=\mathfrak{su}(2)\otimes\C$. The fundamental group $\pi'$ of $\Sigma^3$ has the following presentation.
\[ \pi'=\langle x_1,x_2,x_3,h\mid h\mbox{ central},\, x_1^2=h,\, x_2^3=h^{-1},\,x_3^5=h^{-1},\,x_1x_2x_3=1\rangle \]
There are exactly two conjugacy classes of irreducible representations $\rho\colon \pi'\to SU(2)$. One is such that
\begin{equation}\label{eq:235}
 \begin{split}
&\rho(h)=-I, \quad
\rho(x_1)=\mathrm{diag}(i,-i),\\ 
&\rho(x_2)=\left(\begin{array}{cc}
	\frac{1}{2}-\alpha i & \beta\\
	-\beta & \frac{1}{2}+\alpha i
\end{array}\right)\sim \mathrm{diag}(e^{\pi i/3},e^{-\pi i/3}),\quad\\
&\rho(x_3)=\left(\begin{array}{cc}
	\alpha-\frac{1}{2}i & -\beta i\\
	-\beta i & \alpha+\frac{1}{2}i
\end{array}\right)\sim \mathrm{diag}(e^{\pi i/5},e^{-\pi i/5}),
\end{split} 
\end{equation}
where $\alpha=\cos{\frac{\pi}{5}}=\frac{1+\sqrt{5}}{4}$, $\beta=\frac{-1+\sqrt{5}}{4}$, and $\sim$ is the conjugacy relation (\cite{Sa1, Bod} etc.). For the local coefficient system $\mathfrak{g}_\rho=\mathfrak{g}$ on $\Sigma^3$ for the adjoint action of $\rho$, we have $H^0(\Sigma^3;\mathfrak{g}_\rho)=0$ by the irreducibility of the representation. Moreover, it is known that $H^1(\Sigma^3;\mathfrak{g}_\rho)=0$ holds for any irreducible $SU(2)$-representation of $\Sigma^3=\Sigma(p,q,r)$ (\cite{FS, Bod} etc.). This together with Poincar\'{e} duality implies the acyclicity.
\begin{equation}\label{eq:acyclic}
 H_*(\Sigma^3;\mathfrak{g}_\rho)=0 
\end{equation}
By taking the tensor product over $\C$ with the homology of the universal cover of $S^1$ with \emph{twisted} coefficients\footnote{The acyclicity (\ref{eq:acyclic2}) would also be obtained even when we took untwisted coefficients here. However, in that case, the values of the invariant for our construction would be all trivial. Also, it is possible to obtain the vanishing (\ref{eq:acyclic2}) by replacing $\C[t^{\pm 1}]$ with $\C(t)$ (the field of rational functions) here, instead of requiring the acyclicity (\ref{eq:acyclic}). In that case, the invariant $Z_\Theta^\even$ would become surprisingly weak.}
\[ H_*(S^1;\C[t^{\pm 1}])=H_*(\R^1;\Z)\otimes_{\Z[t^{\pm 1}]}\C[t^{\pm 1}]\cong\C, \]
we have
\begin{equation}\label{eq:acyclic2}
 H_*(\Sigma^3\times S^1;\mathfrak{g}_\rho[t^{\pm 1}])=0. 
\end{equation}
Also, by mapping the generator of $\pi_1(S^1)=\Z$ into the center $\Z_2=\{\pm 1\}$ of $SU(2)$, an irreducible extension $\rho_1\colon \pi_1(\Sigma^3\times S^1)\to SU(2)$ is obtained. Since the adjoint action of the center $\{\pm 1\}$ of $SU(2)$ on $\mathfrak{g}_{\rho_1}$ is trivial, $\mathfrak{g}_{\rho_1}$ is isomorphic as a $\pi'\times \Z$-module to $\mathfrak{g}_\rho\boxtimes_\C \C$, and we have 
\begin{equation}\label{eq:acyclic-rho1}
 H_*(\Sigma^3\times S^1;\mathfrak{g}_{\rho_1})=H_*(\Sigma^3;\mathfrak{g}_\rho)\otimes_\C H_*(S^1;\C)=0. 
\end{equation}
The same results hold for another irreducible representation $\rho\colon \pi'\to SU(2)$ defined by the same formula as (\ref{eq:235}) with $\alpha=\cos{\frac{3\pi}{5}}=\frac{1-\sqrt{5}}{4}$, $\beta=\frac{1+\sqrt{5}}{4}$. 

In relation to the local coefficient system $\mathfrak{g}_\rho[t^{\pm 1}]$ above, let us compute the value of 
\[ 
W_{\mathfrak{g}[t^{\pm 1}]}^\even
\bigl(\Theta(1,\mathrm{Ad}(\alpha_3)\,t,\mathrm{Ad}(\alpha_3^2)\,t^p)\bigr) \quad (\alpha_3=\rho(x_3))
 \]
 for example. The map $W_{\mathfrak{g}[t^{\pm 1}]}^\even$ was defined in Proposition~\ref{prop:w0-well-defined}, $\mathrm{Ad}(\cdot)$ was defined before Definition~\ref{def:A}. By applying Proposition~\ref{prop:W(Ad)}, the value of this weight is 
\begin{equation}\label{eq:weight}
 2\bigl(w(\alpha_3)\,w(\alpha_3^2)-w(\alpha_3^3)\bigr)[t\wedge t^{-p}\wedge t^{p-1}], 
\end{equation}
where $w(x)$ is the weight $W_{\mathfrak{g}[t^{\pm 1}]}^\even$ of the oriented circle decorated by $\mathrm{Ad}(x)$, or the trace of $\mathrm{Ad}(\alpha)$. If $V$ is the standard representation of $SU(2)$, the adjoint $SU(2)$-representation $\mathfrak{g}$ can be considered as the codimension 1 subrepresentation of $V\otimes V^*=\mathfrak{gl}_2$, where the $SU(2)$-action is given by $v\otimes w^*\mapsto gv\otimes gw^*$, $gw^*(\cdot)=w^*(g^{-1}(\cdot))$ ($g\in SU(2)$). This shows that 
\[ w(x)=\Tr(x)\Tr(x^*)-1=|\Tr(x)|^2-1. \]
By $w(\alpha_3)=\frac{1+\sqrt{5}}{2}$, $w(\alpha_3^2)=w(\alpha_3^3)=\frac{1-\sqrt{5}}{2}$, the value of the weight (\ref{eq:weight}) is
\[ (-3+\sqrt{5})\,[t\wedge t^{-p}\wedge t^{p-1}]. \]
Here, by Proposition~\ref{prop:1ttp}, we know that $[t\wedge t^{-p}\wedge t^{p-1}]\neq 0$ for $p\geq 3$.
Also, replacing $\alpha_3$ with $I$ in the above formula gives the weight of $\Theta(1,t,t^p)$, and its value is 
\[ 12\,[t\wedge t^{-p}\wedge t^{p-1}]. \]
Hence $[\Theta(1,t,t^p)]$ and $[\Theta(1,\mathrm{Ad}(\alpha_3)\,t,\mathrm{Ad}(\alpha_3^2)\,t^p)]$ are linearly independent over $\Q$.
\qed
\end{Exa}

\begin{Prop}\label{prop:A-incl}
Let $\calA_\Theta^\even(\mathfrak{g}^{\otimes 2};\rho(\pi'))=\tbigwedge^3(\mathfrak{g}^{\otimes 2})/\rho(\pi')^{\times 2}$. 
Then the set $\{\Theta(1,t,t^p)\mid p\geq 3\}$ is linearly independent in the cokernel of the natural embedding
\[ \begin{split}
i_{\mathfrak{g}}\colon \calA_\Theta^\even(\mathfrak{g}^{\otimes 2};\rho(\pi'))\to \calA_\Theta^\even(\mathfrak{g}^{\otimes 2}[t^{\pm 1}];\rho(\pi')\times \Z).
\end{split}\]
\end{Prop}
\begin{proof}
The result follows immediately from Proposition~\ref{prop:1ttp} and Example~\ref{ex:235}. Namely, we consider the following commutative diagram:
\[ \xymatrix{
  & \calA_\Theta^\even(\C[t^{\pm 1}];\Z) \ar[d] \ar[rd]^-{W_{\C[t^{\pm 1}]}^\even} & \\
 \calA_\Theta^\even(\mathfrak{g}^{\otimes 2};\rho(\pi')) \ar@<-2pt>[r]_-{i_{\mathfrak{g}}} &
  \calA_\Theta^\even(\mathfrak{g}^{\otimes 2}[t^{\pm 1}];\rho(\pi')\times \Z) \ar[r]_-{W_{\mathfrak{g}[t^{\pm 1}]}^\even} \ar@<-2pt>[l]_-{t=1}&
 \calS_\Theta^\even 
} \]
where $W_{\mathfrak{g}[t^{\pm 1}]}^\even\circ i_{\mathfrak{g}}=0$ since $[1\wedge 1\wedge 1]=0$. Since $W_{\C[t^{\pm 1}]}^\even$ embeds $\{\Theta(1,t,t^p)\mid p\geq 3\})$ to a linearly independent set by Proposition~\ref{prop:1ttp}, the result follows.
\end{proof}

%%%%%%%%%%%%%%%%%%%%%%%%%%%%%%%
\subsection{Propagator in a fiber for $n$ odd}

Let $X$ be a closed parallelizable $(n+1)$-manifold with a local coefficient system $A$ satisfying the acyclicity Assumption~\ref{assum:acyclic}. Let $\Delta_X$ be the diagonal of $X\times X$. The configuration space of two points of $X$ is 
\[ \Conf_2(X)=X\times X-\Delta_X.\]
The Fulton--MacPherson compactification of $\Conf_2(X)$ is
\[ \bConf_2(X)=B\ell(X\times X,\Delta_X). \]
The right hand side is the analogue of the blow-up for real smooth manifolds, which roughly replaces $\Delta_X$ with its normal sphere bundle in $X\times X$. The boundary $\partial\bConf_2(X)$ is identified with the normal sphere bundle of $\Delta_X$, which is canonically identified with the unit tangent bundle $ST(X)$. Since $X$ is parallelizable, there is a diffeomorphism $\partial \bConf_2(X)\cong S^n\times X$.

We also denote by $A^{\boxtimes 2}$ the pullback of the local coefficient system $A^{\boxtimes 2}=A\boxtimes_R A$, $R=\C$ or $\C[t^{\pm 1}]$, on $X\times X$ to $\bConf_2(X)$, and by $A^{\otimes 2}$ its restriction to $\partial\bConf_2(X)$, on which $\pi$ acts diagonally. Also, we write $\partial_A$ for $\partial_{A^{\boxtimes 2}}$ for simplicity. There is an $R$-submodule of $A^{\otimes 2}$ spanned by $c_A$, and the diagonal action of the holonomy on $\Delta_X$ is trivial on it. Since the natural map $H_*(X;R)=H_*(X;\C)\otimes R\to H_*(\Delta_X;A^{\otimes 2})$ given by $\sigma\mapsto \sigma\otimes c_A$ is a section of the projection, the $R$-module $H_*(\Delta_X;A^{\otimes 2})$ has a direct summand isomorphic to $H_*(X;R)$, where $R$ is the untwisted coefficient. We make the following assumption.

\begin{Assum}\label{assum:delta-trivial}
$H_*(\Delta_X;A^{\otimes 2})\cong H_*(X;R)$ and this is generated over $R$ by $\sigma\otimes c_A$ for $\C$-cycles $\sigma$ of $X$. 
\end{Assum}

We will see later in Proposition~\ref{prop:Sigma-delta-trivial} that Assumption~\ref{assum:delta-trivial} is satisfied by the local coefficient system $A=\mathfrak{g}_\rho[t^{\pm}]$ on $X=\Sigma(2,3,5)\times S^1$ of Example~\ref{ex:235}. The following is an analogue of \cite[Proposition~2.12]{Les1}.

\begin{Lem}\label{lem:H(Conf)}
For an odd integer $n\geq 3$, let $X$ be a parallelizable $\Z$ homology $S^n\times S^1$. Let $K$ be an oriented knot in $X$ that generates $H_1(X;\Z)$, and $\Sigma$ be an oriented $n$-submanifold of $X$ that generates $H_n(X;\Z)$ and such that $\langle \Sigma, K\rangle=1$. Under Assumptions~\ref{assum:acyclic} and \ref{assum:delta-trivial}, we have
\[ H_i(\bConf_2(X);A^{\boxtimes 2})=\left\{
\begin{array}{ll}
	R[ST(*)\otimes c_A] & (i=n)\\
	R[ST(K)\otimes c_A] & (i=n+1)\\
	R[ST(\Sigma)\otimes c_A] & (i=2n)\\
	R[ST(X)\otimes c_A] & (i=2n+1)\\
	0 & (\mbox{otherwise})
\end{array}\right.\]
where for an oriented submanifold $\sigma$ of $X$, we denote by $ST(\sigma)$ the restriction of the unit sphere bundle $ST(X)$ to $\sigma$, for which $ST(\sigma)\otimes c_A$ is an $A^{\boxtimes 2}$-cycle by Assumption~\ref{assum:delta-trivial}. 
\end{Lem}
\begin{proof}
We consider the exact sequence
\[\begin{split}
 &\,H_{i+1}(X^{\times 2};A^{\boxtimes 2})\to H_{i+1}(X^{\times 2},\Conf_2(X);A^{\boxtimes 2})\to H_i(\Conf_2(X);A^{\boxtimes 2})\\
\to &\,H_i(X^{\times 2};A^{\boxtimes 2}),
\end{split}\]
where we have $H_*(X^{\times 2};A^{\boxtimes 2})=0$ by (\ref{eq:acyclic}) and Lemma~\ref{lem:tensor-acyclic}. 
Letting $N(\Delta_X)$ be a closed tubular neighborhood of $\Delta_X$, we have
\[ H_{i+1}(X^{\times 2},\Conf_2(X);A^{\boxtimes 2})\cong H_{i+1}(N(\Delta_X),\partial N(\Delta_X);A^{\otimes 2}) \]
by excision. Since $X$ is parallelizable and the normal bundle of $\Delta_X$ can be canonically identified with $TX$, the normal bundle of $\Delta_X$ is trivial. By Assumption~\ref{assum:delta-trivial}, we have
\[\begin{split}
& H_{i+1}(N(\Delta_X),\partial N(\Delta_X);A^{\otimes 2})
=H_{n+1}(D^{n+1},\partial D^{n+1};R)\otimes_R H_{i-n}(\Delta_X;A^{\otimes 2})\\
&\cong H_{n+1}(D^{n+1},\partial D^{n+1};R)\otimes_R H_{i-n}(X;R)\cong H_{i-n}(X;R).
\end{split} \]
Here, $H_{i-n}(X;R)$ is rank 1 for $i-n=0,1,n,n+1$, and its generator is $*,K,\Sigma,X$, respectively.
\end{proof}

\begin{Lem}
Let $n, X, K, A$ be as in Lemma~\ref{lem:H(Conf)}.
Let $s_{\tau_0}\colon X\to ST(X)$ be the section given by the normalization of the first vector of a framing $\tau_0$ of $X$. Then we have
\[ \begin{split}
	H_{n+1}(\partial \bConf_2(X);A^{\otimes 2})=R[ST(K)\otimes c_A] \oplus R[s_{\tau_0}(X)\otimes c_A].
\end{split} \]
\end{Lem}
\begin{proof}
This follows from the trivialization $\partial \bConf_2(X)\cong S^n\times X$ induced by $\tau_0$, the K\"{u}nneth formula for $R$-modules, and Assumption~\ref{assum:delta-trivial}.
\end{proof}

Since $H_{n+2}(\bConf_2(X);A^{\boxtimes 2})=0$ by Lemma~\ref{lem:H(Conf)}, we have the following exact sequence of $R$-modules.
\[\begin{split}
0&\to H_{n+2}(\bConf_2(X),\partial\bConf_2(X);A^{\boxtimes 2})\\
&\stackrel{r}{\to} H_{n+1}(\partial\bConf_2(X);A^{\otimes 2})
\stackrel{i}{\to} H_{n+1}(\bConf_2(X);A^{\boxtimes 2})
\end{split} \]

\begin{Cor}\label{cor:prop-exists}
Let $n,X, K, A$ be as in Lemma~\ref{lem:H(Conf)}.
There exists an element $\calO_{\tau_0}(X,A)\in R$ such that 
\[ i([s_{\tau_0}(X)\otimes c_A])=\calO_{\tau_0}(X,A)[ST(K)\otimes c_A].
 \]
Hence we have
\[ [s_{\tau_0}(X)\otimes c_A] - \calO_{\tau_0}(X,A)[ST(K)\otimes c_A]\in \mathrm{Ker}\,i=\mathrm{Im}\,r. \]
\end{Cor}
\begin{proof}
This follows from $H_{n+1}(\bConf_2(X);A^{\boxtimes 2})=R[ST(K)\otimes c_A]$ of Lemma~\ref{lem:H(Conf)}.
\end{proof}

\begin{Def}[Propagator]
Let $X, K, A$ be as in Lemma~\ref{lem:H(Conf)}.
A {\it propagator} is an $(n+2)$-chain $\omega$ of $\bConf_2(X)$ with coefficients in $A^{\boxtimes 2}$ that is transversal to the boundary and that satisfies
\[ \partial_A\, \omega = s_{\tau_0}(X)\otimes c_A - \calO_{\tau_0}(X,A)\,ST(K)\otimes c_A. \]
\end{Def}

\begin{Rem}
For 3-manifolds with $b_1=1$, Lescop described $\calO_{\tau_0}(X,\C(t))$ in terms of the logarithmic derivative of the Alexander polynomial (\cite{Les1}). 
\end{Rem}

\begin{Exa}\label{ex:O=0}
Let $X=\Sigma^n\times S^1$ ($\Sigma^n$: parallelizable homology $S^n$) and $\pi'=\pi_1(\Sigma^n)$. We consider $A=A_1\otimes_\C\Lambda$ of \S\ref{ss:acyclic-complex} as a $\C[\pi]$-module by the $\pi=\pi'\times \Z$-action 
\[ (g\times n)(v\otimes f(t))=\rho_{A_1}(g)(v)\otimes t^nf(t).\]
If $C_*(\Sigma^n;A_1)$ is acyclic, so is $C_*(X;A)=C_*(\Sigma^n;A_1)\otimes_\C C_*(S^1;\Lambda)$. Let $\tau_0$ be a framing of $X$ such that the first vector is tangent to $S^1$ and the remaining $n$ vectors are tangent to $\Sigma^n$. 

\begin{Prop}\label{prop:O=0}
Under the above assumption, we have $ \calO_{\tau_0}(X,A)=0$.
\end{Prop}
\begin{proof}
For each point $(x,y)$ of $\bConf_2(\Sigma^n)$, the chain $(x,y)\times S^1=\{(x,\theta)\times (y,\theta)\mid \theta\in S^1\}$ of $\bConf_2(X)$ is defined. Similarly, for each $k$-chain $\sigma$ of $\bConf_2(\Sigma^n)$, the $(k+1)$-chain $\sigma\times S^1$ of $\bConf_2(X)$ is defined. It follows from the invariance of the coefficient $A\boxtimes_\Lambda A=A_1^{\boxtimes 2}[t^{\pm 1}]$ under the diagonal action of $\Lambda$ that $\sigma\times S^1$ is a cycle over $A_1^{\boxtimes 2}[t^{\pm 1}]$ if $\sigma$ is a cycle and that the correspondence $\sigma\mapsto \sigma\times S^1$ induces a well-defined map
\[ \cdot\times S^1\colon H_i(\bConf_2(\Sigma^n);A_1^{\boxtimes 2})\to H_{i+1}(\bConf_2(X);A_1^{\boxtimes 2}[t^{\pm 1}]). \]
Similarly, we have the following commutative diagram.
\[ \xymatrix{
  H_{n+2}(\bConf_2(X),\partial\bConf_2(X);A_1^{\boxtimes 2}[t^{\pm 1}]) \ar[r]^-{r} &
  H_{n+1}(\partial\bConf_2(X);A_1^{\otimes 2}[t^{\pm 1}])\\
  H_{n+1}(\bConf_2(\Sigma^n),\partial\bConf_2(\Sigma^n);A_1^{\boxtimes 2}) \ar[r]^-{r'} \ar[u]^-{\cdot\times S^1}&
  H_{n}(\partial\bConf_2(\Sigma^n);A_1^{\otimes 2}) \ar[u]_-{\cdot\times S^1}
}\]
One may see by an argument similar to Lemma~\ref{lem:H(Conf)} that 
\[ H_i(\bConf_2(\Sigma^n);A_1^{\boxtimes 2})=\left\{\begin{array}{ll}
R[ST'(*)\otimes c_{A_1}] & (i=n-1)\\
R[ST'(\Sigma^n)\otimes c_{A_1}] & (i=2n-1)\\
0 & (\mbox{otherwise})
\end{array}\right. \]
where $ST'(\sigma)$ is the restriction of the unit tangent bundle of $\Sigma^n$ to $\sigma$, and that $H_n(\partial\bConf_2(\Sigma^n);A_1^{\otimes 2})$ is spanned over $R$ by $[s_{\tau_0}'(\Sigma^n)\otimes c_{A_1}]$. The following holds.
\[
 [(s_{\tau_0}'(\Sigma^n)\times S^1)\otimes c_{A_1}] = [s_{\tau_0}(X)\otimes c_{A_1}] \in H_{n+1}(\partial\bConf_2(X);A_1^{\otimes 2}[t^{\pm 1}])
 \]
Moreover, by an argument similar to Corollary~\ref{cor:prop-exists}, it follows that there exists an $(n+1)$-chain $\omega_{\Sigma^n}$ of $\bConf_2(\Sigma^n)$ with coefficients in $A_1^{\boxtimes 2}$ that satisfies 
\[ r'([\omega_{\Sigma^n}])=[s_{\tau_0}'(\Sigma^n)\otimes c_{A_1}]. \]
Then by the commutativity of the above diagram, we have
\[ r([\omega_{\Sigma^n}\times S^1])=[s_{\tau_0}(X)\otimes c_{A_1}]. \]
Since this belongs to the kernel of the map $i\colon H_{n+1}(\partial\bConf_2(X);A_1^{\otimes 2}[t^{\pm 1}])\to H_{n+1}(\bConf_2(X);A_1^{\boxtimes 2}[t^{\pm 1}])$, it follows that $\calO_{\tau_0}(X,A)=0$.
\end{proof}
\end{Exa}

\begin{Prop}\label{prop:Sigma-delta-trivial}
The local coefficient system $A=\mathfrak{g}_\rho[t^{\pm}]$ on $X=\Sigma(2,3,5)\times S^1$ of Example~\ref{ex:235} satisfies Assumption~\ref{assum:delta-trivial}.
Namely, we have $H_*(\Delta_X;\mathfrak{g}_\rho^{\otimes 2}[t^{\pm}])\cong H_*(X;\C[t^{\pm 1}])$ and it is generated by $\sigma\otimes c_A$ for cycles $\sigma$ of $X$. 
\end{Prop}
\begin{proof}
It is well-known that the adjoint representation of $SU(2)$ on $\mathfrak{g}=\mathfrak{sl}_2$ is irreducible and $\mathfrak{g}^{\otimes 2}$ equipped with the diagonal adjoint $SU(2)$-action $\mathrm{Ad}\otimes \mathrm{Ad}$ has the following decomposition into irreducible $SU(2)$-modules:
\[ \mathfrak{g}\otimes \mathfrak{g}\cong V_0\oplus V_2\oplus V_4, \]
where $V_m$ is the irreducible representation of $\mathfrak{sl}_2$ of highest weight $m$ and hence of dimension $m+1$ (Clebsch-Gordan formula, e.g., \cite[Propotision~V.5.1]{Kas}). More explicitly, $V_0\cong \C$ is the trivial representation spanned by $c_{\mathfrak{g}}$, and $V_2\cong\mathfrak{g}$ as $SU(2)$-modules. This together with (\ref{eq:acyclic}) in Example~\ref{ex:235} shows that
\[ \begin{split}
	H_*(\Delta_X;\mathfrak{g}_\rho^{\otimes 2}[t^{\pm 1}])
	&\cong H_*(\Delta_X;\C[t^{\pm 1}])\oplus H_*(\Delta_X;\mathfrak{g}_\rho[t^{\pm 1}])
	\oplus H_*(\Delta_X;V_4[t^{\pm 1}]_\rho)\\
	&\cong H_*(\Delta_X;\C[t^{\pm 1}])\oplus H_*(\Delta_X;V_4[t^{\pm 1}]_\rho),
\end{split} \]
where $\pi_1(\Delta_X)$ acts trivially on $\C[t^{\pm 1}]$. 
Hence it suffices to prove the vanishing of $H_*(\Sigma(2,3,5);(V_4)_{\rho})$. Here, we put the subscript $\rho$ to emphasize that the representation is given through $\rho$. 

The module $V_m=\Sym^m V$, $V=\C^2$, can be considered as the space of homogeneous (commutative) polynomials of degree $m$ in two variables, on which $g\in SU(2)$ acts by
\[ (g\cdot f)\2vec{x}{y}= f(g^{-1}\2vec{x}{y}). \]
(e.g., \cite[IV.1]{Kna}.) One can see that $H^0(\Sigma(2,3,5);(V_4)_{\rho})\cong V_4^{\pi'}=\{v\in V_4\mid g\cdot v=v\,\,(\forall g\in \pi')\}$ is zero. In other words, $V_4$ is an irreducible $\pi'$-module (Proposition~\ref{prop:V4-irred}).

Since $\pi'=\pi_1\Sigma(2,3,5)$ is finite, $\C[\pi']$ is semisimple in the sense of \cite[\S{I.4}]{CE} by Maschke's theorem and we have $H^1(\pi';(V_4)_\rho)=0$ (Theorem~VI.16.6 and Lemma~VI.16.7 of \cite{HS}). By the universal coefficient theorem, which is valid if the ring is hereditary, the sequence
\[ 0\to \mathrm{Ext}_{\C[\pi']}^1(H_{i-1}(C),W)\to H^i(\mathrm{Hom}_{\C[\pi']}(C,W))\to \mathrm{Hom}_{\C[\pi']}(H_i(C),W)\to 0 \]
is exact for $C=S_*(S^3;\C)$ (as a $\C[\pi']$-module) and any $\C[\pi']$-module $W$ (e.g., \cite[Theorem~VI.3.3]{CE}). Hence we have
\[ \begin{split}
&H^3(\Sigma(2,3,5);W)\cong \mathrm{Hom}_{\C[\pi']}(\C,W)\cong H^0(\pi';W)=W^{\pi'},\\
&H^2(\Sigma(2,3,5);W)=0,\\
&H^1(\Sigma(2,3,5);W)\cong\mathrm{Ext}_{\C[\pi']}^1(\C,W)=H^1(\pi';W),\\
&H^0(\Sigma(2,3,5);W)\cong\mathrm{Hom}_{\C[\pi']}(\C,W)\cong H^0(\pi';W)=W^{\pi'}.
\end{split} \]
Namely, we have $H^*(\Sigma(2,3,5);(V_4)_\rho)=0$. Then by Poincar\'{e} duality (e.g., \cite[\S{3.H}]{Hat2}, \cite[Theorem~7.17]{Hatt} etc.), we also have $H_*(\Sigma(2,3,5);(V_4)_{\rho})=0$. 
\end{proof}
\begin{Rem}\label{rem:Sigma-delta-trivial}
A result similar to Proposition~\ref{prop:Sigma-delta-trivial} holds also for $X=S^{4k-1}/\pi'$ ($k\geq 2$), in which case $H_*(C)=H_*(S^{4k-1};\C)$ and
\[ H^i(X;W)\cong\left\{\begin{array}{ll}
W^{\pi'} & (i=0,4k-1),\\
H^1(\pi';W) & (i=1),\\
0 & (\mbox{otherwise}),
\end{array}\right. \]
which vanishes for all $i$ when $W=V_4$.
\end{Rem}

%%%%%%%%%%%%%%%%%%%%%%%%%%%%%%%
%%%%%%%%%%%%%%%%%%%%%%%%%%%%%%%
\mysection{Framed fiber bundles and their fiberwise configuration spaces}{s:bundle}

We make an assumption on $X$-bundles $\pi\colon E\to B$ with local coefficient system, and we compute the homology of the $\bConf_2(X)$-bundle $E\bConf_2(\pi)$ over $B$ associated to $\pi$ with fiber the configuration space of two points. We also recall a geometric interpretation of chains with local coefficients and intersections among them. 

%%%%%%%%%%%%%%%%%%%%%%%%%%%%%%%
\subsection{Moduli spaces of manifolds with some structures}\label{ss:moduli-sp}

Let $\Sigma^n$ be a stably parallelizable $n$-manifold and let $X=\Sigma^n\times S^1$. We fix a basepoint $x_0\in X$. Let 
\[ K=\{*\}\times S^1\subset X, \]
and we assume that $K$ is disjoint from $x_0$. Let $\tau_0\colon TX\to \R^{n+1}\times X$ be the standard framing, i.e., the one obtained from a stable framing on $\Sigma^n\times\{*\}$ by $S^1$-symmetry. In the following, we assume that $\pi\colon E\to B$ is an $X$-bundle with structure group $\Diff_0(X^\bullet,\partial)$, equipped with the following data.
\begin{enumerate}
\item A trivialization of the restriction of $\pi$ on a tubular neighborhood of the basepoint section $\widetilde{x}_0$. 
\item A smooth trivialization of the vertical tangent bundle $T^vE=\mathrm{Ker}\,d\pi$ over $E$
\[ \tau\colon T^vE\to \R^{n+1}\times E\]
that agrees with $\tau_0$ on the base fiber, and that agrees with the trivialization of the normal bundle of $\widetilde{x}_0$ induced from the item 1.
\item A fiberwise pointed homotopy equivalence $f\colon E\to X$.
\end{enumerate}
The classifying space for $X$-bundles $\pi$ with such structures is given by $\widetilde{B\Diff}_{\mathrm{deg}}(X^\bullet,\partial)$, defined in \S\ref{ss:main-result}. We have the following fibration sequences.
\begin{equation}\label{eq:fibrations}
 \begin{split}
\calF_*(X)\to &\widetilde{B\Diff}(X^\bullet,\partial)\to B\Diff(X^\bullet,\partial)\\
\calF_*(X)\times \mathrm{Map}_*^{\mathrm{deg}}(X,X)\to &\widetilde{B\Diff}_{\mathrm{deg}}(X^\bullet,\partial)\to B\Diff_0(X^\bullet,\partial)
\end{split} 
\end{equation}

\begin{Prop}\label{prop:pi-F-M}
Suppose that $\Sigma^n$ is a homology $n$-sphere and $X=\Sigma^n\times S^1$. Then for $p\geq 1$, $\pi_p\calF_*(X)$ and $\pi_p\mathrm{Map}_*^{\mathrm{deg}}(X,X)$ are finitely generated abelian groups.
\end{Prop}
\begin{proof}
This follows immediately from the following two Lemmas~\ref{lem:pi-F} and \ref{lem:pi-M}.
\end{proof}

\begin{Lem}\label{lem:pi-F}
Let $p\geq 1$.
If $\Sigma^n$ is a homology $n$-sphere and $X=\Sigma^n\times S^1$, 
\[ \pi_p \calF_*(X) \cong \pi_{p+1} SO_{n+1} \oplus \pi_{p+n} SO_{n+1} \oplus \pi_{p+n+1} SO_{n+1}. \]
\end{Lem}
\begin{proof}
Fixing the standard framing $\tau_0$ gives a homotopy equivalence
\[ \calF_*(X)=\mathrm{Map}_*(X,GL_{n+1}^+(\R))\simeq\mathrm{Map}_*(X,SO_{n+1}), \]
where $GL_{n+1}^+(\R)$ is the space of orientation preserving linear isomorphisms of $\R^{n+1}$.
When $X=\Sigma^n\times S^1$, where $\Sigma^n$ is a homology $n$-sphere, we take a pointed degree 1 map $c\colon X\to S^n\times S^1$. 
One can see that $\Omega^p\mathrm{Map}_*(X,SO_{n+1})\cong \mathrm{Map}_*(S^p\wedge X,SO_{n+1})$, and the iterated suspension $S^p c\colon S^p\wedge X\to S^p\wedge (S^n\times S^1)$ is a homology equivalence between 1-connected spaces. By Whitehead's theorem, $S^p c$ is a homotopy equivalence. Hence $S^p c$ induces an isomorphism
\[ \pi_0 \mathrm{Map}_*(S^p\wedge X,SO_{n+1}) \stackrel{\cong}{\to} \pi_0 \mathrm{Map}_*(S^p\wedge (S^n\times S^1),SO_{n+1}). \]
The result follows by $S^p\wedge (S^n\times S^1)\simeq S^p(S^n)\vee S^p(S^1)\vee S^p(S^n\wedge S^1)\simeq S^{p+n}\vee S^{p+1}\vee S^{p+n+1}$ (e.g., \cite[Proposition~4I.1]{Hat2}). 
\end{proof}

Proof of the next lemma is the same as that of Lemma~\ref{lem:pi-F}. Note that $\Map_*^\mathrm{deg}(X,X)$ is homotopy equivalent to the identity component of $\Map_*(X,X)$.
\begin{Lem}\label{lem:pi-M}
Let $p\geq 1$.
If $\Sigma^n$ is a homology $n$-sphere and $X=\Sigma^n\times S^1$,
\[ \pi_p \Map_*^\mathrm{deg}(X,X) \cong \pi_{p+1} X\oplus \pi_{p+n} X\oplus \pi_{p+n+1} X. \]
\end{Lem}

For example, when $n=3$, $p=1$, and $\Sigma^3=\Sigma(2,3,5)$, we have
\begin{equation}\label{eq:pi-F}
 \begin{split}
&\pi_1\calF_*(X)\cong \pi_2SO_4\oplus \pi_4SO_4\oplus \pi_5SO_4\cong \Z_2\oplus\Z_2\oplus\Z_2\oplus\Z_2,\\
&\pi_1\Map_*^{\mathrm{deg}}(X,X)\cong \pi_2X\oplus\pi_4X\oplus \pi_5X\cong \Z_2\oplus \Z_2.
\end{split} 
\end{equation}
More generally, when $n=4k-1$, $p=n-2=4k-3$ ($k\geq 2$), and $\Sigma^n=S^n/\pi_1\Sigma(2,3,5)$, we have
\begin{equation}\label{eq:pi-F-4k-1}
 \begin{split}
&\pi_{4k-3}\calF_*(X)\cong \pi_{4k-2}SO_{4k}\oplus \pi_{8k-4}SO_{4k}\oplus \pi_{8k-3}SO_{4k},\\
&\pi_{4k-3}\Map_*^{\mathrm{deg}}(X,X)\cong \pi_{4k-2}S^{4k-1}\oplus\pi_{8k-4}S^{4k-1}\oplus \pi_{8k-3}S^{4k-1}.
\end{split} 
\end{equation}
It is known that $\pi_i S^{4k-1}$ is finite if $i\neq 4k-1$ (e.g. \cite[5.1 (Serre Classes)]{Hat3}), and $\pi_iSO_{4k}$ is finite if $i>4k-5$ (e.g. \cite[3.D]{Hat2}).

\begin{Cor}\label{cor:abel-infinite}
\begin{enumerate}
\item If the abelianization of the group $\pi_1\widetilde{B\Diff}_{\mathrm{deg}}(X^\bullet,\partial)$ has countable infinite rank, then so is the abelianization of $\pi_1B\Diff_0(X^\bullet,\partial)$.
\item For $p\geq 2$, if the abelian group $\pi_p\widetilde{B\Diff}_{\mathrm{deg}}(X^\bullet,\partial)$ has countable infinite rank, then so is the abelian group $\pi_pB\Diff_0(X^\bullet,\partial)$.
\end{enumerate}
\end{Cor}
\begin{proof}
For 1, we consider the exact sequence 
\[\pi_1(\calF_*(X)\times \mathrm{Map}_*^{\mathrm{deg}}(X,X))
\to \pi_1\widetilde{B\Diff}_{\mathrm{deg}}(X^\bullet,\partial)
\to \pi_1B\Diff_0(X^\bullet,\partial)\to 0 \]
for the fibration (\ref{eq:fibrations}).
If we put $G=\pi_1\widetilde{B\Diff}_{\mathrm{deg}}(X^\bullet,\partial)$, and let $H\subset G$ be the image from $\pi_1(\calF_*(X)\times \mathrm{Map}_*^{\mathrm{deg}}(X,X))$, then by exactness, $H$ is a normal subgroup of $G$, $G/H$ is isomorphic to $\pi_1B\Diff_0(X^\bullet,\partial)$, and we have the homology exact sequence
\[ H_1(H;\Z)_G\to H_1(G;\Z)\to H_1(G/H;\Z)\to 0, \]
which is a part of the five term exact sequence for group homology (e.g., \cite[Corollary~VI.8.2]{HS}). Since $H$ is the image from a finitely generated abelian group (of finite rank), and $H_1(G;\Z)$ has countable infinite rank by assumption, the result follows. 

The assertion 2 follows immediately from the long exact sequence for homotopy groups of fibration, and that the homotopy groups of the fiber is finitely generated abelian groups by Proposition~\ref{prop:pi-F-M}.
\end{proof}

\begin{Prop}\label{prop:puncture}
\begin{enumerate}
\item If the abelianization of the group $\pi_1B\Diff_0(X^\bullet,\partial)$ has countable infinite rank, then so is the abelianization of $\pi_1B\Diff_0(X)$.
\item For $p\geq 2$, if the abelian group $\pi_pB\Diff_0(X^\bullet,\partial)$ has countable infinite rank, then so is the abelian group $\pi_pB\Diff_0(X)$.
\end{enumerate}
\end{Prop}
\begin{proof}
In the fibration sequence,
\[ \Diff_0(X^\bullet,\partial)\to \Diff_0(X)\to \Emb^\fr(\{x_0\},X), \]
we have $\Emb^\fr(\{x_0\},X)\simeq SO_{n+1}\times X$. Hence for $p\geq 2$, the cokernel of the homomorphism $\pi_p(SO_{n+1}\times X)\to \pi_pB\Diff_0(X^\bullet,\partial)$ has countable infinite rank. For $p=1$, the result follows from the exact sequence for group homology as in the proof of Corollary~\ref{cor:abel-infinite}. 
\end{proof}

\begin{Rem}
The results in this subsection hold even if $X$ is replaced with a parallelizable $\Z$-homology $\Sigma^n\times S^1$. In that case, one should choose $K$ and $\tau_0$ to define $\widetilde{B\Diff}_{\mathrm{deg}}(X^\bullet,\partial)$, which may not be canonical. 
\end{Rem}

%%%%%%%%%%%%%%%%%%%%%%%%%%%%%%%
\subsection{Local coefficient system on $E$}\label{ss:localc}

We shall recall a geometric interpretation of singular chains of the total space $E$ with local coefficients. Recall that a local coefficient system (with a fixed module $A$) on a space $Z$ is given by a groupoid homomorphism $\Pi(Z)\to \mathrm{Aut}(A)$, where $\Pi(Z)$ is the fundamental groupoid of $Z$ (e.g., \cite[Ch.VI]{Wh}, \cite[II-Ch.6]{Hatt}). We consider concrete local coefficient systems on $X$ and $E$ induced by a stratification of $X$ from a Morse function to make some genericity assumption (``general position with respect to holonomy'', \S\ref{s:invariant}) to simplify the main computation of the intersection invariant (Proof of Theorem~\ref{thm:surgery}). 

Namely, let $f\colon X\to \R$ be a Morse function whose number of critical points of index 0 is one, and let $\xi$ be its gradient-like vector field that is Morse--Smale, i.e., the descending and ascending manifolds of $\xi$ intersect transversally (e.g., \cite[\S{2.2}]{AD}). In the handle decomposition of $X$ with respect to $\xi$, the 2-skeleton gives a presentation of $\pi_1(X)$. If we let $p$ denote the critical point of $f$ of index 0, the compactification $\bcalA_p(\xi)$ of its ascending manifold has a structure of manifold with corners, and its codimension 1 strata consists of the ascending manifolds of critical points of index 1 (e.g., \cite[\S{4.9}]{AD}, \cite[Theorem~1]{BH}). This is not a submanifold of $X$, but there is a canonical smooth map $\bcalA_p(\xi)\to X$ whose restriction to the interior is an embedding to an open dense subset of $X$. We fix an orientation of $\bcalA_r(\xi)$ for each critical point $r$ of $f$.

The stratification on $\bcalA_p(\xi)$ allows to define a holonomy on a path in $X$. A generic smooth path in $X$ may intersect the codimension 1 strata of $\bcalA_p(\xi)$ transversally finitely many times. The sequence of intersections with codimension 1 strata aligned on the path gives a word in the generator of the presentation of $\pi_1(X)$. Namely, if the descending manifold of each critical point $q$ of index 1 represents a generator $x_q$ of $\pi_1(X)$ and if
a path intersects the ascending manifolds of critical points $q_1,q_2,\ldots,q_r$ of index 1 in this order, then the word $x_{q_r}^{\pm 1}\cdots x_{q_2}^{\pm 1}x_{q_1}^{\pm 1}$ (the signs are determined by the orientations of the intersection) gives an element of $\pi_1(X)$. When an endpoint of a path is on a codimension 1 stratum, we push the endpoint slightly in the positive direction of the coorientation of the stratum to determine a word. Here, we fix the coorientations of the ascending manifolds of codimension 1 in a way that the intersection with the corresponding generator of $\pi_1(X)$ is counted positively. When an endpoint of a path is on a stratum of codimension $\geq 2$, we consider similarly, namely, we slightly push it into a positive direction against all the adjacent codimension 1 strata. 
Then a homomorphism
\[ \mathrm{Hol}_X\colon \Pi(X) \to \pi_1(X) \]
is obtained. Furthermore, by postcomposing with the representation $\rho_A\colon \pi_1(X)\to \mathrm{End}_\C(A)$, we obtain a local coefficient system over $X$, in the sense of \cite[VI.1]{Wh}. We say that this local coefficient system is \emph{trivial} if $\rho_A$ is the map to the identity.

We assume that the $X$-bundle $\pi\colon E\to B$ is equipped with a fiberwise pointed degree 1 map $q\colon E\to X$. In this case, by pulling back the local coefficient system over $X$ by $q$, we obtain a local coefficient system on $E$. More precisely, taking the holonomy of a smooth path $\gamma$ in $E$ by that of a path $q\circ \gamma$ in $X$ gives a homomorphism
\[ \mathrm{Hol}_E\colon \Pi(E)\to \pi_1(X), \]
and by postcomposing with the representation $\rho_A\colon \pi_1(X)\to \mathrm{End}_\C(A)$, we obtain a local coefficient system over $E$. 

%%%%%%%%%%%%%%%%%%%%%%%%%%%%%%%
\subsection{Chains of $E$ with local coefficients}

Let $\widetilde{E}$ be the $\pi$-covering over $E$ defined by pullback of the universal cover $\widetilde{X}\to X$ by the pointed degree 1 map $q\colon E\to X$. 
We have a non-canonical identification
\[ S_p(\widetilde{E})\otimes_{\C \pi}A = S_p(E)\otimes_\C A \]
given as follows. 

We take a basepoint $\delta_0$ at the barycenter of the standard $p$-simplex $\Delta^p$. A smooth simplex $\widetilde{\sigma}\colon \Delta^p\to \widetilde{E}$ corresponds bijectively to the pair of a smooth simplex $\sigma\colon \Delta^p\to E$ and 
the equivalence class of a path of $E$ from $\sigma(\delta_0)$ to the basepoint $x_0$ in a fixed fiber $X$, where we consider two such paths are equivalent if the endpoints agree and if their image under the homomorphism $q\colon \Pi(E)\to \Pi(X)$ agree. Moreover, through the holonomy homomorphism $\Hol_X\colon \Pi(E)\to \pi_1(X)$, the equivalence class of a base path in $E$ with fixed endpoints corresponds bijectively to an element of $\pi$. 

\begin{Claim}
This gives a bijective correspondence between $\widetilde{\sigma}\colon \Delta^p\to\widetilde{E}$ and a pair $(\sigma,g)$ of a smooth simplex $\sigma\colon \Delta^p\to X$ and an element $g$ of $\pi_1(X)$. 
\end{Claim}

If we denote this pair by $\sigma\cdot g$, a $\C$-chain of $\widetilde{E}$ is formally written as a finite sum
\[ \sum_{g\in \pi} \sigma_g\cdot g, \]
where $\sigma_g$ is a $\C$-chain of $E$. We obtain an identification $S_p(\widetilde{E})\otimes_{\C\pi}A=S_p(E)\otimes_{\C}A$
by the transformation
\[ \sum_g \sigma_g\cdot g\otimes a_g = \sum_g \sigma_g\otimes g\cdot a_g, \]
where we identify smooth simplices in $S_p(E)$ with that whose basepoints are mapped to the fundamental domain of $X$. 
Note that this correspondence depends on the stratification of $X$ and may not be canonical.

The boundary operator $\partial_A=\partial\otimes\mathbbm{1}$ of $S_p(\widetilde{E})\otimes_{\C\pi}A$ induces a twisted boundary operator on $S_p(E)\otimes_{\C}A$, which can be described as follows. For a smooth simplex of $E$ with a base path $(\sigma,\gamma)$, an induced base path for a face $\sigma_j$ of $\sigma$ can be defined by connecting $\gamma$ and the segment $\eta_j$ between the barycenters of $\sigma$ and $\sigma_j$ (see Figure~\ref{fig:path-simplex}). The boundary of $(\sigma,\gamma)$ is the sum of such faces $(\sigma_j,\gamma\circ \eta_j)$:
\[ \partial (\sigma,\gamma)=\sum_j(\sigma_j,\gamma\circ \eta_j). \]
If a segment $\eta_j$ between the barycenters meets a codimension 1 stratum of $\bcalA_p(\xi)$, the coefficient of $\sigma_j$ in $\partial (\sigma,\gamma)$ differs from that of $\sigma$ by the action of $\eta_j$. Thus the twisted boundary of $S_p(E)\otimes_{\C}A$ is given by the formula
\[ \partial_A (\sigma\otimes a)=\sum_j \sigma_j\otimes \mathrm{Hol}_E(\eta_j)\cdot a. \]
\begin{Rem}
Note that the basepoint need not be taken at the barycenter, as long as it is fixed on the simplex. Usually, the basepoint is taken at the 0-th vertex of a simplex (\cite[Ch.VI]{Wh}). We choose the barycenter since it makes the incidence coefficients symmetric.
\end{Rem}
\begin{figure}
\begin{center}
\includegraphics{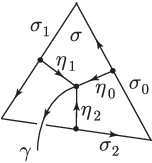}
\end{center}
\caption{A singular simplex $\sigma$ with a path $\gamma$ to $x_0$. }\label{fig:path-simplex}
\end{figure}

%%%%%%%%%%%%%%%%%%%%%%%%%%%%%%%
\subsection{Chains of family of configuration spaces with local coefficients}\label{ss:twisted-chain-family}

We shall explain a geometric interpretation of singular chains of the family $E\bConf_2(\pi)$ of configuration spaces (Definition~\ref{def:EConf_2} below) with local coefficients.
The direct product of projections $\beta\colon \widetilde{X}\times \widetilde{X}\to X\times X$ is a $\pi\times\pi$-covering of $X\times X$. The space $\widetilde{X}\times \widetilde{X}$ is the set of pairs $([\gamma_1],[\gamma_2])$ of equivalence classes of paths $\gamma_1,\gamma_2$ in $X$ to the basepoint $x_0\in X$, where we say that two such paths are equivalent if they are relatively homotopic in $X$ fixing the endpoints.
The subspace $\beta^{-1}\Delta_X\subset \widetilde{X}\times\widetilde{X}$ is a disjoint union of copies of lifts of $\Delta_X$ and the set of components in $\beta^{-1}\Delta_X$ corresponds bijectively to $\pi_1(X)$. 
The group $\Diff(X^\bullet,\partial)$ acts diagonally on $\widetilde{X}\times\widetilde{X}$ by $([g\circ \gamma_1],[g\circ \gamma_2])$, and on 
\[ \bConf_\vee(X)=B\ell(\widetilde{X}\times\widetilde{X},\beta^{-1}\Delta_X). \]

\begin{Def}\label{def:EConf_2}
We denote by 
\[ \begin{split}
&\bConf_2(\pi)\colon E\bConf_2(\pi)\to B\quad\mbox{ and } \quad \bConf_\vee(\pi)\colon E\bConf_\vee(\pi)\to B
\end{split} \]
the associated $\bConf_2(X)$-bundle with structure group $\Diff(X^\bullet,\partial)$ and $\bConf_\vee(X)$-bundle to the $X$-bundle $\pi\colon E\to B$. 
\end{Def}

As before, a chain of $S_p(E\bConf_\vee(\pi))\otimes_{\C[\pi^2]}A^{\boxtimes 2}$ can be considered as that of $S_p(E\bConf_2(\pi))\otimes_\C A^{\boxtimes 2}$ as follows. 
The local coefficient system $A$ on $E$ induces that on 
\[ E\times_{B}E=\{(x,y)\in E\times E\mid \pi(x)=\pi(y)\}\]
with coefficients in $A^{\boxtimes 2}$. Its structure can be interpreted as follows. A point of $E\times_{B}E$ can be given by data $(x_1,x_2,s)$ ($s\in B$, $x_1,x_2\in \pi^{-1}(s)$). The fiberwise universal cover $\widetilde{E}\times_{B}\widetilde{E}$ of $E\times_{B}E$ is given by the set of all data $(x_1,x_2,s,[\gamma_1],[\gamma_2])$, where $\gamma_i\colon [0,1]\to E$ is a path in a fiber $\pi^{-1}(s)$ from $x_i$ to the basepoint $\pi^{-1}(s)\cap \widetilde{x}_0$, and $[\gamma_i]$ is the equivalence class of $q\circ \gamma_i\colon [0,1]\to X$ under homotopy fixing endpoints, where $q\colon E\to X$ is the pointed degree 1 map fixed in \S\ref{ss:localc}.

A path $\eta$ in $E\times_B E$ is projected to a pair of two paths $(\eta_1,\eta_2)$ in $E$ by the fiberwise projections $\mathrm{pr}_i\colon E\times_B E\to E$, ($i=1,2$). By taking the holonomy of each path, we obtain an element of $\pi\times \pi$ (see Figure~\ref{fig:path-path}). This gives a homomorphism
\[ \mathrm{Hol}_{E\times_B E}\colon \Pi(E\times_B E)\to \mathrm{End}_{\C}(A^{\otimes 2}) \]
and gives a local coefficient system on $E\times_{B}E$ with coefficients in the $\pi\times \pi$-module $A^{\boxtimes 2}$. 
\begin{figure}
\begin{center}
\includegraphics{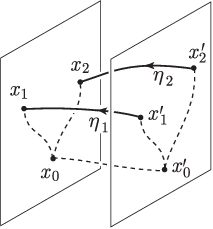}
\end{center}
\caption{A path $\eta$ between configurations in fibers}\label{fig:path-path}
\end{figure}

The local coefficient system on $E\times_{B}E$ given here induces that on $E\bConf_2(\pi)$. A chain of the $\pi\times\pi$-covering $E\bConf_\vee(\pi)$ of $E\bConf_2(\pi)$ has the following expression
\[ \sum_{(g,h)\in\pi^2}\sigma_{g,h}\cdot (g,h), \]
where $\sigma_{g,h}$ is a $\C$-chain of $E\bConf_2(\pi)$. By forwarding the action of $(g,h)$ to that on $A^{\boxtimes 2}$, a chain of $S_p(E\bConf_2(\pi))\otimes_\C A^{\boxtimes 2}$ is obtained.

The twisted boundary operator on $S_p(E\bConf_2(\pi))\otimes_\C A^{\boxtimes 2}$ can be defined by considering the $\pi\times\pi$-holonomy of the path $\eta$ between barycenters as in the case of $E$.

%%%%%%%%%%%%%%%%%%%%%%%%%%%%%%%
\subsection{Propagator in a family for $n$ odd}

\begin{Lem}\label{lem:E2-EConf}
We assume that $n$ is odd, $n\geq 3$. Let $X$ be a parallelizable $\Z$ homology $S^n\times S^1$ equipped with a local coefficient system $A$, let $\pi\colon E\to B$ be an $X$-bundle as in \S\ref{ss:moduli-sp}, and we assume that the base of $B$ is an $(n-2)$-dimensional closed oriented manifold. Under Assumptions~\ref{assum:acyclic} and \ref{assum:delta-trivial}, we have
\[ \begin{split}
	&H_{2n}(E\bConf_2(\pi);A^{\boxtimes 2})\cong H_0(B)\otimes_{\C}H_{2n}(\bConf_2(X);A^{\boxtimes 2})\\
	&H_{2n}(E\bConf_2(\pi),\partial E\bConf_2(\pi);A^{\boxtimes 2})\\
	&\hspace{30mm}\cong H_{n-2}(B)\otimes_\C H_{n+2}(\bConf_2(X),\partial \bConf_2(X);A^{\boxtimes 2})
\end{split} \]
and the natural map
\begin{equation}\label{eq:H2n-restrict}
 H_{2n}(E\bConf_2(\pi);A^{\boxtimes 2})\to H_{2n}(E\bConf_2(\pi),\partial E\bConf_2(\pi);A^{\boxtimes 2})
\end{equation}
is zero.
\end{Lem}
\begin{proof}
By Lemma~\ref{lem:H(Conf)} and Poincar\'{e}--Lefschetz duality, the $E^2$-terms of the Leray--Serre spectral sequences for the $\bConf_2(X)$-bundle and the $(\bConf_2(X),\partial \bConf_2(X))$-bundle over $B$ that converge to $H_{2n}$ are isomorphic to $E_{0,2n}^2$ and $E_{n-2,n+2}^2$, respectively, which agree with $H_{2n}$ (Figure~\ref{fig:E2}). Note that the action of $\pi_1B$ on the twisted homology of the fiber of $E\bConf_2(\pi)$ is trivial since the action of $\pi_1B$ factors through $\pi_0\Diff_0(X)$ and it is trivial on the generators obtained in Lemma~\ref{lem:H(Conf)}. Also, the maps induced from (\ref{eq:H2n-restrict}) on the terms $E_{0,2n}^2$ and $E_{n-2,n+2}^2$ are obviously 0, by the naturality of the Leray--Serre spectral sequence (e.g., \cite[Section~1.1]{Hat3}).
\begin{figure}
\[ \includegraphics[height=40mm]{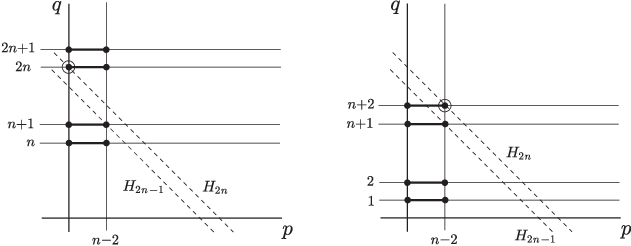} \]
\caption{$E^2_{p,q}$ for $H_*(E\bConf_2(\pi);A^{\boxtimes 2})$ and $H_*(E\bConf_2(\pi),\partial E\bConf_2(\pi);A^{\boxtimes 2})$. Nonzero terms are aligned on the thick lines.}\label{fig:E2}
\end{figure}
\end{proof}

\begin{Def}[Propagator in family]\label{def:propagator-family}
Let $n, X, A, \pi$ be as in Lemma~\ref{lem:E2-EConf}. Let $K$ be an oriented knot in $X$ that generates $H_1(X;\Z)$. Let $s_{\tau}\colon E\to ST^v(E)$ be the section that is given by the first vector of the framing $\tau$ and let $f\colon E\to X$ be a fiberwise pointed degree 1 map which is transversal to $K$. We identify $ST^v(E)$ with the unit normal $S^n$-bundle of the diagonal $\Delta_E\subset E\times_B E$ along the fiber. 
Let $\widetilde{K}=f^{-1}(K)$, which is a framed codimension $n$ submanifold of $E$. 
A {\it propagator in family} is a $2n$-chain $\omega$ of $E\bConf_2(\pi)$ with coefficients in $A^{\boxtimes 2}$ that is transversal to the boundary and that satisfies
\begin{equation}\label{eq:propagator-family}
   \partial_A\,\omega = s_\tau(E)\otimes c_A - \calO_{\tau_0}(X,A)\,ST^v(\widetilde{K})\otimes c_A,
\end{equation}
where for an oriented submanifold $\sigma$ of $E$, we denote by $ST^v(\sigma)$ the restriction of $ST^v(E)$ to $\sigma$. (For the reason of this definition, see Lemma~\ref{lem:ss-basis}).
\end{Def}

\begin{Lem}\label{lem:propagator-family}
Let $n, X, A, \pi$ be as in Lemma~\ref{lem:E2-EConf}. Then there exists a propagator in family. Moreover, two propagators $\omega, \omega'$ in family with $\partial_A\, \omega=\partial_A\, \omega'$ (as chains) are related by
\[ \omega' - \omega = \lambda\,ST(\Sigma^n)\otimes c_A + \partial_A \,\eta \]
for some $\lambda\in R$ and a $(2n+1)$-chain $\eta$ of $E\bConf_2(\pi)$ with coefficients in $A^{\boxtimes 2}$. 
\end{Lem}
\begin{proof}
By the homology exact sequence for pair and by Lemma~\ref{lem:E2-EConf}, we have the following exact sequence.
\[\begin{split}
0&\to H_{2n}(E\bConf_2(\pi),\partial E\bConf_2(\pi);A^{\boxtimes 2})\\
&\stackrel{\widetilde{r}}{\to} H_{2n-1}(\partial E\bConf_2(\pi);A^{\otimes 2})
\stackrel{\widetilde{i}}{\to} H_{2n-1}(E\bConf_2(\pi);A^{\boxtimes 2})
\end{split} \]
Here, it follows from Lemma~\ref{lem:H(Conf)} and $\dim{B}=n-2$ that both 
\begin{equation}\label{eq:H(2n-1)}
 H_{2n-1}(\partial E\bConf_2(\pi);A^{\otimes 2})\quad\mbox{and}\quad H_{2n-1}(E\bConf_2(\pi);A^{\boxtimes 2}) 
\end{equation}
are $E_{n-2,n+1}^\infty=E_{n-2,n+1}^2$ in the Leray--Serre spectral sequence of bundles over $B$. The map induced between the $E^2$-terms is induced by the natural map
\[ 
i\colon H_{n+1}(\partial \bConf_2(X);A^{\otimes 2})
\to H_{n+1}(\bConf_2(X);A^{\boxtimes 2}), \]
by the naturality of the Leray--Serre spectral sequence (e.g., \cite[Section~1.1]{Hat3}). 
Hence $\widetilde{i}$ can be described explicitly by using the result for a single fiber (Corollary~\ref{cor:prop-exists}). Lemma~\ref{lem:ss-basis} below gives an explicit basis of $H_{2n-1}(\partial E\bConf_2(\pi);A^{\otimes 2})$ and the existence of a chain $\omega$ with the boundary (\ref{eq:propagator-family}) follows. 

For two propagators $\omega,\omega'$ in family with common boundary $\partial_A\,\omega=\partial_A\,\omega'$, the chain $\omega'-\omega$ is a $2n$-cycle of $E\bConf_2(\pi)$. By Lemmas~\ref{lem:E2-EConf} and \ref{lem:H(Conf)}, $\omega'-\omega$ is homologous to $\lambda\,ST(\Sigma^n)\otimes c_A$ for some $\lambda\in R$. This completes the proof.
\end{proof}

\begin{Lem}\label{lem:ss-basis}
Let $n, X, A, \pi$ be as in Lemma~\ref{lem:E2-EConf}. Then the $E_{n-2,n+1}^\infty$ in the Leray--Serre spectral sequences for
\[ H_*(\partial E\bConf_2(\pi);A^{\otimes 2})\mbox{ and }H_*(E\bConf_2(\pi);A^{\boxtimes 2}) \]
for the bundle structures over $B$ associated to $\pi$ 
are spanned over $R$ by the cycles $s_\tau(E)\otimes c_A,ST^v(\widetilde{K})\otimes c_A$, and by $ST^v(\widetilde{K})\otimes c_A$, respectively. 
\end{Lem}
\begin{proof}
Since $\partial E\bConf_2(\pi)\cong S^n\times E$ and $E\bConf_2(\pi)$ is homologically equivalent to $(D^{n+1},\partial D^{n+1})[-1]\times E$ (proof is the same as Lemma~\ref{lem:H(Conf)}) with untwisted coefficients in $R$, it suffices to check that $[E]$ and $[\widetilde{K}]$ span $H_{2n-1}(E;R)$ and $H_{n-1}(E;R)$, respectively. The former is obvious since $\dim{E}=2n-1$. 
For the latter, we consider the Leray--Serre spectral sequence for $\pi$ with untwisted coefficient $R$. The $E^2$ page is as in Figure~\ref{fig:E2_2}.
\begin{figure}
\[ \includegraphics[height=30mm]{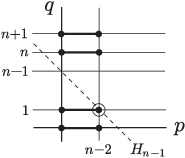} \]
\caption{$E^2_{p,q}$ for $H_*(E; R)$}\label{fig:E2_2}
\end{figure}
From this we have
\[ \begin{split}
  &E^2_{n-2,2}=E^\infty_{n-2,2}=H_{n-1}(E;R),\\
  &E^2_{0,n}=E^\infty_{0,n}=H_n(E;R),
\end{split} \]
and both modules are of rank 1. Since $[\Sigma^n]\in H_n(E;R)$, $[\widetilde{K}]\in H_{n-1}(E;R)$, and $\langle \Sigma^n, \widetilde{K}\rangle =\pm 1$,
it follows that $[\widetilde{K}]$ generates $H_{n-1}(E;R)$.
\end{proof}

%%%%%%%%%%%%%%%%%%%%%%%%%%%%%%%
\subsection{Invariant intersections of chains with local coefficients}

In \cite{Les1}, equivariant intersection form was used for intersections of chains.
On the other hand, on the twisted de Rham complex of a compact manifold, the pairing
\[ (\alpha,\beta)=\int_M \Tr(\alpha\wedge\beta) \]
is often used and it gives Poincar\'{e} duality for the twisted de Rham cohomology (e.g., \cite[\S{5.1}]{Sa2} for flat $SU(2)$-bundles).
In the following, we use an analogue of these pairings on singular chains of $\bConf_2(X)$ with coefficients in $A^{\boxtimes 2}$, which utilizes smoothness of $X$. 

As before, let $n,X,A,\pi$ be as in Lemma~\ref{lem:E2-EConf}. Let $\rho_A\colon \pi\to \mathrm{End}_\C(A)$ be the representation for $A$ as in \S\ref{ss:acyclic-complex}. We assume that the local coefficient system $A^{\boxtimes 2}$ on $\bConf_2(X)$ is defined as in \S\ref{ss:twisted-chain-family} using the intersections of paths with $\partial\bcalA_p(\xi)$ in $X$. 
\begin{Assum}
Suppose that two chains $C_1,C_2$ of $S_*(\bConf_\vee(X))\otimes_{\C[\pi^2]}A^{\boxtimes 2}=S_*(\bConf_2(X))\otimes_\C A^{\boxtimes 2}$ are piecewise strata transversal and satisfy the condition $\dim\,C_1+\dim\,C_2=\dim\,\bConf_2(X)=2n+2$. 
When $C_i=\sum_{k_i}\sigma^{k_i}_i\otimes x_{k_i}\otimes y_{k_i}$, where $\sigma^{k_i}_i\colon \Delta^{p_i}\to\bConf_2(X)$ is a small simplex and $x_{k_i},y_{k_i}\in A$, we assume that whenever $\mathrm{Im}\,\sigma_1^{k_1}\cap \mathrm{Im}\,\sigma_2^{k_2}\neq \emptyset$, the following conditions are satisfied:
\begin{itemize}
\item The intersection is transversal and at a single point $z$.
\item Let $\delta_0^{k_1},\delta_0^{k_2}$ be the barycenters of $\Delta^{p_1},\Delta^{p_2}$, respectively. Then there are paths $\eta_i^{k_i}$ in $\Delta^{p_i}$ from $\delta_0^{k_i}$ to $(\sigma_i^{k_i})^{-1}(z)$ such that
\begin{equation}\label{eq:Hol=1}
 \Hol_{\bConf_2(X)}(\sigma_1^{k_1}\circ\eta_1^{k_1})=\Hol_{\bConf_2(X)}(\sigma_2^{k_2}\circ\eta_2^{k_2})=1. 
\end{equation}
\end{itemize}
We say that a such pair $C_1,C_2$ is in {\it piecewise general position with respect to holonomy}. 
\end{Assum}

As the image of $\partial\bcalA_p(\xi)$ in $X$ is closed, we may assume by taking the simplices sufficiently small and by perturbing the intersections slightly that $C_1,C_2$ are piecewise general position with respect to holonomy. For such a pair $C_1,C_2$, we define $\langle C_1,C_2\rangle\in (A^{\otimes 2})^{\otimes 2}$ by the formula
\[ \langle C_1,C_2\rangle=\sum_{k_1,k_2}\langle \sigma_1^{k_1},\sigma_2^{k_2}\rangle_{\bConf_2(X)}(x_{k_1}\otimes y_{k_1})\otimes (x_{k_2}\otimes y_{k_2}), \]
where $\langle\sigma_1^{k_1},\sigma_2^{k_2}\rangle_{\bConf_2(X)}$ is defined by the sign of the intersection point $\sigma_1^{k_1}\cap \sigma_1^{k_1}$. Recall that the sign of the intersection is determined by comparing the wedge product of the {\it coorientations} of $\sigma_1^{k_1}$ and $\sigma_2^{k_2}$ and the orientation of $\bConf_2(X)$\footnote{This convention might not be usual but this will lead to a convention consistent with \cite{Wa2} given by integrals of forms along cycles (\cite[\S{D.2}]{Wa2}) by the correspondence $\langle M,N\rangle_R\leftrightarrow \int_R\eta_M\wedge\eta_N=\int_N\eta_M$ ($\eta_M$ etc. is a form representative of the Poincar\'{e} dual of $M$). In particular, we do not need to change or check the sign in the formula of Proposition~\ref{prop:normalization} (and also other signs in its proof) then.}. 
This can be generalized to more general pairs $C_1,C_2$ without the condition (\ref{eq:Hol=1}) by defining $\langle C_1,C_2\rangle$ by 
\begin{equation}\label{eq:int-hol}
 \sum_{k_1,k_2}\langle \sigma_1^{k_1},\sigma_2^{k_2}\rangle_{\bConf_2(X)}\Hol(\eta_1^{k_1})(x_{k_1}\otimes y_{k_1})\otimes \Hol(\eta_2^{k_2})(x_{k_2}\otimes y_{k_2}), 
\end{equation}
where $\Hol=\Hol_{\bConf_2(X)}$, which looks slightly complicated.

\begin{enumerate}
\item When $A$ is finite dimensional over $\C$, we define $\Tr\colon (A^{\otimes 2})^{\otimes 2}\to \C$ by
\[ \Tr((x_1\otimes y_1)\otimes (x_2\otimes y_2))=B(x_1,x_2^*)B(y_1,y_2^*), \]
where $x^*$ is the dual of $x$ with respect to $B$, namely, $x=\sum_i a_ie_i$ corresponds to $x^*=\sum_i \overline{a}_i e_i^*$. This is $\rho_A(\pi)$-invariant in the sense of the following identities.
\[ \begin{split}
  \Tr((x_1\otimes \rho_A(g)y_1)\otimes (x_2\otimes \rho_A(g)y_2))&=\Tr((x_1\otimes y_1)\otimes (x_2\otimes y_2))\\
  \Tr((\rho_A(g)x_1\otimes y_1)\otimes (\rho_A(g)x_2\otimes y_2))&=\Tr((x_1\otimes y_1)\otimes (x_2\otimes y_2))\\
\end{split} \]
Namely, invariant under the $\pi\times\pi$-action on $(A^{\boxtimes 2})^{\otimes 2}$. 
Also, $\Tr$ is $\C$-sesquilinear, i.e., for $\alpha,\beta\in \C$,
\[ \Tr(\alpha(x_1\otimes y_1)\otimes \beta(x_2\otimes y_2))
=\alpha\overline{\beta}\,\Tr((x_1\otimes y_1)\otimes(x_2\otimes y_2)). \]
This definition of $\Tr$ corresponds to the map $\End_\C(A)^{\otimes 2}\to \C$ given by $f\otimes h\mapsto \Tr(f\circ h^*)$. 

\item When $A=\Lambda=\C[t^{\pm 1}]$, we define $\Tr\colon \Lambda^{\otimes 2}\to \Lambda$ by
\[ \Tr(t^a\otimes t^b)=t^{a-b}. \]
This is $\Z$-invariant, i.e., $\Tr(t^{a+k}\otimes t^{b+k})=\Tr(t^a\otimes t^b)$.

\item When $A=\mathfrak{g}\otimes_\C \Lambda=\mathfrak{g}[t^{\pm 1}]$, $\Lambda=\C[t^{\pm 1}]$, we define $\Tr\colon (\mathfrak{g}^{\otimes 2}[t^{\pm 1}])^{\otimes 2}\to \Lambda$ by
\[ \Tr((x_1\otimes y_1)t^a\otimes (x_2\otimes y_2)t^b)=B(x_1,x_2^*)B(y_1,y_2^*)\,t^{a-b}. \]
This is $\mathrm{Ad}(G)\times\Z$-invariant and can be considered as obtained from a sesquilinear form on $\mathfrak{g}^{\otimes 2}$ by $\C[t^{\pm 1}]$-sesquilinear extension. 
\end{enumerate}
In each case, $\Tr$ is a Hermitian form, i.e., $\Tr(h\otimes f)=\overline{\Tr(f\otimes h)}$ ($f,h\in A^{\otimes 2}$), where $\overline{\alpha\beta}=\overline{\alpha}\overline{\beta}$ and $\overline{t^n}=t^{-n}$.

\begin{Lem}\label{lem:tr-coboundary}
$\Tr\langle\cdot,\cdot\rangle$ is a coboundary. Namely, when $\dim\,C_1+\dim\,C_2=\dim{\bConf_2(X)}+1=2n+3$, we have
\begin{equation}\label{eq:tr-cboundary}
 \Tr\langle \partial_A C_1,C_2\rangle + (-1)^{2n+2-\dim{C_2}}\Tr\langle C_1,\partial_A C_2\rangle =0. 
\end{equation}
Hence $\Tr\langle\cdot,\cdot\rangle$ induces well-defined sesquilinear pairings (corresponding to the above three cases)
\[ \begin{split}
  &H_p(\bConf_2(X),\partial \bConf_2(X);A^{\boxtimes 2})\otimes_\C H_q(\bConf_2(X);A^{\boxtimes 2})\to \C,\\
  &H_p(\bConf_2(X),\partial \bConf_2(X);\Lambda)\otimes_\Lambda H_q(\bConf_2(X);\Lambda)\to \Lambda,\\
  &H_p(\bConf_2(X),\partial \bConf_2(X);\mathfrak{g}^{\boxtimes 2}[t^{\pm 1}])\otimes_{\Lambda} H_q(\bConf_2(X);\mathfrak{g}^{\boxtimes 2}[t^{\pm 1}])\to \Lambda,
\end{split} \]
where $\Lambda=\C[t^{\pm 1}]$, $p+q=2n+2$.
We call $\Tr\langle\cdot,\cdot\rangle$ an \emph{invariant intersection}.
\end{Lem}
\begin{proof}
If we extend the above definition of $\langle\cdot,\cdot\rangle$ to the case $\dim\,C_1+\dim\,C_2=2n+3$ similarly, then $\langle C_1,C_2\rangle$ for piecewise strata transversal chains $C_i=\sum_{k_i}\sigma^{k_i}_i\otimes x_{k_i}\otimes y_{k_i}$ ($i=1,2$) gives a chain of $S_1(\bConf_2(X))\otimes_\C (A^{\boxtimes 2})^{\otimes 2}$. In this case, we assume that the simplices of $C_1,C_2$ of dimensions $\dim\,C_1,\dim\,C_2$, respectively, intersect transversally in 1-simplices. When the simplices $\sigma_1^{k_1}\colon \Delta^{p_1}\to \bConf_2(X),\sigma_2^{k_2}\colon\Delta^{p_2}\to \bConf_2(X)$ intersect transversally in an embedded 1-simplex $\zeta$, let $z$ be the barycenter of $\zeta$, and we choose paths $\eta_1^{k_1}$ in $\Delta^{p_1}$ from the barycenter $\delta_0^{k_1}$ of $\Delta^{p_1}$ to $(\sigma_1^{k_1})^{-1}(z)$ and $\eta_2^{k_2}$ in $\Delta^{p_2}$ from the barycenter $\delta_0^{k_2}$ of $\Delta^{p_2}$ to $(\sigma_2^{k_2})^{-1}(z)$. Then $\langle C_1,C_2\rangle$ is defined by
\[ \sum_{k_1,k_2}(\sigma_1^{k_1}\cap\sigma_2^{k_2})\otimes\Hol_{\bConf_2(X)}(\sigma_1^{k_1}\circ\eta_1^{k_1})(x_{k_1}\otimes y_{k_1})\otimes \Hol_{\bConf_2(X)}(\sigma_2^{k_2}\circ\eta_2^{k_2})(x_{k_2}\otimes y_{k_2}).
\]

We first prove the following identity
\begin{equation}\label{eq:int-cochain}
 (-1)^{2n+2-\dim{C_2}}\langle \partial_A C_1,C_2\rangle +\langle C_1,\partial_A C_2\rangle=\partial_A\langle C_1,C_2\rangle. 
\end{equation}
By the bilinearity of the intersection pairing, it suffices to prove this identity for $C_1=\sigma_\lambda\otimes m_\lambda$, $C_2=\sigma_\mu\otimes m_\mu$, where $\sigma_\lambda,\sigma_\mu$ are smooth singular simplices in $\bConf_2(X)$, $m_\lambda,m_\mu\in A^{\boxtimes 2}$. To prove it, we suppose the following (see Figure~\ref{fig:sigma-sigma}):
\begin{enumerate}
\item[(a)] $\sigma_\lambda$ and $\sigma_\mu$ are embeddings.

\item[(b)] $\sigma_\lambda$ and $\sigma_\mu$ intersect transversally in a 1-simplex $\zeta$ with $\partial \zeta=\alpha-\beta$. 

\item[(c)] The 0-simplex $\alpha$ is the intersection of $\sigma_\mu$ and a face $\tau_\lambda$ of $\sigma_\lambda$.

\item[(d)] The 0-simplex $\beta$ is the intersection of $\sigma_\lambda$ and a face $\tau_\mu$ of $\sigma_\mu$.
\end{enumerate}
Note that these conditions can be satisfied for generic pairs of simplices which are sufficiently small\footnote{More concrete examples are simplicial chains in a $c$-stratifold (\cite{Kre}) given by a smooth triangulation of $\bConf_2(X)$ into small simplices. We take two different such $c$-stratifold structures and then require stratifold transversality of a $c$-stratifold and each simplex in the other $c$-stratifold.}. 
For a pair of embedding simplices $\sigma,\sigma'$ such that the image of $\sigma'$ is included in that of $\sigma$, we denote by $\eta_{\sigma,\sigma'}$ a path in $\mathrm{Im}\,\sigma$ from the barycenter of $\sigma'$ to that of $\sigma$. Then we have the following.
\[\begin{split}
\partial_A\langle \sigma_\lambda\otimes m_\lambda, \sigma_\mu\otimes m_\mu\rangle
&=\partial_A(\zeta\otimes\Hol(\eta_{\sigma_\lambda,\zeta})m_\lambda\otimes\Hol(\eta_{\sigma_\mu,\zeta})m_\mu)\\
&=\alpha\otimes\Hol(\eta_{\sigma_\lambda,\zeta}\circ\eta_{\zeta,\alpha})m_\lambda\otimes\Hol(\eta_{\sigma_\mu,\zeta}\circ\eta_{\zeta,\alpha})m_\mu\\
&\phantom{=}-\beta\otimes\Hol(\eta_{\sigma_\lambda,\zeta}\circ\eta_{\zeta,\beta})m_\lambda\otimes\Hol(\eta_{\sigma_\mu,\zeta}\circ\eta_{\zeta,\beta})m_\mu,\\
\langle \partial_A(\sigma_\lambda\otimes m_\lambda), \sigma_\mu\otimes m_\mu\rangle
&=\langle \tau_\lambda\otimes\Hol(\eta_{\sigma_\lambda,\tau_\lambda})m_\lambda,\sigma_\mu\otimes m_\mu\rangle\\
&=\pm \alpha\otimes \Hol(\eta_{\sigma_\lambda,\tau_\lambda}\circ\eta_{\tau_\lambda,\alpha})m_\lambda\otimes\Hol(\eta_{\sigma_\mu,\alpha})m_\mu,\\
\langle \sigma_\lambda\otimes m_\lambda, \partial_A(\sigma_\mu\otimes m_\mu)\rangle
&=\langle \sigma_\lambda\otimes m_\lambda, \tau_\mu\otimes\Hol(\eta_{\sigma_\mu,\tau_\lambda})m_\mu,\rangle\\
&=-\beta\otimes \Hol(\eta_{\sigma_\lambda,\beta})m_\lambda\otimes\Hol(\eta_{\sigma_\mu,\tau_\mu}\circ\eta_{\tau_\mu,\beta})m_\mu,
\end{split} \]
where $\pm=-(-1)^{2n+2-\dim{C_2}}$. The signs are determined by the identity
\[ \alpha-\beta=\partial\zeta=\partial\langle\sigma_\lambda,\sigma_\mu\rangle
=(-1)^{2n+2-\dim{C_2}}\langle\partial\sigma_\lambda,\sigma_\mu\rangle
+\langle\sigma_\lambda,\partial\sigma_\mu\rangle. \]
(See Lemma~\ref{lem:d-int-sigma}.) Then the identity (\ref{eq:int-cochain}) follows from the following identities:
\[\begin{split}
&\Hol(\eta_{\sigma_\lambda,\zeta}\circ\eta_{\zeta,\alpha})=\Hol(\eta_{\sigma_\lambda,\tau_\lambda}\circ\eta_{\tau_\lambda,\alpha}),\quad \Hol(\eta_{\sigma_\mu,\zeta}\circ\eta_{\zeta,\alpha})=\Hol(\eta_{\sigma_\mu,\alpha}),\\
&\Hol(\eta_{\sigma_\lambda,\zeta}\circ\eta_{\zeta,\beta})=\Hol(\eta_{\sigma_\lambda,\beta}),\quad \Hol(\eta_{\sigma_\mu,\zeta}\circ\eta_{\zeta,\beta})=\Hol(\eta_{\sigma_\mu,\tau_\mu}\circ\eta_{\tau_\mu,\beta}).
\end{split}\]
\begin{figure}
\[ \includegraphics[height=55mm]{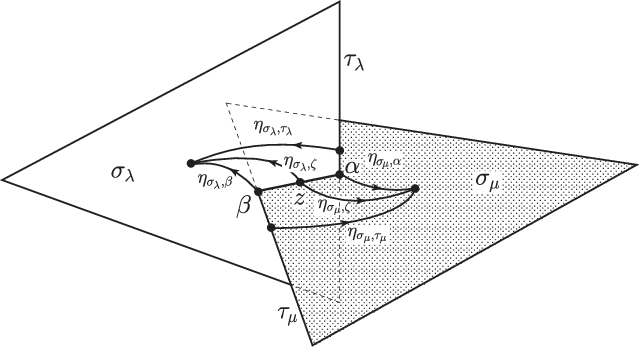} \]
\caption{The simplices $\sigma_\lambda$ and $\sigma_\mu$, intersecting in a 1-simplex $\zeta$.}\label{fig:sigma-sigma}
\end{figure}

As the trace factors through the projection of the coefficient module $(A^{\boxtimes 2})^{\otimes 2}$ over $\bConf_2(X)$ to the $\pi\times \pi$-invariant part, or $H_0(\bConf_2(X);(A^{\boxtimes 2})^{\otimes 2})$, we have
\[ \Tr\,\partial_A\langle C_1,C_2\rangle=\partial \Tr\langle C_1,C_2\rangle, \]
whose count with orientations (or the homology class) is zero. This completes the proof.
\end{proof}

\begin{Rem}
$\Tr\langle\cdot,\cdot\rangle$ is not a new object and would also be interpreted by the evaluation of the trace of the cup product 
\[\begin{split}
& H^{p'}(\bConf_2(X);A^{\boxtimes 2})\otimes_R H^{q'}(\bConf_2(X),\partial \bConf_2(X);A^{\boxtimes 2})\\
&\to H^{2n+2}(\bConf_2(X),\partial\bConf_2(X);R)
\end{split}\] $(p'+q'=2n+2$) of the Poincar\'{e}--Lefschetz duals at the fundamental class $[\bConf_2(X),\partial\bConf_2(X)]$. We will not prove the equivalence between the two interpretations.
\end{Rem}

%%%%%%%%%%%%%%%%%%%%%%%%%%%%%%%
\subsection{Cross product of cycles with local coefficients}

In Section~\ref{s:sformula}, we will need to consider cross products of two $\C$-cycles as $A^{\boxtimes 2}$-cycles. However, there may not be a canonical way to write a $\C$-chain of $X$ as an $A$-chain, unless $A$ is a $\C$-algebra with 1. Namely, there may not be a canonical $\C$-linear map $S_*(X;\C)\to S_*(X;A)$. Instead, we will use the following interpretation later (in Section~\ref{s:sformula}) to construct a possibly twisted dual basis to a basis consisting of cross products of $\C$-cycles. The precise statement we need is Lemma~\ref{lem:cross-product} below. For our purpose, it is enough to restrict to $A=\mathfrak{g}[t^{\pm 1}]$ and $R=\C[t^{\pm 1}]$. 
\begin{Lem}\label{lem:eta}
Suppose that $A$ is a free module over $R=\C$ or $\C[t^{\pm 1}]$ equipped with a finite $R$-basis $\{e_i\}$ and a nondegenerate symmetric $R$-bilinear form $B(\cdot,\cdot)$. For a subspace $V$ of $X$, we define $\C$-linear maps
\[ \begin{split}
& \eta\colon S_*(V;\C) \to \mathrm{Hom}_R(A,S_*(V;A)),\\
& \eta_A\colon S_*(V;A) \to \mathrm{Hom}_R(A,S_*(V;A)),\\
& \overline{\eta}_A\colon S_*(V;A) \to \mathrm{Hom}_R(A,S_*(V;A))
\end{split} \]
by $\eta(\sigma)=(a\mapsto \sigma\otimes a)$, $\eta_A(\gamma)=(e_i\mapsto (1\otimes B(\cdot,e_i^*)\,e_i)(\gamma))$, and $\overline{\eta}_A(\gamma)=(e_i^*\mapsto (1\otimes B(e_i,\cdot)\,e_i^*)(\gamma))$. Then $\eta_A$ and $\overline{\eta}_A$ are chain maps. Also, if the restriction of the local coefficient system $A$ on $V$ is trivial, then $\eta$ is a chain map, too. In that case, by the K\"{u}nneth formula (\cite[Theorem~VI.3.1a]{CE}), the maps between homologies
$\eta_*\colon H_*(V;\C)\to \mathrm{Hom}_R(A,H_*(V;A))$,
$\eta_{A*}\colon H_*(V;A)\to \mathrm{Hom}_R(A,H_*(V;A))$, 
$\overline{\eta}_{A*}\colon H_*(V;A)\to \mathrm{Hom}_R(A,H_*(V;A))$
are induced.
\end{Lem}
\begin{proof}
For $\gamma=\sum_\mu \gamma_\mu \otimes e_\mu\in S_*(\widetilde{X})\otimes_{\C[\pi]}A$, where $\widetilde{X}$ is the universal cover of $X$, we have
\[ \begin{split}
  \eta_A(\partial_A\gamma)&=\bigl(e_i\mapsto \sum_\mu (\partial\gamma_\mu)\otimes B(e_\mu,e_i^*)\,e_i\bigr)\\
  &=\bigl(e_i\mapsto (\partial\otimes \mathbbm{1})\sum_\mu \gamma_\mu\otimes B(e_\mu,e_i^*)\,e_i\bigr)=\partial_A(\eta_A\gamma).
\end{split}\]
This completes the proof that $\eta_A$ is a chain map. The proof for $\overline{\eta}_A$ is the same. When $\gamma$ is a chain of $S_*(V;\C)$ and $A$ is trivial as a local coefficient system on $V$, that $\eta$ is a chain map is proved as follows.
\[ \begin{split}
  \partial_A(\eta(\gamma))&=\partial_A(a\mapsto \gamma\otimes a)=(a\mapsto \partial_A(\gamma\otimes a)=\partial\gamma\otimes a)=\eta(\partial\gamma).
\end{split} \]
\end{proof}
\begin{Rem}
Note that the map $\eta$ may depend on the choice of the identification $S_*(X;A)=S_*(X)\otimes_\C A$. 
Also, the map $\eta_A$ depends on the choice of the $R$-basis $\{e_i\}$ of $A$. 
The definition of $\eta_A$ generalizes that of $\eta$ in the sense that for a $\C$-cycle $\sigma$, 
\[ \eta_A\Bigl(\sigma\otimes \sum_ie_i\Bigr)=(e_j\mapsto \sigma\otimes e_j)=\eta(\sigma). \]
There is no reason that the definition of $\eta_A$ above is the canonical choice. Under the duality, $\eta$ can also be considered as the map $S_*(V;\C)\to S_*(V;A^{\otimes 2})$ induced by the map $\C\to A^{\otimes 2}; 1\mapsto c_A=\sum_i e_i\otimes e_i^*$, and $\eta_A$ can be considered as the map $S_*(V;A)\to S_*(V;A^{\otimes 2})$ induced by the map $A\to A^{\otimes 2}; e_i\mapsto e_i\otimes e_i^*$. 
\end{Rem}

When the restrictions of $A$ on $V,W\subset X$ are both trivial, $\eta$ induces a homomorphism
\[  \eta_*^{\otimes 2}\colon H_*(V;\C)\otimes_\C H_*(W;\C)\to \mathrm{Hom}_R(A^{\otimes 2},H_*(V\times W;A^{\boxtimes 2})). \]
Here, the following identity holds in $S_*(V\times W;A^{\boxtimes 2})$:
\[ \eta^{\otimes 2}(\alpha\times \beta)(c_A)=(a\otimes b\mapsto (\alpha\times\beta)\otimes(a\otimes b))(c_A)=(\alpha\times\beta)\otimes c_A.\]
We denote this element by $\alpha\times_{c_A} \beta$. Then $\eta^{\otimes 2}(\alpha\times \beta)$ can be considered as the mapping $a\mapsto a(\alpha\times_{c_A}\beta)$. Note that the map $\eta_*^{\otimes 2}$ has indeterminacy by multiplication by an invertible element if $A^{\otimes 2}$ due to the choice of the stratification of $X$. If $V=W$, there is no indeterminacy.
When only the restriction of $W$ on $A$ is trivial, we still have
\[ \begin{split}
  &\eta_{A*}\otimes \eta_*\colon H_*(V;A)\otimes_\C H_*(W;\C)\to \mathrm{Hom}_R(A^{\otimes 2},H_*(V\times W;A^{\boxtimes 2})),\\
  &\eta_*\otimes \overline{\eta}_{A*}\colon H_*(W;\C)\otimes_\C H_*(V;A)\to \mathrm{Hom}_R(A^{\otimes 2},H_*(W\times V;A^{\boxtimes 2})).
\end{split}\]
We denote $(\eta_{A}\otimes\eta)(\alpha\times \beta)(c_A)$ or $(\eta\otimes\overline{\eta}_{A})(\alpha\times \beta)(c_A)$ simply by $\alpha\times_{c_A} \beta$ for $\alpha\in S_*(V;A)$, $\beta\in S_*(W;\C)$ etc. More generally,
\[ \eta_{A*}\otimes\overline{\eta}_{A*}\colon H_*(V;A)\otimes_\C H_*(W;A)\to \mathrm{Hom}_R(A^{\otimes 2},H_*(V\times W;A^{\boxtimes 2})) \]
is defined. Again, we denote $(\eta_{A}\otimes\overline{\eta}_{A})(\alpha\times \beta)(c_A)$ simply by $\alpha\times_{c_A} \beta$ for $\alpha\in S_*(V;A)$, $\beta\in S_*(W;A)$ etc. (See (\ref{eq:cross-sigma}) for the cross product of $A$-chains.) Explicitly, if $\alpha=\sum_\lambda \alpha_\lambda\otimes e_\lambda$ and $\beta=\sum_\mu \beta_\mu\otimes e_\mu^*$, then
\[ \alpha\times_{c_A}\beta
=\sum_i (\alpha_i\times\beta_i)\otimes e_i\otimes e_i^*. \]
With this definition of $\alpha\times_{c_A}\beta$, we have the following property, which is desired.
\begin{Lem}\label{lem:cross-product}
Let $A$ be as in Lemma~\ref{lem:eta}.
\begin{enumerate}
\item When the restrictions of $A$ on $V,W\subset X$ are both trivial, let $\alpha,\beta$ be $\C$-chains of $V,W$, respectively. Let $\gamma,\delta$ be $A$-chains of $X$. Then the following identity holds.
\[ \Tr\langle \alpha\times_{c_A}\beta, \gamma\times\delta\rangle
=\Tr\langle \alpha\times_{c_A}\beta, \gamma\times_{c_A}\delta\rangle,\]
where $\gamma\times\delta$ is the cross product of $\gamma$ and $\delta$.
\item Suppose that a submanifold $U$ of $X^{\times 2}$ is such that $H_*(U;A^{\boxtimes 2})$ is freely spanned over the algebra $A^{\otimes 2}$ by cycles of the forms $\alpha_i\times_{c_A} \beta_i$ ($i=1,\ldots,r$, $\alpha_i,\beta_i$ are $\C$-cycles of $X$). Suppose moreover that $\Tr\langle\cdot,\cdot\rangle$ gives a duality between $H_*(U;A^{\boxtimes 2})$ and $H_*(U,\partial U;A^{\boxtimes 2})$, and that the dual $A^{\otimes 2}$-basis of $H_*(U,\partial U;A^{\boxtimes 2})$ to $\{\alpha_i\times_{c_A} \beta_i\}$ is $\gamma_i\times \delta_i$ ($i=1,\ldots,r$, $\gamma_i,\delta_i$ are $A$-cycles of $X$). Then the dual $A^{\otimes 2}$-basis can be represented by the cycles $\gamma_i\times_{c_A} \delta_i$ ($i=1,\ldots,r$). In other words, $\gamma_i\times \delta_i$ is homologous to $\gamma_i\times_{c_A} \delta_i$.
\item Let $\alpha, \beta$ be $A$-chains of $X$. The following identity for $A^{\otimes 2}$-chains holds.
\[ \partial_A(\alpha\times_{c_A}\beta)=(\partial_A \alpha)\times_{c_A}\beta\pm \alpha\times_{c_A}(\partial_A \beta), \]
where the sign depends on the dimension of $\alpha$.
\end{enumerate}
\end{Lem}
\begin{proof}
Let $\gamma=\sum_{\lambda}\gamma_\lambda\otimes e_\lambda$, $\delta=\sum_\mu \delta_\mu\otimes e_\mu^*$, where $\gamma_\lambda$, $\delta_\mu$ are $R$-chains. Then by the sesquilinearity of $\Tr$ we have
\[ \begin{split}
  \Tr\langle \alpha\times_{c_A}\beta, \gamma\times\delta\rangle&=\sum_{\lambda,\mu}\langle \alpha\times\beta,\overline{\gamma}_\lambda\times \overline{\delta}_\mu\rangle \Tr(c_A\otimes (e_\lambda\otimes e_\mu^*))\\
  &=\sum_\lambda \langle \alpha\times\beta,\overline{\gamma}_\lambda\times \overline{\delta}_\lambda\rangle \Tr(c_A\otimes (e_\lambda\otimes e_\lambda^*))\\
  &=\Tr\langle \alpha\times_{c_A}\beta, \gamma\times_{c_A}\delta\rangle,
\end{split}\]
where $\overline{\gamma}_\lambda$ etc. are the conjugates of $\gamma_\lambda$ etc. on $R$.
The assertion 2 follows immediately from 1. 
The assertion 3 holds since $\alpha\times_{c_A}\beta=(\eta_{A}\otimes\overline{\eta}_{A})(\alpha\times\beta)(c_A)$ and $\eta_{A}\otimes\overline{\eta}_{A}$ is a chain map.
\end{proof}
The relation Lemma~\ref{lem:cross-product}-3 can be used to find a chain whose boundary is the right hand side, when $\partial_A\alpha$ and $\partial_A\beta$ are given by $\C$-cycles, for example.

\begin{Exa} We shall consider the case $R=A=\C[t^{\pm 1}]$. When the restriction of $A$ on $V,W\subset X$ are both trivial, let $\alpha,\beta$ be $\C$-chains of $V,W$, respectively, and let $F$ be a chain of $X^{\times 2}$ with coefficients in $A\otimes_R A=\C[t^{\pm 1}]$. In this case, the product $\alpha\times\beta$ can be considered as the mapping $1\mapsto \alpha\times_{c_A}\beta$ in
\[ \mathrm{Hom}_R\bigl(R,S_*(V\times W;R)\bigr). \]
Then the intersection $\Tr\langle F, \alpha\times \beta\rangle$ as an element of $\mathrm{Hom}_R(R,R)=R$ is 
\[ \Tr\langle F,\alpha\times_{c_A}\beta\rangle. \]

If $F=\sum_m F_m t^m$, where $F_m$ is a $\C$-chain of $X$, we have
\begin{equation}\label{eq:eq-linking}
\begin{split}
\Tr\langle F,\alpha\times_{c_A}\beta\rangle&=\sum_m\langle F_m,\alpha\times \beta\rangle_{X\times X}\, t^m\\
&= \sum_m\langle F,t^m(\alpha\times \beta)\rangle_{\widetilde{X}\times_\Z \widetilde{X}}\, t^m,
\end{split}
\end{equation}
where $\widetilde{X}\times_{\Z}\widetilde{X}$ is the quotient of $\widetilde{X}\times\widetilde{X}$ by the diagonal action of $\Z\colon$  $([\gamma_1],[\gamma_2])\mapsto (t^{\pm 1}[\gamma_1],t^{\pm 1}[\gamma_2])$ and the $\Z$-action on $\widetilde{X}$ is that of the restriction of the covering transformation to the subgroup $\{1\}\times \Z\subset \pi=\pi'\times\Z$.
The right hand side of (\ref{eq:eq-linking}) agrees with the definition of the equivariant intersection in \cite{Les1}. 

Both sides of (\ref{eq:eq-linking}) are determined only up to multiplication by invertible elements of $A^{\otimes 2}=R$, and the indeterminacy of the left hand side is due to the choice of the stratification of $X$. The indeterminacy of the right hand side is due to the choice of a lift of each given $\alpha\times\beta$ in the covering $\widetilde{X}\times_\Z \widetilde{X}$. \qed
\end{Exa}

%%%%%%%%%%%%%%%%%%%%%%%%%%%%%%%
\subsection{Linking number}

Let $\alpha,\beta$ be two disjoint null-homotopic embedded spheres of $X$ such that $\dim{\alpha}+\dim{\beta}=n$. We assume that $\alpha$ and $\beta$ are disjoint from the image of the boundary of $\bcalA_p(\xi)$ to define the local coefficient system on $X$. We define the linking number of $\alpha$ and $\beta$ by 
\[ \Lk_A(\alpha,\beta) = \Tr\langle \omega,\alpha\times_{c_A}\beta\rangle\in A^{\otimes 2}, \]
where $\omega$ is a propagator.
\begin{Lem}\label{lem:lk-symmetry}
$\Lk_A(\alpha,\beta)=(-1)^{(\dim\,\alpha+1)(\dim\,\beta+1)}\,\Lk_A(\beta,\alpha)^*$.
\end{Lem}
\begin{proof}
The cycle $\alpha\times_{c_A}\beta$ is homologous to an $A^{\otimes 2}$-linear combination of the cycles $\alpha_i\times_{c_A}\beta_i$ for finitely many small Hopf links $\alpha_i\cup \beta_i$ in $X$. This can be shown as follows. The inclusions of $\alpha$ and $\beta$ extends to smooth maps $D_\alpha\colon D^a\to X$, $D_\beta\colon D^b\to X$ of disks of dimensions $a=\dim{\alpha}+1$, $b=\dim{\beta}+1$, respectively. We take basepoints $u\in \alpha$, $v\in \beta$ and paths $\gamma_\alpha,\gamma_\beta\colon [0,1]\to X$ with $\gamma_\alpha(0)=u$, $\gamma_\alpha(1)=x_0$, $\gamma_\beta(0)=v$, $\gamma_\beta(1)=x_0$ that are disjoint from the image of the boundary of $\bcalA_p(\xi)$. We assume that the intersections $D_\alpha\cap\beta$ and $\alpha\cap D_\beta$ are both transversal. 

Now we deform $D_\alpha$ by precomposing a smooth contraction of the disk radially onto the basepoint $u$ as tight as possible, avoiding the transversal crossing points in $D_\alpha\cap \beta$. Then we also deform $D_\beta$ similarly. The resulting disks consist of finitely many small disks at the crossing points, which are connected to the basepoints $u,v$ of the original disks by thin bands (Figure~\ref{fig:D_alpha}). The result depends on which of $D_\alpha$ and $D_\beta$ is deformed first. By bordisms, the thin bands can be supressed and the result is a disjoint union of small Hopf links $\alpha_i\cup \beta_i$ at the transversal crossing points. The removed thin bands together with the paths $\gamma_\alpha,\gamma_\beta$ give paths $\gamma_{\alpha_i},\gamma_{\beta_i}$ to $x_0$ from basepoints on the components of each small Hopf link. The cycle $\alpha\times_{c_A}\beta$ is homologous to 
\[ \sum_i (\rho_A\mathrm{Hol}_X(\gamma_{\alpha_i})\otimes \rho_A\mathrm{Hol}_X(\gamma_{\beta_i}))(\alpha_i\times_{c_A}\beta_i). \]
Hence it suffices to prove the lemma for each term in this sum.
\begin{figure}[h]
\[ \includegraphics[height=50mm]{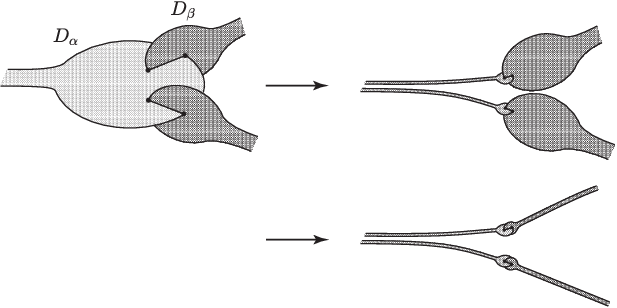} \]
\caption{Deforming $D_\alpha$ and $D_\beta$ into small Hopf links with bands.}\label{fig:D_alpha}
\end{figure}

We may also assume that each small Hopf link $\alpha_i\cup \beta_i$ is included in an $(n+1)$-ball that is disjoint from the boundary of the image of $\bcalA_p(\xi)$. Then we have
\[ \begin{split}
  &\Tr\langle \omega, (\rho_A\mathrm{Hol}_X(\gamma_{\alpha_i})\otimes \rho_A\mathrm{Hol}_X(\gamma_{\beta_i}))(\alpha_i\times_{c_A}\beta_i)\rangle\\
  =& \Tr(\rho_A \mathrm{Hol}_X(\gamma_{\beta_i}^{-1} \gamma_{\alpha_i}))\,\Lk(\alpha_i,\beta_i), 
\end{split}\]
where $\Lk(\alpha_i,\beta_i)$ is the linking number defined in a small ball including $\alpha_i$ and $\beta_i$. Now the result follows by the idenitities:
\[ \begin{split}
&\Tr(\rho_A \mathrm{Hol}_X(\gamma_{\alpha_i}^{-1} \gamma_{\beta_i}))=\Tr(\rho_A \mathrm{Hol}_X(\gamma_{\beta_i}^{-1} \gamma_{\alpha_i}))^*,\\
&\Lk(\beta_i,\alpha_i)=(-1)^{(\dim\,\alpha+1)(\dim\,\beta+1)}\,\Lk(\alpha_i,\beta_i).
\end{split} \]
\end{proof}

%%%%%%%%%%%%%%%%%%%%%%%%%%%%%%%
%%%%%%%%%%%%%%%%%%%%%%%%%%%%%%%
\mysection{Perturbative invariant}{s:invariant}

For compact oriented codimension $n$ submanifolds $F_1,F_2,F_3$ of $E\bConf_2(\pi)$ with corners that intersect strata transversally, we define their oriented intersection by
\begin{equation}\label{eq:Gamma-intersection}
 \langle F_1,F_2,F_3\rangle_\Theta
=F_1\cap F_2\cap F_3,
\end{equation}
which is an oriented codimension $3n$ submanifold with corners (\cite[Appendix]{BT}).
This is defined only when $F_1,F_2,F_3$ are in a general position. This can be extended to generic $\C$-chains of $E\bConf_2(\pi)$ by $\C$-linearity.

When $\dim{B}=n-2$, $E\bConf_2(\pi)$ is $3n$-dimensional and the intersection (\ref{eq:Gamma-intersection}) is 0-dimensional. We say that compact oriented codimension $n$ submanifolds $F_1,F_2,F_3$ of $E\bConf_2(\pi)$ with corners are in {\it general position with respect to holonomy} if the following conditions are satisfied:
\begin{itemize}
\item The intersection is transversal.
\item Let $\delta_0^1,\delta_0^2,\delta_0^3$ be the basepoints of $F_1,F_2,F_3$, respectively. Then for each triple intersection point $z$, there are paths $\eta_i$ in $F_i$ from $\delta_0^i$ to $z$ such that
\begin{equation}\label{eq:Hol=1-triple}
 \Hol_{E\bConf_2(\pi)}(\eta_1)
 =\Hol_{E\bConf_2(\pi)}(\eta_2)
 =\Hol_{E\bConf_2(\pi)}(\eta_3)=1. 
\end{equation}
\end{itemize}
For embedding simplices $\sigma_i\colon \Delta^{2n}\to E\bConf_2(\pi)$ ($i=1,2,3$) that are in general position with respect to holonomy and for $m_i\in A^{\boxtimes 2}$, we define the oriented intersection $\langle \sigma_1\otimes m_1,
\sigma_2\otimes m_2,
\sigma_3\otimes m_3\rangle_\Theta\in (A^{\boxtimes 2})^{\otimes 3}$ by
\begin{equation}\label{eq:theta-intersection}
 \langle \sigma_1, \sigma_2, \sigma_3\rangle_\Theta\,m_1\otimes m_2\otimes m_3. 
\end{equation}
This can be extended by $A^{\boxtimes 2}$-linearity to codimension $n$ chains $F_i$ ($i=1,2,3$) from
\begin{equation}\label{eq:id_S_2n}
 S_{2n}(E\bConf_\vee(\pi))\otimes_{\C[\pi^2]}A^{\boxtimes 2}=S_{2n}(E\bConf_2(\pi))\otimes_\C A^{\boxtimes 2}
\end{equation}
that are in piecewise general position with respect to holonomy. The condition (\ref{eq:Hol=1}) for the genericity can be removed by replacing (\ref{eq:theta-intersection}) with an analogue of (\ref{eq:int-hol}) for the triple intersection. Note that $\langle F_1,F_2,F_3\rangle_\Theta$ thus defined may depend on the identification (\ref{eq:id_S_2n}). 

%%%%%%%%%%%%%%%%%%%%%%%%%%%%%%%
\subsection{Definition of the invariant when $\calO_{\tau_0}(X,A)=0$, $n$ odd}\label{ss:inv-O=0}

Let $n, X=\Sigma^n\times S^1, A, \pi$ be as in Lemma~\ref{lem:E2-EConf}. We take propagators $\omega_1,\omega_2,\omega_3$ in family $E\bConf_2(\pi)$ with coefficients in $A^{\boxtimes 2}$ so that they are parallel on the boundary (Figure~\ref{fig:P-tilde}). Namely, let $s_\tau^{(i)}\colon E\to ST^v(E)\cong E\times S^n$ ($i=1,2,3$) be the sections that are obtained from $s_\tau$ (Definition~\ref{def:propagator-family}) by the actions on $S^n$ by three distinct elements $g^{(i)}\in SO_{n+1}$ ($i=1,2,3$) close to the identity, so that the images are disjoint. Then we assume that 
\begin{equation}\label{eq:bPi}
\partial_A\,\omega_i=s_\tau^{(i)}(E)\otimes c_A,
\end{equation}
which is possible by Example~\ref{ex:O=0}. 
Moreover, we assume that the three propagators are piecewise strata transversal. Each $\omega_i$ can be considered as an element of $S_{2n}(E\bConf_2(\pi);\C)\otimes_{\C}A^{\boxtimes 2}$ by the identification (\ref{eq:id_S_2n}).
\begin{figure}
\[ \includegraphics[height=35mm]{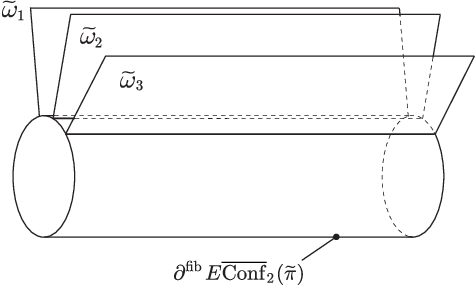} \]
\caption{$\widetilde{\omega}_i$ are parallel near the boundary.}\label{fig:P-tilde}
\end{figure}

\begin{Def}
When $n$ is odd, we define the invariant $Z_\Theta^\even$ by 
\[ Z_\Theta^\even(\omega_1,\omega_2,\omega_3)=\frac{1}{6}\Tr_\Theta\langle \omega_1,\omega_2,\omega_3\rangle_\Theta\in\calA_\Theta^\even(A^{\boxtimes 2};\rho_A(\pi)\times \Z), \]
where $\Tr_\Theta\colon (A^{\boxtimes 2})^{\otimes 3}\to \calA_\Theta^\even(A^{\boxtimes 2};\rho_A(\pi')\times \Z)$ is the projection (see Definition~\ref{def:A} for the definition of $\calA_\Theta^\even$). 
\end{Def}

\begin{Lem} The class $Z_\Theta^\even(\omega_1,\omega_2,\omega_3)$ does not depend on the choice of the identification (\ref{eq:id_S_2n}).
\end{Lem}
Thanks to the operator $\Tr_\Theta$, this follows for a similar reason as in Lemma~\ref{lem:tr-coboundary}. 
Since we fixed canonical lifts of the propagators by the boundary condition (\ref{eq:bPi}), there is no indeterminacy due to the choice of the stratification of $X$.

\begin{Thm}\label{thm:Z-inv-O=0}
Let $n, X=\Sigma^n\times S^1, A, \pi=\pi'\times\Z$ be as in Lemma~\ref{lem:E2-EConf}. Then $Z_\Theta^\even(\omega_1,\omega_2,\omega_3)$ does not depend on the choice of $\omega_i$, and gives a homomorphism
\[ Z_\Theta^\even\colon \Omega_{n-2}^{SO}(\widetilde{B\Diff}_{\mathrm{deg}}(X^\bullet,\partial))\to \calA_\Theta^\even(A^{\otimes 2};\rho_A(\pi')\times\Z). \]
\end{Thm}
\begin{proof}
Let $(\pi^+\colon E^+\to B^+,\tau_+)$, $(\pi^-\colon E^-\to B^-,\tau_-)$ be framed $(X,U_{x_0})$-bundles over closed oriented manifolds $B^+$, $B^-$ of dimension $n-2$ with fiberwise pointed homotopy equivalences $f^\pm\colon E^\pm \to X$ that are bundle bordant as framed $(X,U_{x_0})$-bundles with fiberwise pointed homotopy equivalences. Namely, there is a triple $(\widetilde{\pi},\widetilde{\tau},\widetilde{f})$ consisting of the following:
\begin{itemize}
\item A framed $(X,U_{x_0})$-bundle $\widetilde{\pi}\colon \widetilde{E}\to \widetilde{B}$ over a connected compact oriented cobordism $\widetilde{B}$ with $\partial \widetilde{B}=B^+\tcoprod (-B^-)$ that restricts to $\pi^+,\pi^-$ on the ends $\partial \widetilde{B}$.

\item A vertical framing $\widetilde{\tau}$ of $\widetilde{E}$ extending $\tau_{\pm}$ on the ends $\partial \widetilde{B}$. 

\item A smooth map $\widetilde{f}\colon \widetilde{E}\to X$ that restricts to $f^+,f^-$ on the ends $\partial \widetilde{B}$.
\end{itemize}
We assume that $\widetilde{f}$ and $\widetilde{\tau}$ satisfy the standardness condition near the basepoint $x_0$ as in Section~\ref{ss:moduli-sp}. The local coefficient systems $A$ on $E^{\pm}$ extends over $\widetilde{E}$ by the pullback by $\widetilde{f}$.

We take triples of propagators $(\omega_1^+,\omega_2^+,\omega_3^+)$ and $(\omega_1^-,\omega_2^-,\omega_3^-)$ that define $Z_\Theta^\even$ on the ends $B^+$ and $B^-$, respectively. We first assume that there is a triple of propagators $(\widetilde{\omega}_1,\widetilde{\omega}_2,\widetilde{\omega}_3)$ in family $E\bConf_2(\widetilde{\pi})$ that extends those on $\partial\widetilde{B}$, extended to the case when the base is not closed (cf. Lemma~\ref{lem:E2-EConf} and Definition~\ref{def:propagator-family} for the closed case). Namely, $\widetilde{\omega}_i$ is a twisted chain of $S_*(E\bConf_2(\widetilde{\pi});\C)\otimes_{\C}A^{\otimes 2}$, satisfying the following identity:
\begin{equation}\label{eq:bPtilde}
 \begin{split}
  &\partial_A\, \widetilde{\omega}_i=\omega_i^+-\omega_i^-+s_{\widetilde{\tau}}^{(i)}(\widetilde{E})\otimes c_A,
\end{split} 
\end{equation}
where $s_{\widetilde{\tau}}^{(i)}$ is a small perturbation of $s_{\widetilde{\tau}}\colon \widetilde{E}\to ST^v(\widetilde{E})$ obtained by applying $SO_{n+1}$-rotations $g^{(i)}$ as before in (\ref{eq:bPi}) so that the images are disjoint.

The intersection $\langle \widetilde{\omega}_1,\widetilde{\omega}_2,\widetilde{\omega}_3\rangle_\Theta$ gives a chain of $S_1(E\bConf_2(\widetilde{\pi}))\otimes_{\C}(A^{\boxtimes 2})^{\otimes 3}$ if $\widetilde{\omega}_i$ are chosen to be piecewise general position. Then the boundary of $\Tr_\Theta\langle \widetilde{\omega}_1,\widetilde{\omega}_2,\widetilde{\omega}_3\rangle_\Theta$ corresponds to the triple intersection of the chains in 
\begin{equation}\label{eq:bECtilde}
 \partial E\bConf_2(\widetilde{\pi})=E\bConf_2(\pi^+)-E\bConf_2(\pi^-)+\partial^\fib E\bConf_2(\widetilde{\pi}). 
\end{equation}
Indeed, as in the proof of Lemma~\ref{lem:tr-coboundary}, we have
\[ 
  \partial\,\Tr_\Theta\langle\widetilde{\omega}_1,\widetilde{\omega}_2,\widetilde{\omega}_3\rangle_\Theta=\Tr_\Theta\,\partial_A\langle\widetilde{\omega}_1,\widetilde{\omega}_2,\widetilde{\omega}_3\rangle_\Theta
 \]
and by (\ref{eq:bPtilde}) all the intersections that survive are in the right hand side of (\ref{eq:bECtilde}).
Since $\widetilde{E}$ has a vertical framing and the propagators $\widetilde{\omega}_1,\widetilde{\omega}_2,\widetilde{\omega}_3$ are parallel on $\partial^\fib E\bConf_2(\widetilde{\pi})$, they do not have intersection there. Hence the boundary only survives in $E\bConf_2(\pi^+)-E\bConf_2(\pi^-)$ and we have
\[\Tr_\Theta\langle \omega_1^+,\omega_2^+,\omega_3^+\rangle_\Theta
-\Tr_\Theta\langle \omega_1^-,\omega_2^-,\omega_3^-\rangle_\Theta=0.
\]
This completes the proof under the assumption of the existence of the triple $(\widetilde{P}_1,\widetilde{P}_2,\widetilde{P}_3)$.

In Lemma~\ref{lem:bordism-propagators} below, we shall see that a propagator $\widetilde{\omega}_i$ as assumed above exists if we replace $\omega_i^-$ with $\omega_i^-+ \lambda\,ST(\Sigma^n)'\otimes c_A$ for some $\lambda\in R$, where $ST(\Sigma^n)'$ is a parallel copy of $ST(\Sigma^n)$ (see Lemma~\ref{lem:H(Conf)} for the definition of $ST(\cdot)$) obtained by pushing slightly toward the inward normal vector field on the boundary of the base fiber $\bConf_2(X)$. Assuming this, it remains to show that the value of $\Tr_\Theta\langle \omega_1^-,\omega_2^-,\omega_3^-\rangle_\Theta$ does not change by additions of $\lambda\,ST(\Sigma^n)'\otimes c_A$. Namely, we prove the identity
\[\begin{split}
 &\Tr_\Theta\langle \omega_1^-+\lambda\,ST(\Sigma^n)'\otimes c_A,\omega_2^-+\lambda\,ST(\Sigma^n)''\otimes c_A,\omega_3^-+\lambda\,ST(\Sigma^n)'''\otimes c_A\rangle_\Theta\\
&=\Tr_\Theta\langle \omega_1^-,\omega_2^-,\omega_3^-\rangle_\Theta, 
\end{split} \]
where $ST(\Sigma^n)''$ and $ST(\Sigma^n)'''$ are parallel copies of $ST(\Sigma^n)'$ in the interior of $\bConf_2(X)$. Since $ST(\Sigma^n)'$, $ST(\Sigma^n)''$, $ST(\Sigma^n)'''$ are mutually disjoint, we need only to check that the terms of triple intersections among two propagators $\omega_i^-,\omega_j^-$ and one copy of $ST(\Sigma^n)$ vanish. By the boundary transversality assumption for the propagators, the values of these are the same as the intersections with $ST(\Sigma^n)$ on the boundary. By the explicit form (\ref{eq:bPi}) of the boundary of $\omega_i^-$, we see that the intersection among $\omega_i^-,\omega_j^-$ and $ST(\Sigma^n)$ is empty. This completes the proof that we may change $\omega_i^-$ as above without changing the value of $\Tr_\Theta\langle \omega_1^-,\omega_2^-,\omega_3^-\rangle_\Theta$. Note that the replacement of $\omega_i^-$ does not affect the argument in the previous paragraph since the boundary of $\omega_i^-$ is not changed, and there is no change in the triple near the boundary of $E\bConf_2(\pi^-)$. 

Now that $Z_\Theta^\even$ is a homomorphism follows by considering the disjoint union of two bundles for two elements of $\Omega_{n-2}^{SO}(\widetilde{B\Diff}_{\mathrm{deg}}(X^\bullet,\partial))$. The additivity of the intersection for the disjoint union is obvious. 
\end{proof}

\begin{Lem}[A bordism analogue of Lemma~\ref{lem:propagator-family}]\label{lem:bordism-propagators}
Let $\omega_i^\pm$, $s_{\widetilde{\tau}}^{(i)}(\widetilde{E})$ be as in the proof of Theorem~\ref{thm:Z-inv-O=0}. Then for some $\lambda\in R$, there exists a $(2n+1)$-chain $\widetilde{P}_i$ of $E\bConf_2(\widetilde{\pi})$ such that 
\[ \partial_A \widetilde{P}_i=\omega_i^+-(\omega_i^-+\lambda\,ST(\Sigma^n)'\otimes c_A)+s_{\widetilde{\tau}}^{(i)}(\widetilde{E})\otimes c_A. \]
\end{Lem}
\begin{proof}
Since $\widetilde{B}$ is connected and does not have a closed component, its $\Z$ homology is isomorphic to that of an $(n-2)$-dimensional subcomplex and the computation of $H_*(E\bConf_2(\widetilde{\pi});A^{\boxtimes 2})$ is similar to that of Lemma~\ref{lem:E2-EConf} and is isomorphic to $E_{n-1,n+2}^2$. The $2n$-cycle $\omega_i^+-\omega_i^-+s_{\widetilde{\tau}}^{(i)}(\widetilde{E})\otimes c_A$ of $E\bConf_2(\widetilde{\pi})$ represents a class in 
\[ H_{2n}(E\bConf_2(\widetilde{\pi});A^{\boxtimes 2})=R[ST(\Sigma^n)\otimes c_A], \]
where the identification is due to Lemmas~\ref{lem:E2-EConf} and \ref{lem:H(Conf)}, and we consider $\Sigma^n$ is in the base fiber of $E^-$. Thus for some $\lambda\in R$, the $2n$-cycle $\omega_i^+-\omega_i^-+s_{\widetilde{\tau}}(\widetilde{E})\otimes c_A$ is homologous to $-\lambda ST(\Sigma^n)\otimes c_A$. This completes the proof.
\end{proof}

%%%%%%%%%%%%%%%%%%%%%%%%%%%%%%%
%%%%%%%%%%%%%%%%%%%%%%%%%%%%%%%
\mysection{Evaluation of the invariant}{s:sformula}

%%%%%%%%%%%%%%%%%%%%%%%%%%%%%%%
\subsection{$R$-decorated $\Theta$-graphs}

Let $R$ be an algebra over $\C$ having 1, and let $H$ be a subset of the set of units of $R$. We assume that $R$ has a $\C$-linear involution $R\to R$, which maps $p$ to its ``adjoint'' $p^*$, such that $p^*=p^{-1}$ for $p\in H$. We call a pair $(\Theta,\phi)$ of the following objects an {\it $R$-decorated $\Theta$-graph}.
\begin{enumerate}
\item $\Theta$: abstract, labeled, edge-oriented graph with two vertices connected by three edges, where a label is the pair of bijections $\alpha\colon \{1,2,3\}\to \mathrm{Edges}(\Theta)$ and $\beta\colon \{1,2\}\to \mathrm{Vertices}(\Theta)$, where $\mathrm{Edges}(\Theta)$ and $\mathrm{Vertices}(\Theta)$ are the sets of edges and vertices of $\Theta$, respectively.
\item $\phi\colon \mathrm{Edges}(\Theta)\to R$: a map.
\end{enumerate}
We also write an $R$-decorated graph $(\Theta,\phi)$ as $\Theta(x_1,x_2,x_3)$ ($x_i=\phi(\alpha(i))$).

%%%%%%%%%%%%%%%%%%%%%%%%%%%%%%%
\subsection{$\Theta$-graph surgery}

We take an embedding $\Theta\to X$ of a labeled, edge-oriented trivalent graph $\Theta$. The homotopy class of this embedding can be represented as $\Theta(g_1,g_2,g_3)$ ($g_i\in\pi$) for the $\C[\pi]$-decoration given by the $\pi$-valued holonomy for each edge. We consider that this expression also represents an embedded graph. 

We put an $(n+1)$-dimensional Hopf link of oriented spheres $S^{n-1}$ and $S^1$ at the middle of each edge, as in Figure~\ref{fig:theta-to-Y-link}. Here the spheres are oriented so that their linking number is $+1$. 
Then the two vertices of $\Theta$ give two disjoint Y-shaped components {\it Y-graphs} of Type I and II.
\[ \includegraphics[height=30mm]{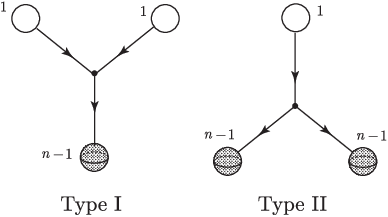} \]
We take small closed tubular neighborhoods of these objects and denote them by $V_1$ and $V_2$. They form a disjoint union of handlebodies embedded in $X$. A Type I Y-graph gives a handlebody (of Type I) which is diffeomorphic to the handlebody obtained from an $(n+1)$-ball by attaching two 1-handles and one $(n-1)$-handle in a standard way, namely, along unknotted unlinked standard attaching spheres in the boundary of the disk. A Type II Y-graph gives a handlebody (of Type II) which is diffeomorphic to the handlebody obtained from an $(n+1)$-ball by attaching one 1-handle and two $(n-1)$-handles in a standard way.
\begin{figure}
\begin{center}
\includegraphics{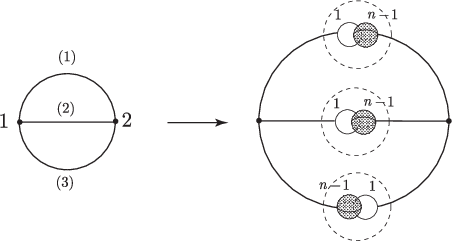}
\end{center}
\caption{Decomposition of embedded $\Theta$-graph into Y-shaped pieces.}\label{fig:theta-to-Y-link}
\end{figure}

Let $V=V_1$ be the Type I handlebody and let $\alpha_{\mathrm{I}}\colon S^0\to \Diff(\partial V)$, $S^0=\{-1,1\}$, be the map defined by $\alpha_{\mathrm{I}}(-1)=\mathbbm{1}$, and by setting $\alpha_{\mathrm{I}}(1)$ as the ``Borromean twist" corresponding to the Borromean string link $D^{n-1}\cup D^{n-1}\cup D^1\to D^{n+1}$. The detailed definition of $\alpha_{\mathrm{I}}$ can be found in \cite[\S{3.7}]{Wa2}.

Let $V=V_2$ be the Type II handlebody and let $\alpha_{\mathrm{II}}\colon S^{n-2}\to \Diff(\partial V)$ be the map defined by comparing the relative isomorphism class of the family of complements of an $S^{n-2}$-family of embeddings $D^{n-1}\cup D^1\cup D^1\to D^{n+1}$ obtained by parametrizing the second component in the Borromean string link $D^{n-1}\cup D^{n-1}\cup D^1\to D^{n+1}$ with the trivial family of the first and third components. The detailed definition of $\alpha_{\mathrm{II}}$ can be found in \cite[\S{3.8}]{Wa2}.

Let $B_\Theta=K_1\times K_2$ ($K_1=S^0$, $K_2=S^{n-2}$). Accordingly, let $\alpha_i\colon K_i\to \Diff(\partial V_i)$ be $\alpha_{\mathrm{I}}$ or $\alpha_{\mathrm{II}}$. By using the families of twists above, we define 
\[ E^\Theta=(B_\Theta\times (X-\mathrm{Int}\,(V_1\cup V_2)))\cup_\partial (B_\Theta\times (V_1\cup V_2)), \]
where the gluing map is given by
\[ \begin{split}
&\psi\colon B_\Theta\times (\partial V_1\cup\partial V_2)\to B_\Theta\times (\partial V_1\cup\partial V_2)\\
&\psi(t_1,t_2,x)=(t_1,t_2,\alpha_i(t_i)(x)) \quad \mbox{(for $x\in \partial V_i$)}.
\end{split}\]
Let $\pi_{V_i}\colon \widetilde{V}_i\to K_i$ be the $V_i$-bundle with the structure group $\Diff(V_i,\partial)$, which is obtained from $K_i\times V_i$ by identifying its boundary with $K_i\times \partial V_i$ by the $K_i$-family of diffeomorphisms $\alpha_i$:
\[ (\widetilde{V}_i,\partial \widetilde{V}_i)=((K_i\times V_i)\cup_{\alpha_i}(K_i\times \partial V_i),K_i\times \partial V_i). \]

\begin{Prop}[{Proposition~\ref{prop:degree1}}]\label{prop:v-framing}
The natural projection $\pi^\Theta\colon E^\Theta\to B_\Theta$ is an $X$-bundle, and it admits a vertical framing that is compatible with the surgery, and it gives an element of 
\[ \Omega_{n-2}^{SO}(\widetilde{B\Diff}_{\mathrm{deg}}(X^\bullet,\partial)).\]
We denote this element by
$\Psi_1(\Theta(g_1,g_2,g_3))$.
\end{Prop}

\begin{Thm}\label{thm:surgery}
Let $n+1\geq 4$ be an even integer and let $X,A$ be as in Lemma~\ref{lem:H(Conf)}. Let $\Theta(g_1,g_2,g_3)$ be a $\C[\pi]$-decorated $\Theta$-graph. 
Then the element
$ [\Psi_1(\Theta(g_1,g_2,g_3))]\in \Omega_{n-2}^{SO}(\widetilde{B\Diff}_{\mathrm{deg}}(X^\bullet,\partial))$
belongs to the image from $\pi_{n-2}\widetilde{B\Diff}_{\mathrm{deg}}(X^\bullet,\partial)$, and the following identity holds when $n$ is odd.
\[ Z_\Theta^\even(\Psi_1(\Theta(g_1,g_2,g_3)))=2\,[\Theta(\rho_A(g_1),\rho_A(g_2),\rho_A(g_3)] \]
\end{Thm}

Combining Theorem~\ref{thm:surgery} with Proposition~\ref{prop:A-incl}, we obtain the following.
\begin{Cor}\label{cor:image-Z}
Let $n+1\geq 4$ be an even integer and let $X,A$ be as in Lemma~\ref{lem:H(Conf)}. Then the image of $Z_\Theta^\even$ in $\calA_\Theta^\even(\mathfrak{g}^{\otimes 2};\rho(\pi')\times\Z)$ includes the free abelian subgroup generated by the infinite set $\{\Theta(1,t,t^p)\mid p\geq 3\}$.
\end{Cor}

\begin{Cor}\label{cor:high-dim}
Let $n\geq 7$ be an integer of the form $4k-1$ and let $\Sigma^n=S^n/\pi'$, where $\pi'=\pi_1\Sigma(2,3,5)$. Then $\pi_{n-2}B\Diff_0(\Sigma^n\times S^1)$ includes the free abelian subgroup generated by the infinite set $\{\Psi_1(\Theta(1,t,t^p))\mid p\geq 3\}$.
\end{Cor}
\begin{proof}
By Remark~\ref{rem:Sigma-delta-trivial}, we can apply Theorem~\ref{thm:surgery} to $X=\Sigma^n\times S^1$.
According to Corollary~\ref{cor:abel-infinite}-2 and Proposition~\ref{prop:puncture}-2, we know by Theorem~\ref{thm:surgery} and Corollary~\ref{cor:image-Z} that $\pi_{n-2}B\Diff_0(\Sigma^n\times S^1)$ is of countable infinite rank. We can say more about the nontrivial subgroup. Namely, we know by (\ref{eq:pi-F-4k-1}), $\pi_{4k-3}\Omega(SO_{4k}\times S^1)=\pi_{4k-2}SO_{4k}$, and from the proofs of Corollary~\ref{cor:abel-infinite}-2 and Proposition~\ref{prop:puncture}-2 that the composition
\[ \pi_{n-2}\widetilde{B\Diff}_{\mathrm{deg}}(X^\bullet,\partial)\otimes\Q
\to \pi_{n-2}B\Diff_0(X^\bullet,\partial)\otimes\Q
\to \pi_{n-2}B\Diff_0(X)\otimes\Q \]
of the natural maps is injective. This proves that $\{\Psi_1(\Theta(1,t,t^p))\mid p\geq 3\}$ is linearly independent in $\pi_{n-2}B\Diff_0(X)\otimes\Q$.
\end{proof}

%%%%%%%%%%%%%%%%%%%%%%%%%%%%%%%
\subsection{Proof of Theorem~\ref{thm:surgery}}\label{ss:proof-surgery}

The outline of the proof is similar to \cite{KT}. Namely, proof of Theorem~\ref{thm:surgery} boils down to one lemma (Proposition~\ref{prop:normalization}), which guarantees the existence of conveniently normalized propagator. Proof of Proposition~\ref{prop:normalization} is almost the same as that of \cite[Proposition~11.1]{Les1} and mostly done in \cite[\S{6}--{7}]{Wa2}.

%%%%%%%%%%%%%%%%%%%%%%%%%
\subsubsection{A decomposition of $E^\Theta$}

Let $\Theta$ be the $\Theta$-graph. Let
$V_\infty = X-\mathrm{Int}(V_1\cup V_2)$,
$B_\Theta=K_1\times K_2$, and
\[ \begin{split}
  & \widetilde{V}_\lambda'=\left\{\begin{array}{ll}
  \widetilde{V}_1\times K_2 & (\lambda=1),\\
  K_1\times\widetilde{V}_2 & (\lambda=2),\\
  B_\Theta\times V_\infty & (\lambda=\infty).
\end{array}\right. 
\end{split} \]
This is a bundle over $B_\Theta$, which is canonically isomorphic to the pullback of the bundle $\pi(\alpha_\lambda)\colon \widetilde{V}_\lambda\to K_\lambda$ (for $\lambda=1,2$) or $V_\infty\to \{t^0\}$ by the projection $B_\Theta\to K_\lambda$ or $B_\Theta\to \{t^0\}$. Then we have
\[ E^\Theta=\widetilde{V}_1'\cup \widetilde{V}_2'\cup \widetilde{V}_\infty',
\]
where the boundaries are glued by the natural trivializations $\partial\widetilde{V}_\lambda'=B_\Theta\times\partial V_\lambda$ ($\lambda=1,2$) and $\partial\widetilde{V}_\infty'=B_\Theta\times (\partial V_1\cup \partial V_2)$. 

%%%%%%%%%%%%%%%%%%%%%%%%%
\subsubsection{A decomposition of $E\bConf_2(\pi^\Theta)$}

For $i,j\in\{1,2\}$ such that $i\neq j$, let 
\[ \Omega_{ij}=\widetilde{V}_i'\times_{B_\Theta} \widetilde{V}_j', \]
which is canonically diffeomorphic to $\widetilde{V}_i\times \widetilde{V}_j$.
For $i\in\{1,2,\infty\}$, let
\[ \Omega_{ii}=p_{B\ell}^{-1}(\widetilde{V}_i'\times_{B_\Theta} \widetilde{V}_i'),\quad
\Omega_{i\infty}=p_{B\ell}^{-1}(\widetilde{V}_i'\times_{B_\Theta} \widetilde{V}_\infty'),\quad
\Omega_{\infty i}=p_{B\ell}^{-1}(\widetilde{V}_\infty'\times_{B_\Theta} \widetilde{V}_i'),
\] 
where $p_{B\ell}\colon E\bConf_2(\pi^\Theta)\to E^\Theta\times_{B_\Theta} E^\Theta$ is the fiberwise blow-down map. We have
\[ E\bConf_2(\pi^\Theta)=\bigcup_{i,j}\Omega_{ij}. \]

%%%%%%%%%%%%%%%%%%%%%%%%%
\subsubsection{Some cycles of $V_i$ and $\widetilde{V}_i$}

\begin{enumerate}
\item Let $b_1^i, b_2^i, b_3^i$ be the cycles in $\partial V_i$ that are parallel to the cores of the three handles of positive indices with the orientations determined by those of the Hopf links in the Y-graph decomposition of embedded $\Theta$-graph. If $V_i$ is of Type I, two of the cycles $b_j^i$ are circles and one of the cycles $b_j^i$ is a $(n-1)$-dimensional sphere. If $V_i$ is of Type II, one of the cycles $b_j^i$ is a circle and two of the cycles $b_j^i$ are $(n-1)$-dimensional spheres. Let $a_1^i,a_2^i,a_3^i$ be dual spheres of $\partial V_i$ to $b_1^i,b_2^i,b_3^i$ with respect to the intersection in $\partial V_i$. They are the boundaries of the cocores of the three handles in $V_i$ of positive indices. We take a basepoint $p^i$ of $\partial V_i$ that is disjoint from the cycles $b_j^i,a_j^i$. We orient $a_j^i$ by the condition
\[ \Lk(b_j^{i-},a_j^i)=+1, \]
where $b_j^{i-}$ is a parallel copy of $b_j^i\subset \partial V_i$ in $\mathrm{Int}\,V_i$ obtained by shifting slightly.

\item Let $S(a_\ell^i)$ be a disk in $V_i$ that is bounded by $a_\ell^i$. Let $S(b_\ell^i)$ be a disk in $X-\mathrm{Int}\,V_i$ that is bounded by $b_\ell^i$. Let $\gamma^i$ be a 1-chain of $V_\infty$ with twisted coefficients in $A$ that is bounded by $p^i$, which exists by $H_0(V_\infty;A)=H_0(X;A)=0$. We orient $S(a_\ell^i)$ and $S(b_\ell^i)$ by the outward-normal-first convention from the orientations of the boundaries $a_\ell^i$ and $b_\ell^i$, respectively.

\item $S(b_\ell^i)$ may intersect a handle of $V_j$ ($j\neq i$) transversally. We assume that the intersection agrees with $S(a_m^j)$ for some unique $(m,j)$ up to orientation. This is possible according to the special linking property of the handlebodies $V_1,V_2$. 

For $i\neq\infty$, we identify a small tubular neighborhood of $\partial V_i$ in $X$ with $[-4,4]\times\partial V_i$ so that $\{0\}\times\partial V_i=\partial V_i$ and $\{-4\}\times\partial V_i\subset \mathrm{Int}\,V_i$. For a cycle $x$ of $\partial V_i$ represented by a manifold, let 
\[ x[h]=\{h\}\times x\subset [-4,4]\times\partial V_i \]
 and let $x^+$ denote a parallel copy of $x$ obtained by slightly shifting $x$ along positive direction in the coordinate $[-4,4]$. Here, $[-4,4]\times \partial V_i$ is a subset of a single fiber $X$. 
Also, let
\[ \begin{split}
  V_i[h]&=\left\{
  \begin{array}{ll}
    V_i\cup ([0,h]\times\partial V_i) & (h\geq 0),\\
    V_i-((h,0]\times\partial V_i) & (h<0),
  \end{array}\right.\\
  S_h(b_\ell^i)&=\left\{
  \begin{array}{ll}
    S(b_\ell^i)\cap (X-\mathrm{Int}({V}_i[h])) & (h\geq 0),\\
    S(b_\ell^i)\cup ([h,0]\times b_\ell^i) & (h<0),
  \end{array}
  \right.\\
  S_h(a_\ell^i)&=\left\{\begin{array}{ll}
  S(a_\ell^i)\cup ([0,h]\times a_\ell^i) & (h> 0),\\
  S(a_\ell^i)\cap V_i[h] & (h\leq  0),\\
  \end{array}\right.\\
  V_\infty[h]&=
    X-\mathrm{Int}\,(V_1[-h]\cup \cdots\cup V_{2k}[-h]).
\end{split}\]

\item The boundary of $\widetilde{V}_i$ ($i\neq\infty$) is $K_i\times \partial V_i$. The factor $K_i$ has nothing to do with the $[-4,4]$ in the previous item.
Let 
\[ \widetilde{b}_\ell^i=K_i\times b_\ell^i\quad\mbox{ and }\quad\widetilde{a}_\ell^i=K_i\times a_\ell^i,  \]
and orient them by\footnote{Please see \cite[\S{4.2}]{Wa2} for the motivation of this convention.}
\[ o(\widetilde{b}_\ell^i)=(-1)^{n-2}o(S^{n-2})\wedge o(b_\ell^i),\quad o(\widetilde{a}_\ell^i)=(-1)^{n-2}o(S^{n-2})\wedge o(a_\ell^i). \]
The cycle $\widetilde{a}_\ell^i$ bounds a submanifold $S(\widetilde{a}_\ell^i)$ of $\widetilde{V}_i$, which corresponds to a Seifert surface of one component in the Borromean link, and can be chosen so that its normal bundle is trivial (\cite[Lemma~4.2]{Wa2}). We orient $S(\widetilde{a}_\ell^i)$ by the outward-normal-first convention from $o(\widetilde{a}_\ell^i)$. 
 We assume without loss of generality that the intersection of $S(\widetilde{a}_\ell^i)$ with $[-4,4]\times \partial\widetilde{V}_i=K_i\times ([-4,4]\times \partial V_i)$ (in $\widetilde{V}_i$) agrees with $[-4,4]\times \widetilde{a}_\ell^i$. 

\end{enumerate}

%%%%%%%%%%%%%%%%%%%%%%%%%
\subsubsection{A normalization of propagator}

Let $\Omega$ be a subset of $E\bConf_2(\pi^\Theta)$. We say that a singular $p$-chain $C'$ of $\Omega$ is the {\it restriction} of a singular $p$-chain $C$ of $E\bConf_2(\pi^\Theta)$ if 
\begin{enumerate}
\item $C$ does not have a term of a $p$-simplex whose interior intersects $\Omega$ and at the same time whose image is not included in $\Omega$, 

\item the set $\Lambda_{C'}$ of $p$-simplices of the terms of $C'$ agrees with the set  $\Lambda_C$ of the $p$-simplices of the terms of $C$ whose interiors intersect $\Omega$, and

\item $C'$ is obtained from $C$ by ignoring the terms of $p$-simplices that are not in $\Lambda_C$.
\end{enumerate}
In this case, we denote $C'$ by $C|_\Omega$. We also say so if $C$ and $C'$ will satisfy the above condition after some subdivisions. 
Let $\Omega_{ij}(\{\ell\})$ be the restriction of the bundle $\Omega_{ij}\to B_\Theta$ on $K_\ell\subset B_\Theta$, where we identify $K_1\times \{t_2^0\}$ with $K_1$ and $\{t_1^0\}\times K_2$ with $K_2$. Let $\Omega_{ij}(\emptyset)$ be the restriction of the bundle $\Omega_{ij}\to B_\Theta$ on the basepoint $(t_1^0,t_2^0)$.

\begin{Prop}[Normalization of propagator]\label{prop:normalization}
There exists a propagator $\omega$ in family $E\bConf_2(\pi^\Theta)$ with coefficients in $A^{\boxtimes 2}$ that satisfies the following conditions.
\begin{enumerate}
\item $\omega|_{\Omega_{\infty\infty}}=B_\Theta\times \omega|_{\Omega_{\infty\infty}(\emptyset)}$.

\item $\omega|_{\Omega_{1\infty}}=\omega|_{\Omega_{1\infty}(\{1\})}\times K_2$, $\omega|_{\Omega_{2\infty}}=K_1\times \omega|_{\Omega_{2\infty}(\{2\})}$.

\item $\omega|_{\Omega_{\infty 1}}=\omega|_{\Omega_{\infty 1}(\{1\})}\times K_2$, $\omega|_{\Omega_{\infty 2}}=K_1\times \omega|_{\Omega_{\infty 2}(\{2\})}$.

\item $\omega|_{\Omega_{11}}=\omega|_{\Omega_{11}(\{1\})}\times K_2$, $\omega|_{\Omega_{22}}=K_1\times \omega|_{\Omega_{22}(\{2\})}$.

\item For $\{i,j\}=\{1,2\}$, 
\[ \omega|_{\Omega_{ij}}=\sum_{\ell,m}L_{\ell m}^{ij} S(\widetilde{a}_\ell^i)\times_{c_A} S(\widetilde{a}_m^j), \]
where $L_{\ell m}^{ij}=(-1)^n\Lk_A(b_\ell^i,b_m^j)^*$ and the sum is over $\ell,m$ such that $\dim{a_\ell^i}+\dim{a_m^j}=n$.
\end{enumerate}
Here, the direct products $\times$ in 1--4 are the cross products $S_*(K;\C)\otimes_\C S_*(\Omega_{ij}(J);A^{\boxtimes 2})\to S_*(\Omega_{ij};A^{\boxtimes 2})$ or $S_*(\Omega_{ij}(J);A^{\boxtimes 2})\otimes_\C S_*(K;\C)\to S_*(\Omega_{ij};A^{\boxtimes 2})$ for some $J\subset\{1,2,\infty\}$ and $K\subset B_\Theta$ such that $K\times \Omega_{ij}(J)\subset \Omega_{ij}$ or $\Omega_{ij}(J)\times K\subset \Omega_{ij}$. Furthermore, we assume that $B_\Theta\times \omega|_{\Omega_{\infty\infty}(\emptyset)}$ etc. is cooriented by the pullback of the coorientation of $\omega|_{\Omega_{\infty\infty}(\emptyset)}$ in $\Omega_{\infty\infty}(\emptyset)$ etc. $S(\widetilde{a}_\ell^i)\times_{c_A} S(\widetilde{a}_m^j)$ is cooriented by the wedge product of coorientations of $S(\widetilde{a}_\ell^i)\subset \widetilde{V}_i$, $S(\widetilde{a}_m^j)\subset \widetilde{V}_j$.
\end{Prop}
Proof of Proposition~\ref{prop:normalization} is postponed to \S\ref{ss:proof-normalization}.

\begin{proof}[Proof of Theorem~\ref{thm:surgery}]
We choose propagators $\omega_1,\omega_2,\omega_3$ as in Proposition~\ref{prop:normalization}.
On the parts 1--4 in Proposition~\ref{prop:normalization}, the triple intersection does not contribute to $\Tr_\Theta\langle \omega_1,\omega_2,\omega_3\rangle_\Theta$ since the triple intersection on each part is the direct product of that in strictly lower dimensional subspace and positive dimensional space, or the same value is counted on each point of $S^0$ with opposite orientation and is cancelled. Thus it follows that $\Tr_\Theta\langle \omega_1,\omega_2,\omega_3\rangle_\Theta$ agrees with that on the following subspace of $E\bConf_2(\pi^\Theta)$.
\[ \Omega_{12}\tcoprod \Omega_{21}=(\widetilde{V}_1\times \widetilde{V}_2)\tcoprod (\widetilde{V}_2\times \widetilde{V}_1) \]
The restriction to one component $\widetilde{V}_i\times \widetilde{V}_j$ corresponds to counting configurations such that the boundary vertices $v_1,v_2$ of each oriented edge $e=(v_1,v_2)$ of $\Theta$ is mapped to $\widetilde{V}_i$ and $\widetilde{V}_j$, respectively. 

Let us compute the value of $\Tr_\Theta\langle \omega_1,\omega_2,\omega_3\rangle_\Theta$ on the component $\widetilde{V}_1\times \widetilde{V}_2$. 
According to Proposition~\ref{prop:normalization}, the restriction of the propagator for the edge $e$ to $\widetilde{V}_1\times \widetilde{V}_2$ is of the form
\[ L_1^*\,S(\widetilde{a}_1^1)\times_{c_A}S(\widetilde{a}_1^2)
+L_2^*\,S(\widetilde{a}_2^1)\times_{c_A}S(\widetilde{a}_2^2)
+L_3^*\,S(\widetilde{a}_3^1)\times_{c_A}S(\widetilde{a}_3^2). \]
($L_i\in A^{\otimes 2}$) We abbreviate this as $L_1^*S_1+L_2^*S_2+L_3^*S_3$. Since the normal bundle of $S(\widetilde{a}_\ell^i)$ is trivial, it has no self-intersection. Moreover, we may assume that $S_1,S_2,S_3$ are in general position with respect to holonomy (\S\ref{s:invariant}). 
Hence the value of the intersection $\Tr_\Theta\langle \omega_1,\omega_2,\omega_3\rangle_\Theta$ on $\widetilde{V}_1\times \widetilde{V}_2$ is 
\[ \sum_{\sigma\in\mathfrak{S}_3}\langle S_{\sigma(1)},S_{\sigma(2)},S_{\sigma(3)}\rangle_{\widetilde{V}_1\times \widetilde{V}_2}
\Tr_\Theta(L_{\sigma(1)}^*\otimes L_{\sigma(2)}^*\otimes L_{\sigma(3)}^*). \]
When $n$ is odd, the terms of the sum are all equal due to the $\mathfrak{S}_3$-antisymmetry of the triple intersection form and the $\mathfrak{S}_3$-antisymmetry of $\Tr_\Theta$. Then the sum is 
\[ 6\,\langle S_1,S_2,S_3\rangle_{\widetilde{V}_1\times \widetilde{V}_2}\Tr_\Theta(L_1^*\otimes L_2^*\otimes L_3^*)=6[\Theta(\rho_A(g_1),\rho_A(g_2),\rho_A(g_3)]. \]

We may consider similarly for the value of $\Tr_\Theta\langle \omega_1,\omega_2,\omega_3\rangle_\Theta$ on the component $\widetilde{V}_2\times \widetilde{V}_1$. By Lemma~\ref{lem:lk-symmetry}, it can be written for some sign $\ve=\pm 1$ as
\[ \begin{split}
&6\,\langle S_1^*,S_2^*,S_3^*\rangle_{\widetilde{V}_2\times \widetilde{V}_1}\Tr_\Theta(L_1\otimes L_2\otimes L_3)\\
=&\,6\ve\,\langle S_1,S_2,S_3\rangle_{\widetilde{V}_1\times \widetilde{V}_2}\Tr_\Theta(L_1\otimes L_2\otimes L_3)
=6\ve[\Theta(\rho_A(g_1),\rho_A(g_2),\rho_A(g_3)],
\end{split} \]
where $((\sigma\times\tau)\otimes e_i\otimes e_j^*)^*=(\tau\times \sigma)\otimes e_j\otimes e_i^*$. 
Here, under the canonical identification $\widetilde{V}_1\times \widetilde{V}_2 \cong \widetilde{V}_2\times \widetilde{V}_1$, the coorientation of $S(\widetilde{a}_i^2)\times S(\widetilde{a}_i^1)\subset \widetilde{V}_2\times \widetilde{V}_1$ differs from that of $S(\widetilde{a}_i^1)\times S(\widetilde{a}_i^2)\subset \widetilde{V}_1\times \widetilde{V}_2$ by $(-1)^{n-1}$. Also, the orientation of $\widetilde{V}_2\times \widetilde{V}_1$ and of $\widetilde{V}_1\times \widetilde{V}_2$ induced from that of $E\bConf_2(\pi)$ differs by $(-1)^{(n+1)\{(n+1)+(n-2)\}}$. Hence we have
\[ \ve= \{(-1)^{n-1}\}^3\times (-1)^{(n+1)\{(n+1)+(n-2)\}} =1. \]
(Similar computations of the effects of the signs of the intersections under the graph automorphisms are done in \cite{Wa2} for general trivalent graphs.)

Now we obtain the following.
\[\begin{split}
Z_\Theta^\even(\Psi_1(\Gamma(g_1,g_2,g_3)))
&=\frac{1}{6}\cdot 12\,[\Theta(\rho_A(g_1),\rho_A(g_2),\rho_A(g_3)]\\
&=2\,[\Theta(\rho_A(g_1),\rho_A(g_2),\rho_A(g_3)].
\end{split}\]
This completes the proof.
\end{proof}

%%%%%%%%%%%%%%%%%%%%%%%%%%%%%%%
\subsection{Proof of Proposition~\ref{prop:normalization}}\label{ss:proof-normalization}

Proposition~\ref{prop:normalization} can be proved in exactly the same way as \cite[Proposition~4.5]{Wa2}, except the following points, for $V=V_j$:

\begin{enumerate}
\item The integral $\int_C\omega$ is replaced by the invariant intersection $\Tr\langle \omega,C\rangle$. 

\item The computations of $H_*(X-V)$ and $H_*(V\times (X-\mathrm{Int}\,V[3]))$ etc. in \cite[Lemma~6.2]{Wa2} are replaced by Lemmas~\ref{lem:H(X-V)}, \ref{lem:H(W)}, and \ref{lem:H,d} below.

\item In \cite[\S{6}]{Wa2}, a propagator $\omega_0$ on $\bConf_2(X)$ is modified as 
\[ \omega_1=\omega_0+d(\chi\mu)\]
for a form $\mu$ defined on a codimension 0 compact submanifold $U$ of $\bConf_2(X)$ with a collar neighborhood $\partial U\times [0,1)$, and a smooth function $\chi\colon U\to [0,1]$ such that $\chi=0$ near $\partial U\times\{0\}=\partial U$ and $\chi=1$ on a neighborhood of $U_1:=U-(\partial U\times[0,1))$, so that $\omega_1|_{U_1}=\omega_0+d\eta$ is a certain explicit form on $U_1$. For a chain propagator $\omega_0$ on $\bConf_2(X)$ (of codimension $n$) in this paper, we take a codimension $n-1$ chain $\eta$ supported on a small closed $\ve$-neighborhood $N$ of $U_1$ in $U$ and the restriction $\omega_0[\ve]$ of $\omega_0$ to $(\bConf_2(X)-U)\cup (\partial U\times [0,\ve))$ for a small $\ve>0$, which is disjoint from $N$. Then we take an extension $\omega_1$ of $\omega_0[\ve]+\partial_A\eta$ over $\bConf_2(X)$, which agrees with $\omega_0$ on a neighborhood of $\bConf_2(X)-U$ and with $\omega_0+\partial_A\eta$ on a neighborhood of $U_1$.

\item The pullback of a form $\omega$ on $\{t_0^i\}\times W$ to $K_i\times W$ as in \cite[\S{6.4}]{Wa2}, where $t_0^i\in K_i$ is the basepoint and $W$ is a subset of $\bConf_2(X)$, is replaced by the cross product $K_i\times \omega$.

\item The restriction of the partially normalized propagator $\omega_{4,i}'$ to $\partial E\bConf_2(\widetilde{V})$, analogous to that of \cite[\S{7}]{Wa2}, is of the form $\omega''\otimes c_A$ for some $\C$-chain $\omega''$, since the local coefficient system $A$ is trivial on $\widetilde{V}$ by construction of $V$ from a trivalent graph (a Y-graph is included in a ball), and its evaluation over a small Hopf link gives the value $1$. 
\end{enumerate}
For other parts, the proof of Proposition~\ref{prop:normalization} is exactly the same as \cite[Proposition~4.5]{Wa2} and we do not repeat the same argument. To state the lemmas announced in the item 2, we put $V=V_j$ and abbreviate $a_i^j,b_\ell^j,\gamma^j$ etc. as $a_i,b_\ell,\gamma$ etc. for simplicity. 

%%%%%%%%%%%%%%%%%%%%%%%%%%%%
\begin{Lem}\label{lem:H(X-V)}
$H_i(X-V;A)=H_{i+1}(V,\partial V;\C)\otimes_\C A$ for $i>0$. Namely, for $i> 0$,
\[ H_i(X-V;A)=\Bigl[\C[a_1]\oplus \C[a_2]\oplus \C[a_3]\oplus \C[\partial V]\Bigr]_i\otimes_\C A. \]
\end{Lem}
\begin{proof}
In the homology long exact sequence for the pair $(X,X-V)$, we have $H_*(X;A)=0$. Also, by excision, we have $H_{i+1}(X,X-V;A)=H_{i+1}(V,\partial V;A)=H_{i+1}(V,\partial V;\C)\otimes_\C A$. The result follows. 
\end{proof}

The following is obtained by the K\"unneth formula.
\begin{Lem}\label{lem:H(W)}
Let $W=V\times (V-\mathring{V}[3])$, where $\mathring{V}=\mathrm{Int}\,V$. Then we have the following for $i>0$, as $\pi\times \pi$-modules:
\[ \begin{split}
	&H_i(W;A^{\boxtimes 2})=\Bigl[ H_*(V)\otimes_\C (\C[a_1]\oplus \C[a_2]\oplus \C[a_3]\oplus \C[\partial V])\Bigr]_i\otimes_\C A^{\boxtimes 2}, \\
	&H_i(\partial V\times(X-\mathring{V}[3]);A^{\boxtimes 2})\\
  &\hspace{19mm}=\Bigl[H_*(\partial V)\otimes_\C (\C[a_1]\oplus \C[a_2]\oplus \C[a_3]\oplus \C[\partial V])\Bigr]_i\otimes_\C A^{\boxtimes 2}.\\
\end{split}\]
\end{Lem}

\begin{Lem}\label{lem:H,d}
The $\pi\times \pi$-module $H_{n+2}(W,\partial W;A^{\boxtimes 2})$ is generated by the following elements over the algebra $A^{\boxtimes 2}$.
\[ \begin{split}
  &[S(a_i)\times_{c_A} S_3(b_\ell)]\quad(\dim{a_i}+\dim{b_\ell}=n),\quad [V\times_{c_A} \gamma[3]].
\end{split}\]
Hence $\Tr\langle\cdot,\cdot\rangle$ gives a nondegenerate pairing
\[ H_n(W;A^{\boxtimes 2})\otimes_R H_{n+2}(W,\partial W;A^{\boxtimes 2})\to R.\]
\end{Lem}
\begin{proof}
The proof of this lemma is parallel to that of \cite[Lemma~11.5]{Les2}.
We only prove the lemma for Type I since the proof for the case of Type II is similar. We use the homology long exact sequence for the pair $(W,\partial W)$. We know that 
\[ H_{n+2}(W;A^{\boxtimes 2})=\left\{\begin{array}{ll}
\C[b_3\times \partial V[3]]\otimes_\C A^{\boxtimes 2} & (n+1=4),\\
0 & (n+1\geq 5),
\end{array}\right. \]
whose image in $H_{n+2}(W,\partial W;A^{\boxtimes 2})$ is 0. Hence we have
\[ H_{n+2}(W,\partial W;A^{\boxtimes 2})\cong \mathrm{Ker}\,\Bigl[ H_{n+1}(\partial W;A^{\boxtimes 2})\to H_{n+1}(W;A^{\boxtimes 2}) \Bigr]. \]

To determine $H_{n+1}(\partial W;A^{\boxtimes 2})$, we apply the Mayer--Vietoris exact sequence for $\partial W=(V\times \partial V[3])\underset{\partial V\times \partial V[3]}{\cup}(\partial V\times (X-\mathring{V}[3]))$:
\[\begin{split}
	&\to H_{n+1}(\partial V\times \partial V[3])\stackrel{i_{n+1}}{\to} H_{n+1}(V\times \partial V[3])\oplus H_{n+1}(\partial V\times (X-\mathring{V}[3]))\to H_{n+1}(\partial W)\\
	&\stackrel{\partial_{\mathrm{MV}}}{\to} H_n(\partial V\times \partial V[3])\stackrel{i_n}{\to} H_n(V\times \partial V[3])\oplus H_n(\partial V\times (X-\mathring{V}[3]))
\end{split} \]
(coefficients are in $A^{\boxtimes 2}$). $\mathrm{Coker}\,i_{n+1}$ is isomorphic to
\[ H_{n+1}(W;A^{\boxtimes 2})=\Bigl[\C\{[*],[b_1],[b_2],[b_3]\}\otimes \C\{[a_1],[a_2],[a_3],[\partial V]\}[3]\Bigr]_{n+1}\otimes_\C A^{\boxtimes 2}. \]
$\mathrm{Ker}\,i_n$ is isomorphic to 
\[ \Bigl[\C\{[a_1],[a_2],[a_3],[\partial V]\}\otimes \C\{[*],[b_1],[b_2],[b_3]\} \Bigr]_n\otimes_\C A^{\boxtimes 2}, \]
which is generated over $A^{\boxtimes 2}$ by the images of 
$A_{i\ell}=\partial_A (S(a_i)\times_{c_A} S_3(b_\ell))$ ($\dim{a_i}+\dim{b_\ell}=n$) and $A_{00}=\partial_A (V\times_{c_A} \gamma[3])$
under the Mayer--Vietoris boundary map $\partial_{\mathrm{MV}}$. Thus we have
\[ H_{n+1}(\partial W;A^{\boxtimes 2})=H_{n+1}(W;A^{\boxtimes 2})\oplus (\C[A_{00}]\oplus\bigoplus_{i,\ell}\C[A_{i\ell}])\otimes_\C A^{\boxtimes 2}.
\]
Then the result follows. 
\end{proof}

%%%%%%%%%%%%%%%%%%%%%%%%%%%%%%%
%%%%%%%%%%%%%%%%%%%%%%%%%%%%%%%
\mysection{$\Sigma^3\times S^1$-bundles supported on $\Sigma^3\times I$}{s:support-I}

\begin{Prop}\label{prop:support-I}
Let $\Sigma^3=\Sigma(2,3,5)$. The image of the composition of the natural map
\[ \widetilde{i}_*\colon H_1(\widetilde{B\Diff}_{\mathrm{deg}}(\Sigma^3\times I,\partial))\to H_1(\widetilde{B\Diff}_{\mathrm{deg}}((\Sigma^3\times S^1)^\bullet,\partial))\]
and 
$ Z_\Theta^\even\colon H_1(\widetilde{B\Diff}_{\mathrm{deg}}((\Sigma^3\times S^1)^\bullet,\partial))\to \calA_\Theta^\even(\mathfrak{g}^{\otimes 2}[t^{\pm 1}];\rho(\pi')\times \Z)$
is included in the (injective) image from $\calA_\Theta^\even(\mathfrak{g}^{\otimes 2};\rho(\pi'))$.
\end{Prop}
\begin{proof}
Let $J=S^1-\mathrm{Int}\,I$ and $B=S^1$. Let $\pi\colon E\to B$ be a framed $\Sigma^3\times S^1$-bundle that has support in $\Sigma^3\times I$. Namely, we assume that $E$ can be obtained by gluing the trivial framed $(\Sigma^3\times J,\partial)$-bundle $B\times(\Sigma^3\times J)\to B$ and some framed $(\Sigma^3\times I,\partial)$-bundle $\pi_I\colon E_I\to B$ together along the boundaries. Moreover, we assume that a fiberwise degree 1 map $E\to \Sigma^3\times S^1$ is given so that its restriction to $B\times(\Sigma^3\times J)$ agrees with the projection $B\times(\Sigma^3\times J)\to \Sigma^3\times J$. The class of such a framed bundle generates the image of $\widetilde{i}_*$.

The product $\omega_{\Sigma^3}\times J$ is a propagator in $\bConf_2(\Sigma^3\times J)=B\ell((\Sigma^3\times J)^{\times 2},\Delta_{\Sigma^3\times J})$, where $\omega_{\Sigma^3}$ is the $(n+1)$-chain of $\bConf_2(\Sigma^3)$ considered in the proof of Proposition~\ref{prop:O=0}. By extending $B\times(\omega_{\Sigma^3}\times \partial J)$ by $s_\tau(E)$ for the vertical framing $\tau$ on $E$, we obtain a cycle in $\partial E\bConf_2(\pi_I)=B\times \partial \bConf_2(\Sigma^3\times I)$. By Lemmas~\ref{lem:H(Conf(W))} and \ref{lem:EConf(pi-I)} below, there is no homological obstruction to extending this cycle to a propagator in family $E\bConf_2(\pi_I)$. The sum of this extension and $B\times(\omega_{\Sigma^3}\times J)$ gives a propagator $\widetilde{\omega}_I$ in family $E\bConf_2(\pi)$. 

We take $\widetilde{\omega}_I$ as above and its perturbed copies $\widetilde{\omega}_I'$ and $\widetilde{\omega}_I''$ which are parallel near the boundary and each of which has similar product structure as $\widetilde{\omega}_I$ on $\Sigma^3\times J$. We consider well-definedness of the triple intersection $\langle \widetilde{\omega}_I,\widetilde{\omega}_I',\widetilde{\omega}_I''\rangle_\Theta$ in $E\bConf_2(\pi)$, which is expected to give $6Z_\Theta^\even$ for $\pi$. Here, the chains $\widetilde{\omega}_I,\widetilde{\omega}_I',\widetilde{\omega}_I''$ are not piecewise strata transversal unless they are perturbed slightly further on $\Sigma^3\times J$. 
Namely, if we merely took fiberwise copies of $\widetilde{\omega}_I$ from perturbed copies $\omega_{\Sigma^3}',\omega_{\Sigma^3}''$ of $\omega_{\Sigma^3}$ in $\bConf_2(\Sigma^3)$ by the same product structure on $\Sigma^3\times J$, then they would have finitely many triple intersection points in each level $\Sigma^3\times \{z\}$, which results in a 1-dimensional locus of triple points. 

We perturb the product structure of $\Sigma^3\times J$ to define $\widetilde{\omega}_I'$, as follows. First, we take the triple intersection of the three propagators defined as above by the same product structure of $\Sigma^3\times J$. Then in each level $\Sigma^3\times\{z\}$, $z\in J$, there may be finitely many configurations of two points corresponding to the triple intersection of the propagators in $\bConf_2(\Sigma^3)$, and the triple intersection in $\bConf_2(\Sigma^3\times J)$ is a product of a finite set $S\subset\bConf_2(\Sigma^3)$ and $J$. We now see that we may assume after small perturbations that such configurations in $S$ are mutually disjoint in each level $\Sigma^3\times\{z\}$. Let $U_\delta^+$ be the union of the (disjoint) balls of small radius $\delta$ around the points in the image of the first projection of $S$ in $\Sigma^3$, and let $U_\delta^-$ be defined similarly for the second projection of $S$ in $\Sigma^3$. By assumption, we have $U_\delta^+\cap U_\delta^-=\emptyset$. Then we perturb the level set $\Sigma^3\times\{z\}$ in $\Sigma^3\times S^1$ into the following embedding $\Sigma^3\times\{z\}\to \Sigma^3\times S^1$:
\[ (x,z)\mapsto \left\{\begin{array}{ll}
(x,z) & (x\notin U_\delta^+\cup U_\delta^-)\\
(x,z+\lambda(x)) & (x\in U_\delta^+)\\
(x,z-\lambda(x)) & (x\in U_\delta^-)
\end{array}\right. \]
where $\lambda\colon \Sigma^3\to [0,\ve]$ (for $\ve>0$ small) is a bump function supported on $U_\delta^+\cup U_\delta^-$ that takes the value $\ve$ at the center of each ball. Let $\Sigma_\lambda^3$ be the image of the perturbed embedding $\Sigma^3\times\{z_0\}\to \Sigma^3\times S^1$ for a fixed $z_0\in J$ (see Figure~\ref{fig:U-delta}). After doing the same perturbation of $\Sigma^3\times\{z\}$ for all $z\in J$, we obtain a slightly perturbed cylinder $\Sigma^3_\lambda\times J$ in $\Sigma^3\times S^1$. We define $\widetilde{\omega}'_I$ by the product of $\omega_{\Sigma^3}'$ in $\bConf_2(\Sigma_\lambda^3)$ with $J$ by using the product structure of this perturbed cylinder $\Sigma^3_\lambda\times J$. 
\begin{figure}
\[ \includegraphics[height=20mm]{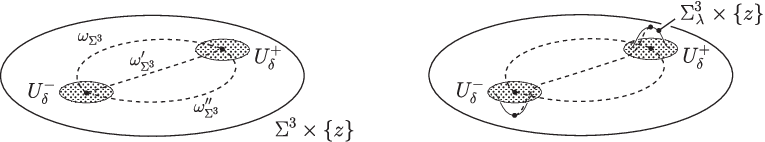} \]
\caption{A perturbation of $\Sigma^3\times J$ to define $\widetilde{\omega}_I'$.}\label{fig:U-delta}
\end{figure}

We claim that there is no triple intersection in $\bConf_2(\Sigma^3\times J)$ among the propagators thus obtained. Indeed, since the perturbation in the previous paragraph does not change the projection $\Sigma^3\times J\to \Sigma^3$, the triple intersection in $\bConf_2(\Sigma^3\times J)$ of the perturbed propagators must be in the preimage of the triple intersection in $\bConf_2(\Sigma^3)$ again, since the locus of the triple intersection is projected to that in $\bConf_2(\Sigma^3)$. By the choices of $\widetilde{\omega}_I$ and $\widetilde{\omega}_I'$ on $\bConf_2(\Sigma^3\times J)$, we see that the triple intersection in $\bConf_2(\Sigma^3\times J)$ is empty. Namely, a two-point configuration $(x_1,x_2)$ in the triple intersection in $\bConf_2(\Sigma^3\times J)$ have to satisfy the following conditions:
\begin{itemize}
\item $(x_1,x_2)$ lies in the preimage of $U_\delta^+\times U_\delta^-$ by the projection.
\item Both $x_1$ and $x_2$ lie in the same level set $\Sigma^3\times \{z\}$ for some $z\in J$.
\item Both $x_1$ and $x_2$ lie in the same level set $\Sigma^3_\lambda\times \{z'\}$ for some $z'\in J$.
\end{itemize}
However, these conditions cannot be satisfied simultanously. 

Since each configuration in the image of $\widetilde{\omega}_I$ is such that two points are both in $\Sigma^3\times I$ or both in $\Sigma^3\times J$, the triple intersection in question is the sum of those on $\Sigma^3\times I$ and $\Sigma^3\times J$. Namely, we do not need to consider configurations of two points such that one point is on $\Sigma^3\times I$ and the other point is on $\Sigma^3\times J$. 

On $\Sigma^3\times J$, namely, on $B\times \bConf_2(\Sigma^3\times J)$, we have seen above that the triple intersection is empty. 
On $\Sigma^3\times I$, namely, on $E\bConf_2(\pi_I)$, we may assume that the $\Z$-component of the holonomies in $\pi'\times \Z$ is trivial, we may take the triple intersection so that no term has power of $t$ in the coefficient. Hence its class in $\calA_\Theta^\even(\mathfrak{g}^{\otimes 2}[t^{\pm 1}];\rho(\pi')\times \Z)$ is included in $\calA_\Theta^\even(\mathfrak{g}^{\otimes 2};\rho(\pi'))$. This completes the proof.
\end{proof}

\begin{Rem}
It follows from Proposition~\ref{prop:support-I} that the invariant
\[ Z_\Theta^\even\circ\widetilde{i}_*\colon 
H_1(\widetilde{B\Diff}_{\mathrm{deg}}(\Sigma^3\times I,\partial))\to 
\calA_\Theta^\even(\mathfrak{g}^{\otimes 2};\rho(\pi')) \]
is defined. Unfortunately, a simple calculation of the character of the $\pi'\times \pi'$-module $\tbigwedge^3(\mathfrak{g}\boxtimes \mathfrak{g})$ shows that the target space $\calA_\Theta^\even(\mathfrak{g}^{\otimes 2};\rho(\pi'))$ is zero. Nevertheless, this can be upgraded by replacing $\mathfrak{g}\boxtimes \mathfrak{g}$ with $\bigoplus_i (V_i\boxtimes V_i)$ for some irreducible $\pi'$-modules $V_i$ (with $H^1(\pi';V_i)=0$). Then the $\calA_\Theta^\even$-space becomes a nonzero vector space. We will write about it in a subsequent paper. 
\end{Rem}

\begin{Lem}\label{lem:H(Conf(W))}
For $A=\mathfrak{g}[t^{\pm 1}]$ and $R=\C[t^{\pm 1}]$, we have
\[ H_i(\bConf_2(\Sigma^3\times I);A^{\boxtimes 2})
\cong\left\{\begin{array}{ll}
R[ST(*)\otimes c_A] & (i=3),\\
R[ST(\Sigma^3)\otimes c_A] & (i=6),\\
0 & (\mbox{otherwise}).
\end{array}\right. \]
\end{Lem}
\begin{proof}
The proof is parallel to that of Lemma~\ref{lem:H(Conf)}. Let $W=\Sigma^3\times I$. 
We consider the exact sequence
\[\begin{split}
 &\,H_{i+1}(W^{\times 2};A^{\boxtimes 2})\to H_{i+1}(W^{\times 2},\Conf_2(W);A^{\boxtimes 2})\to H_i(\Conf_2(W);A^{\boxtimes 2})\\
\to &\,H_i(W^{\times 2};A^{\boxtimes 2}),
\end{split}\]
where we have $H_*(W^{\times 2};A^{\boxtimes 2})=H_*(W;A)^{\otimes 2}=(H_*(\Sigma^3;\mathfrak{g})\otimes_\C \C[t^{\pm 1}])^{\otimes 2}=0$ by (\ref{eq:acyclic}) and Lemma~\ref{lem:tensor-acyclic}. 
Letting $N(\Delta_W)$ be a closed tubular neighborhood of $\Delta_W$, we have
$H_{i+1}(W^{\times 2},\Conf_2(W);A^{\boxtimes 2})\cong H_{i+1}(N(\Delta_W),\partial N(\Delta_W);A^{\otimes 2})$
by excision. By Assumption~\ref{assum:delta-trivial}, we have
\[\begin{split}
& H_{i+1}(N(\Delta_W),\partial N(\Delta_W);A^{\otimes 2})
=H_4(D^4,\partial D^4;R)\otimes_R H_{i-3}(\Delta_W;A^{\otimes 2})\\
&\cong H_4(D^4,\partial D^4;R)\otimes_R H_{i-3}(W;R)\cong H_{i-3}(X;R).
\end{split} \]
Here, $H_{i-3}(W;R)\cong H_{i-3}(\Sigma^3;R)$ is rank 1 for $i-3=0,3$, and its generator is $*,\Sigma^3$, respectively.
\end{proof}

\begin{Lem}\label{lem:EConf(pi-I)}
We have
$H_5(E\bConf_2(\pi_I);A^{\boxtimes 2})=0$ and 
the natural map
\[ H_6(E\bConf_2(\pi_I);A^{\boxtimes 2})
\to H_6(E\bConf_2(\pi_I),\partial E\bConf_2(\pi_I);A^{\boxtimes 2}) \]
is zero.
\end{Lem}
\begin{proof}
The proof is analogous to that of Lemma~\ref{lem:E2-EConf}. Namely, the lemma follows from Lemma~\ref{lem:H(Conf(W))} and the Leray--Serre spectral sequence for $\pi_I$.
\end{proof}

There is a fibration sequence (\cite{Ce1,Pa})
\[ \Diff_0(\Sigma^3\times I,\partial)\stackrel{i}{\longrightarrow} 
\Diff_0(\Sigma^3\times S^1)\longrightarrow
\Emb_0(\Sigma^3,\Sigma^3\times S^1), \]
where the second map is given by the action to the inclusion $\Sigma^3=\Sigma^3\times \{\mathrm{pt}\}\subset \Sigma^3\times S^1$.
\begin{Thm}\label{thm:infinite-rank-restate}
Let $\Sigma^3=\Sigma(2,3,5)$.
The quotient set 
\[ \pi_0\Diff_0(\Sigma^3\times S^1)/i_*\pi_0\Diff_0(\Sigma^3\times I,\partial) \]
is countable infinite. Thus the set $\pi_0\Emb_0(\Sigma^3,\Sigma^3\times S^1)$ is countable infinite. 
\end{Thm}
\begin{proof}
We have the following commutative diagram.
\[ \xymatrix{
  \pi_1 \widetilde{B\Diff}_{\mathrm{deg}}(\Sigma^3\times I,\partial)\ar[r]^-{\widetilde{i}_*} \ar[d] \ar[rd]^-{j_*}& \pi_1 \widetilde{B\Diff}_{\mathrm{deg}}((\Sigma^3\times S^1)^\bullet,\partial) \ar[d]^-{k_*}\\
  \pi_1 B\Diff_0(\Sigma^3\times I,\partial)\ar[r]^-{i_*} & \pi_1 B\Diff_0(\Sigma^3\times S^1) 
} \]
The vertical maps are both surjective, since they are the maps induced from projections of fibrations and the diffeomorphisms in $\Diff_0(\Sigma^3\times I,\partial)$ and $\Diff_0(\Sigma^3\times S^1)$ preserves the path-components of the fibers (\S\ref{ss:moduli-sp}): they acts trivially on $\pi_0$ of the spaces of maps since the diffeomorphisms do not change the homotopy classes of maps. 

It follows from Corollary~\ref{cor:abel-infinite} and Proposition~\ref{prop:puncture} that the set $\Phi$ of the infinitely many nontrivial elements of $\pi_1 \widetilde{B\Diff}_{\mathrm{deg}}((\Sigma^3\times S^1)^\bullet,\partial)$ detected in Theorem~\ref{thm:surgery}, Corollary~\ref{cor:image-Z} is mapped to a subset of infinitely many nontrivial elements of $\pi_1 B\Diff_0(\Sigma^3\times S^1)$. Moreover, it turns out that, in the abelianization, $\Phi$ is independent of the image of $\widetilde{i}_*$ from $\pi_1 \widetilde{B\Diff}_{\mathrm{deg}}(\Sigma^3\times I,\partial)$ by Propositions~\ref{prop:support-I}, \ref{prop:A-incl}. Let 
$k_*'\colon \pi_1 \widetilde{B\Diff}_{\mathrm{deg}}((\Sigma^3\times S^1)^\bullet,\partial)/\mathrm{Im}\,\widetilde{i}_*\to \pi_1 B\Diff_0(\Sigma^3\times S^1)$ be the homomorphism induced by $k_*$. By the commutativity of the diagram above, it follows that
\[ \begin{split}
  &\pi_1 B\Diff_0(\Sigma^3\times S^1)/\mathrm{Im}\,i_*
  =\pi_1 B\Diff_0(\Sigma^3\times S^1)/\mathrm{Im}\,j_*\\
  &=k_*'\bigl(\pi_1 \widetilde{B\Diff}_{\mathrm{deg}}((\Sigma^3\times S^1)^\bullet,\partial)/\mathrm{Im}\,\widetilde{i}_* \bigr)
\end{split}\]
is an infinite set, since $\mathrm{Ker}\,k_*$ is a finite abelian group, by the identites below (see also (\ref{eq:pi-F})), and so is $\mathrm{Ker}\,k_*'=\mathrm{Ker}\,k_*/\widetilde{i}_*\mathrm{Ker}\,(k_*\circ \widetilde{i}_*)$.
\[ \pi_1(\calF_*(X)\times \Map_*^{\mathrm{\deg}}(X,X))=(\Z_2)^{\oplus 6}, \quad \pi_1(\Omega(SO_4\times X))=0 \]
(for $X=\Sigma^3\times S^1$).
\end{proof}

%%%%%%%%%%%%%%%%%%%%%%%%%%%%%%%
%%%%%%%%%%%%%%%%%%%%%%%%%%%%%%%
\mysection{More properties of $\Theta$-graph surgery}{s:pseudo-isotopy}

The following propositions hold for any compact parallelizable manifold $X$.

\begin{Prop}[{\cite[Proposition~4.2]{BW}}]\label{prop:bw-hopf}
Let $d=\dim{X}\geq 4$. The $(X,\partial)$-bundle $\pi^\Theta\colon E^\Theta\to B_\Theta$ for an embedding $\phi\colon \Theta\to X$ is obtained from the product bundle $B_\Theta\times X\to B_\Theta$ by fiberwise surgeries along a $B_\Theta$-family of framed links $h_s\colon S^1\cup S^{d-2}\to X$, $x\in B_\Theta$, which satisfies the following conditions:
\begin{enumerate}
\item[(a)] $h_s$ is isotopic to the Hopf link with standard framing in a small $d$-ball for each $s$.
\item[(b)] The restriction of $h_s$ to $S^{d-2}$ is a constant $B_\Theta$-family.
\item[(c)] There is a small neighborhood $N$ of $\mathrm{Im}\,\phi$ such that the image of $h_s$ is included in $N$ for all $s\in B_\Theta$.
\end{enumerate}
\end{Prop}
\begin{Rem}\label{rem:bw-hopf}
It is straightforward from the proof of \cite[Proposition~4.2]{BW} that the condition (b) of Proposition~\ref{prop:bw-hopf} can be replaced with that for the restriction of $h_s$ to $S^1$.
\end{Rem}

\begin{Rem}\label{rem:bw-hopf-2}
\begin{enumerate}
\item By attaching a $B_\Theta$-family of 2-handles and digging a $B_\Theta$-family of $(d-1)$-handles along $h_s$ of Proposition~\ref{prop:bw-hopf} to/from the top face $B_\Theta\times (X\times \{1\})$ of $B_\Theta\times (X\times I)$, one obtains a fiber bundle $\Pi^\Theta\colon W^\Theta\to B_\Theta=S^0\times S^{d-3}$ with fiber $W=X\times I$ and structure group $C(X)$ (see Figure~\ref{fig:cancelling-handles}). This was suggested by Peter Teichner and the detail for general trivalent graphs and fiber dimensions $d\geq 4$ can be found in \cite{BW}.
\begin{figure}%
\[ \includegraphics[height=35mm]{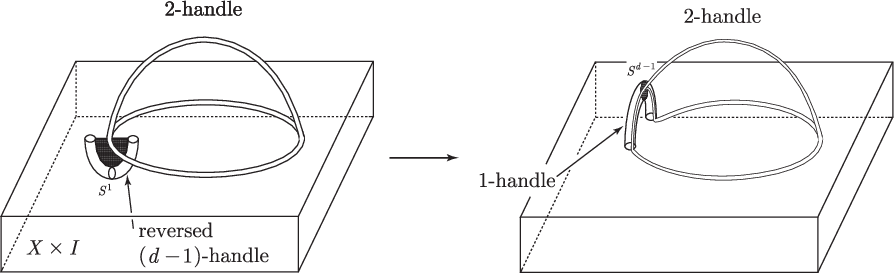} \]%
\caption{Replacing a Hopf link with a cancelling pair of handles.}\label{fig:cancelling-handles}%
\end{figure}%

\item Since digging a reversed $(d-2)$-handle from $X\times I$ gives diffeomorphic cobordism as attaching a 1-handle, the $X\times I$-bundle $\Pi^\Theta$ of the previous item admits a Cerf graphic (in the sense of \cite{Ce2}) as in Figure~\ref{fig:lens12} on each component of $B_\Theta$. Also, digging a reversed $1$-handle from $X\times I$ gives the same result as attaching a $(d-1)$-handle. Thus the corresponding $X\times I$-bundle $\Pi'^\Theta$ over $B_\Theta$ admits a Cerf graphic with a $(d-2,d-1)$-handle pair of critical values.
\begin{figure}
\begin{center}
\includegraphics{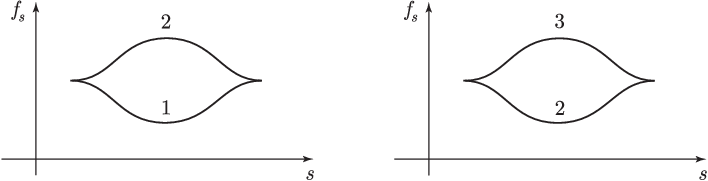}
\caption{The graphics for the $(X\times I)$-bundles with the Morse indices of the critical loci. Left: $\Pi^\Theta$. Right: $\Pi'^\Theta$.}\label{fig:lens12}
\end{center}
\end{figure}

\item When $d=4$, the $(2,3)$-handle pair giving $\Pi'^\Theta$ of the previous item can be presented by an isotopy $P_s$ ($0\leq s\leq 1$) of the attaching 2-sphere of the 3-handle in the middle level surface $X\#(S^2\times S^2)$. This isotopy can be explicitly described by applying \cite[Lemmas~3.7 and 3.10]{BW}, which is roughly as follows. The initial sphere $P_0$ has the geometrically dual sphere $Q=\mathrm{pt}\times S^2\subset X\#(S^2\times S^2)$, which is the belt-sphere of the 2-handle. During the isotopy, the intersection $P_s\cap Q$ changes as first introducing two pairs of intersections by two finger moves, and then cancelling the same pairs by two Whitney moves, reversing the two finger moves (\cite{Tei}).
\end{enumerate}
\end{Rem}

\begin{Prop}\label{prop:homotopically-trivial}
Let $d=\dim{X}\geq 4$. The element of $\Omega_{d-3}^{SO}(B\Diff(X))$ given by $\pi^\Theta\colon E^\Theta\to B_\Theta$ belongs to the image from $\Omega_{d-3}^{SO}(B\Diff_0(X))$.
\end{Prop}
\begin{proof}
The element of $\Omega_{d-3}^{SO}(B\Diff(X))$ given by $\pi^\Theta\colon E^\Theta\to B_\Theta$ is mapped from $\Omega_{d-3}^{SO}(BC(X))$ and the result follows since an element of $C(X)$ gives a continuous homotopy of a diffeomorphism of $X$ to the identity by composing with the projection $X\times I\to X$. Hence the image of the natural map $C(X)\to \Diff(X)$ is contained in $\Diff_0(X)$. 
Namely, the $X\times I$-bundle over $B_\Theta$ given by Theorem~\ref{thm:teichner} can be constructed inductively by attaching thickened cells $(X\times I)\times(\mbox{cell}\subset B_\Theta)$ by families of diffeomorphisms of $X\times I$ in $C(X)$. The induced families of diffeomorphisms of $X\times\{1\}$ are in $\Diff_0(X)$. 
\end{proof}

\begin{Prop}\label{prop:degree1}
Let $d=\dim{X}\geq 4$. Suppose $X$ is closed. Then the element of $\Omega_{d-3}^{SO}(B\Diff_0(X))$ given by 
$\Psi_1(\Theta(g_1,g_2,g_3))$ ($g_1,g_2,g_3\in\pi_1(X)$) 
belongs to the image of the projection induced map
\[ \Omega_{d-3}^{SO}(\widetilde{B\Diff}_{\mathrm{deg}}(X^\bullet,\partial))\to \Omega_{d-3}^{SO}(B\Diff_0(X^\bullet,\partial)).\]
\end{Prop}
\begin{proof}
The existence of vertical framing was shown in \cite[Corollary~3.22]{Wa2}. Thus we need only to show that a fiberwise homotopy equivalence $E^\Theta\to X$ exists. The rest is an analogue of \cite[Lemma~2.4]{GL}. Namely, each component of the family of the 2-component framed links of Proposition~\ref{prop:bw-hopf} that yields the $\Theta$-graph surgery (along $V_1\cup V_2$) is fiberwise null-homotopic in $B_\Theta\times N$ by Proposition~\ref{prop:bw-hopf} (b) and Remark~\ref{rem:bw-hopf}. By attaching handles along the family of $2$-component framed links in the top face $B_\Theta\times(X\times \{1\})$ of $B_\Theta\times(X\times I)$, we obtain a fiberwise cobordism $\widetilde{E}$ between $E^\Theta$ and $-(B_\Theta\times X)$. Since the families of attaching spheres of handles are fiberwise null-homotopic in $B_\Theta\times N$, the projection $B_\Theta\times (X\times I)\to X$ extends over the families of handles continuously and we obtain a continuous extension $\widetilde{f}\colon \widetilde{E}\to X$. We may assume moreover that $\widetilde{f}$ is smooth. That its restriction $f=\widetilde{f}|_{E^\Theta}$ is fiberwise homotopy equivalence follows since its restriction to the fiber over the basepoint of each component of $B_\Theta$ is homotopic to a map that collapses the part attached to a Hopf link in a ball into a point and thus is a homotopy equivalence. 
\end{proof}

%%%%%%%%%%%%%%%%%%%%%%%%%%%%%%
%%%%%%%%%%%%%%%%%%%%%%%%%%%%%%
\appendix

%%%%%%%%%%%%%%%%%%%%%%%%%
\section{Local models of products of singular chains}

\subsection{Intersection of two simplices}

\begin{Lem}\label{lem:d-int-sigma}
Let $\sigma_\lambda\colon \Delta^p\to R$, $\sigma_\mu\colon \Delta^q\to R$ be smooth simplices in a $r$-dimensional manifold $R$ such that $p+q=r+1$. Suppose that $\sigma_\lambda$ and $\sigma_\mu$ satisfy the conditions (a)--(d) in the proof of Lemma~\ref{lem:tr-coboundary} except that the ambient space is $R$. Then we have
\[  
\partial\langle \sigma_\lambda,\sigma_\mu\rangle
=(-1)^{r-q}\langle \partial \sigma_\lambda,\sigma_\mu\rangle 
+\langle \sigma_\lambda,\partial \sigma_\mu\rangle. 
\]
\end{Lem}
\begin{proof}
By the transversality of the intersection of the embeddings $\sigma_\lambda$ and $\sigma_\mu$, it suffices to compute the signs of the 0-dimensional intersection points in a standard local coordinate model. Namely, let
\[ \begin{split}
&P=\{(t,x,0)\in \R\times \R^{p-1}\times \R^{q-1}\mid t\leq 1\},\\
&Q=\{(t,0,y)\in \R\times \R^{p-1}\times \R^{q-1}\mid t\geq 0\},
\end{split} \]
which correspond to $\sigma_\lambda$ and $\sigma_\mu$, respectively. The intersection of the half-planes $P,Q$ is the closed interval $[0,1]$ in the $t$-axis. Let
$o=(0,0,0)$, $a=(1,0,0)\in\R\times\R^{p-1}\times \R^{q-1}=\R^r$. We put
\[ 
o(P)=\alpha\,\partial t\wedge\partial x,\quad
o(Q)=\beta\,\partial t\wedge \partial y,\quad
o(R)=\partial t\wedge\partial x\wedge\partial y,
\]
where $\alpha,\beta\in\{-1,1\}$, $\partial t=\frac{\partial}{\partial t}$, $\partial x=\frac{\partial}{\partial x_1}\wedge\cdots\wedge\frac{\partial}{\partial x_{p-1}}$, $\partial y=\frac{\partial}{\partial y_1}\wedge\cdots\wedge\frac{\partial}{\partial y_{q-1}}$. According to the outward-normal-first convention, we have
\[ o(\partial P)_a=\alpha\, \partial x,\quad 
o(\partial Q)_o=-\beta\, \partial y. \]
The coorientations of $P$ and $Q$ in $R$ are as follows:
\[ o^*_R(P)=\alpha\,\partial y,\quad
o^*_R(Q)=(-1)^{(p-1)(q-1)}\beta\,\partial x. \]
Then we have
\[ \begin{split}
&o^*_R(P\cap Q)=o^*_R(P)\wedge o^*_R(Q)=\alpha\beta\,\partial x\wedge \partial y,\\
&o(P\cap Q)=\alpha\beta\,\partial t,\\
&\partial\langle P, Q\rangle=\alpha\beta (a-o).
\end{split} \]
On the other hand, we have
\[ \begin{split}
&o^*_R(\partial P)_a=(-1)^{p-1}\alpha\,\partial t\wedge \partial y,\quad
o^*_R(\partial Q)_o=-(-1)^{p(q-1)}\beta\,\partial t\wedge \partial x,\\
&o^*_R(\partial P\cap Q)_a=o^*_R(\partial P)_a\wedge o^*_R(Q)_a=(-1)^{p-1}\alpha\beta\,\partial t\wedge\partial x\wedge \partial y,\\
&o^*_R(P\cap\partial Q)_o=o^*_R(P)_o\wedge o^*_R(\partial Q)_o=-\alpha\beta\, \partial t\wedge\partial x\wedge \partial y.
\end{split}\]
Hence
\[ \langle\partial P, Q\rangle=(-1)^{p-1}\alpha\beta\,a,\quad
\langle P,\partial Q\rangle=-\alpha\beta\,o, \quad\mbox{and}\]
\[ \partial\langle P, Q\rangle
=(-1)^{p-1}\langle\partial P, Q\rangle
+\langle P,\partial Q\rangle. \]
Since $p-1=r-q$, the result follows. 
\end{proof}

\subsection{Cross product of simplices with local coefficients}

Let $C_1=\sigma_\lambda\otimes m_\lambda$, $C_2=\sigma_\mu\otimes m_\mu$ be $A$-chains of $X$ such that $\sigma_\lambda,\sigma_\mu$ are smooth singular simplices of dimensions $p,q$ in $X$, respectively, and $m_\lambda,m_\mu\in A$. We now define the cross product $(\sigma_\lambda\otimes m_\lambda)\times (\sigma_\mu\otimes m_\mu)$. 

Let $\delta_{0,\lambda},\delta_{0,\mu}$ be the barycenters of $\Delta^p,\Delta^q$, respectively. The cross product $\sigma_\lambda\times \sigma_\mu$ can be subdivided into several $(p+q)$-simplices according to the simplicial decomposition:
\[ \Delta^p\times \Delta^q=\sum_\rho \Delta_\rho^{p+q}, \]
where $\rho$ is an ``edgepath'' in $\{(i,j)\mid 0\leq i\leq p,\,0\leq j\leq q\}$ from $(0,0)$ to $(p,q)$ (\cite[3.B (The Simplicial Cross Product)]{Hat2}). Let $\delta_{0,\rho}$ be the barycenter of $\Delta_\rho$. Let $\eta_\rho$ be a path in $\Delta^p\times \Delta^q$ from $\delta_{0,\rho}$ to $(\delta_{0,\lambda},\delta_{0,\mu})$. Then we define
\begin{equation}\label{eq:cross-sigma}
 (\sigma_\lambda\otimes m_\lambda)\times (\sigma_\mu\otimes m_\mu)
=\sum_\rho (\sigma_\lambda\times\sigma_\tau)|_{\Delta_\rho^{p+q}}\otimes \Hol((\sigma_\lambda\times\sigma_\mu)\circ\eta_\rho)(m_\lambda\otimes m_\mu). 
\end{equation}
We extend $\times$ to $A^{\boxtimes 2}$-chains of $X$ by $A^{\boxtimes 2}$-linearity. The following lemma can be proved by computing the $\partial_A$-boundary of (\ref{eq:cross-sigma}) above, which is routine. 
\begin{Lem}
The cross product defined as above is a $\partial_A$-chain map. Hence induces 
\[ \times:H_p(V;A)\otimes_R H_q(W;A)\to H_{p+q}(V\times W;A^{\boxtimes 2}),\]
where $V,W\subset X$ are submanifolds of $X$.
\end{Lem}

%%%%%%%%%%%%%%%%%%%%%%%%%%%%%%%%
\section{Irreducibility of the $\pi'$-module $V_4$}

\begin{Prop}\label{prop:V4-irred}
The $\pi'$-module $V_4=\Sym^4 V$ defined by (\ref{eq:235}) is irreducible. Namely, $H^0(\pi';V_4)\cong V_4^{\pi'}=\{v\in V_4\mid g\cdot v=v\,(\forall g\in \pi')\}$ is zero.
\end{Prop}
\begin{proof}
Let $L_1=\{f\in V_4\mid \rho(x_1)\cdot f=f\}$, $L_2=\{f\in V_4\mid \rho(x_2)\cdot f=f\}$. We show that $L_1\cap L_2=0$. It is easy to see that $L_1$ consists of polynomials of the form
\[ f=a_0x^4+ a_2x^2y^2 + a_4y^4. \]
We consider the condition for such a polynomial to belong to $L_2$. For $\alpha=\frac{1+\sqrt{5}}{4}$, $\beta=\frac{-1+\sqrt{5}}{4}$, we compute
\footnotesize
\[\begin{split}
&\rho(x_2)\cdot(a_0x^4+ a_2x^2y^2 + a_4y^4)-(a_0x^4+ a_2x^2y^2 + a_4y^4)\\
&=\frac{(3 \sqrt{5}  + 41 +4i(\sqrt{5}+3)) a_0
+ (\sqrt{5} -1)(1-2i)a_2
+(3 \sqrt{5} - 7) a_4}{32}x^4 \\
&+\frac{(-2(\sqrt{5}+2)-i(\sqrt{5}-3))a_0
-(2+i(\sqrt{5}+1))a_2
-(2(\sqrt{5}-2)+i(\sqrt{5}-3))a_4}{8}x^3y  \\
&+\frac{3(\sqrt{5}-1)(1-2i)a_0
-3(\sqrt{5}-5)a_2
+3(\sqrt{5}-1)(1+2i)a_4}{16}x^2y^2  \\
&+\frac{(2(\sqrt{5}-2)+i(3-\sqrt{5}))a_0
+(2-i(\sqrt{5}+1))a_2
+(2(\sqrt{5}+2)+i(3-\sqrt{5}))a_4
}{8}xy^3 \\
&+\frac{(3\sqrt{5}-7)a_0+(\sqrt{5}-1)(1+2i)a_2
+(3\sqrt{5}+41-4i(\sqrt{5}+3))a_4}{32}y^4.
\end{split} \]\normalsize
The resulting polynomial is zero if and only if $a_0=a_2=a_4=0$, which can be checked, for example, by 
\[\begin{split}
& \det\left(\begin{array}{ccc}
 \frac{3 \sqrt{5}  + 41 +4i(\sqrt{5}+3)}{32} &
 \frac{(\sqrt{5} -1)(1-2i)}{32} & 
 \frac{3 \sqrt{5} - 7}{32}\\
 \frac{-2(\sqrt{5}+2)+i(\sqrt{5}-3)}{8} &
 -\frac{2+i(\sqrt{5}+1)}{8} &  
 -\frac{2(\sqrt{5}-2)+i(\sqrt{5}-3)}{8}\\
 \frac{3(\sqrt{5}-1)(1-2i)}{16} &
 -\frac{3(\sqrt{5}-5)}{16} &
 \frac{3(\sqrt{5}-1)(1+2i)}{16}
\end{array}\right)\\
 &=\frac{3(\sqrt{5}-1)+3i(\sqrt{5}-3)}{8}\neq 0.
\end{split} \]
For the other case $\alpha=\frac{1+\sqrt{5}}{4}$, $\beta=\frac{1-\sqrt{5}}{4}$, we compute
\footnotesize
\[\begin{split}
&\rho(x_2)\cdot(a_0x^4+ a_2x^2y^2 + a_4y^4)-(a_0x^4+ a_2x^2y^2 + a_4y^4)\\
&=\frac{(3 \sqrt{5}  + 41 +4i(\sqrt{5}+3)) a_0
+ (\sqrt{5} -1)(1-2i)a_2
+(3 \sqrt{5} - 7) a_4}{32}x^4 \\
&+\frac{(2(\sqrt{5}+2)+i(\sqrt{5}-3))a_0
+(2+i(\sqrt{5}+1))a_2
+(2(\sqrt{5}-2)+i(\sqrt{5}-3))a_4}{8}x^3y  \\
&+\frac{3(\sqrt{5}-1)(1-2i)a_0
-3(\sqrt{5}-5)a_2
+3(\sqrt{5}-1)(1+2i)a_4}{16}x^2y^2  \\
&+\frac{-(2(\sqrt{5}-2)+i(3-\sqrt{5}))a_0
-(2-i(\sqrt{5}+1))a_2
-(2(\sqrt{5}+2)+i(3-\sqrt{5}))a_4
}{8}xy^3 \\
&+\frac{(3\sqrt{5}-7)a_0+(\sqrt{5}-1)(1+2i)a_2
+(3\sqrt{5}+41-4i(\sqrt{5}+3))a_4}{32}y^4.
\end{split} \]\normalsize
The resulting polynomial is zero if and only if $a_0=a_2=a_4=0$, which can be checked, for example, by 
\[\begin{split}
& \det\left(\begin{array}{ccc}
 \frac{3 \sqrt{5}  + 41 +4i(\sqrt{5}+3)}{32} &
 \frac{(\sqrt{5} -1)(1-2i)}{32} & 
 \frac{3 \sqrt{5} - 7}{32}\\
 \frac{2(\sqrt{5}+2)+i(\sqrt{5}-3)}{8} &
 \frac{2+i(\sqrt{5}+1)}{8} &  
 \frac{2(\sqrt{5}-2)+i(\sqrt{5}-3)}{8}\\
 \frac{3(\sqrt{5}-1)(1-2i)}{16} &
 -\frac{3(\sqrt{5}-5)}{16} &
 \frac{3(\sqrt{5}-1)(1+2i)}{16}
\end{array}\right)\\
 &=-\frac{3(\sqrt{5}-1)+3i(\sqrt{5}-3)}{8}\neq 0.
\end{split} \]
This completes the proof.
\end{proof}

%%%%%%%%%%%%%%%%%%%%%%%%%%%%%%%
%%%%%%%%%%%%%%%%%%%%%%%%%%%%%%%


\begin{thebibliography}{CCTW}

\bibitem[ArSz]{ArSz}
G.~Arone, M.~Szymik, \emph{Spaces of knotted circles and exotic smooth structures}, Canad. J. Math. {\bf 74} (1) (2022), 1--23. 

\bibitem[AD]{AD}
M.~Audin, M.~Damian, Morse Theory and Floer Homology, Universitext. Springer, London; EDP Sciences, Les Ulis, 2014. xiv+596 pp.


\bibitem[AxSi]{AxSi}
S.~Axelrod, I.~M.~Singer, \emph{Chern--Simons perturbation theory}, in Proceedings of the XXth DGM Conference, Catto S., Rocha A. (eds.), pp. 3--45, World Scientific, Singapore, 1992{ , II, J. Diff. Geom. {\bf 39} (1994), 173--213.}

\bibitem[BK]{BK}
D.~Baraglia, H.~Konno, \emph{A gluing formula for families Seiberg--Witten invariants}, Geom. Topol. {\bf 24} (2020), 1381--1456.

\bibitem[BN]{BN}
D.~Bar-Natan, \emph{On the Vassiliev knot invariants}, Topology {\bf 34} (1995), no. 2, 423--472. 

\bibitem[Bod]{Bod}
H.~Boden, \emph{Unitary representations of Brieskorn spheres}, Duke Math. J. {\bf 75}-1 (1994), 193--220.

\bibitem[BC]{BC}
R.~Bott, A.~Cattaneo, \emph{Integral invariants of 3-manifolds. II}, J. Diff. Geom. {\bf 53}-1 (1999), 1--13.

\bibitem[BT]{BT}
R.~Bott, C.~Taubes, \emph{On the self-linking of knots}, J. Math. Phys. {\bf 35}, (1994), 5247--5287. 

\bibitem[BW]{BW} B.~Botvinnik, T.~Watanabe, \emph{Families of diffeomorphisms and concordances detected by trivalent graphs}, J. Topol. {\bf 16} (1) (2023), 207--233.

\bibitem[BG]{BG}
R.~Budney, D.~Gabai, \emph{Knotted 3-balls in $S^4$}, arXiv:1912.09029.

\bibitem[BH]{BH}
D.~Burghelea, S.~Haller, \emph{On the topology and analysis of a closed one form I}, Monogr. Enseign. Math., {\bf 38} (2001), 133--175.

\bibitem[BRW]{BRW}
M.~Bustamante, O.~Randal-Williams, \emph{On automorphisms of high-dimensional solid tori}, Geom. Topol., to appear, arXiv:2010.10887.

\bibitem[CE]{CE} 
H.~Cartan, S.~Eilenberg, Homological algebra, Princeton University Press, Princeton, N. J., 1956. xv+390 pp.

\bibitem[CS]{CS}
A.~Cattaneo, T.~Shimizu, \emph{A note on the $\Theta$-invariant of 3-manifolds}, Quantum Topol. {\bf 12} (1) (2021), 111--127.

\bibitem[Ce1]{Ce1} J.~Cerf, \emph{Topologie de certains espaces de plongements}, Bull. Soc. Math. France, {\bf 89} (1961), 227--380. 

\bibitem[Ce2]{Ce2} J.~Cerf, \emph{La stratification naturelle des espaces de fonctions diff\'{e}rentiables r\'eelles et le th\'eor\`eme de la pseudo-isotopie}, Publ. Math. I.H.\'E.S. \textbf{39} (1970), 5--173.

\bibitem[CDM]{CDM}
S.~Chmutov, S.~Duzhin, J.~Mostovoy, Introduction to Vassiliev knot invariants. Cambridge University Press, Cambridge, 2012. xvi+504 pp. 

\bibitem[CVa]{CVa} 
S.~Chmutov, A.~Varchenko, \emph{Remarks on the Vassiliev knot invariants coming from $sl_2$}, Topology {\bf 36} (1997), 153--178.

\bibitem[FS]{FS}
R.~Fintushel, R.~Stern, \emph{Instanton homology of Seifert fibred homology three spheres}, Proc. Lond. Math. Soc. (3) {\bf 61} (1990), 109--137.

\bibitem[Fuk]{Fuk}
K.~Fukaya, \emph{Morse homotopy and Chern-Simons perturbation theory}, Comm. Math. Phys. {\bf 181}-1 (1996), 37--90.

\bibitem[Ga]{Ga}
D.~Gabai, \emph{The 4-dimensional light bulb theorem}, J. Amer. Math. Soc. {\bf 33} (2020), 609--652. 

\bibitem[GGP]{GGP} S.~Garoufalidis, M.~Goussarov, M.~Polyak, \emph{Calculus of clovers and finite type invariants of 3-manifolds}, Geom. Topol. {\bf 5}, no. 1 (2001), 75--108.

\bibitem[GL]{GL} S.~Garoufalidis, J.~Levine, \emph{Homology surgery and invariants of 3–manifolds}, Geom. Topol. {\bf 5} (2001), 551--578.

\bibitem[Gu]{Gu} M.~Gusarov, \emph{Variations of knotted graphs. The geometric technique of $n$-equivalence}, Algebra i Analiz {\bf 12} (4) (2000), 79--125. English version: St. Petersburg Math. J. {\bf 12} (4) (2001), 569--604.

\bibitem[Hab]{Hab} K.~Habiro, \emph{Claspers and finite type invariants of links}, Geom. Topol. {\bf 4} (2000), 1--83.

\bibitem[Hat]{Hat}
A.~Hatcher, \emph{Concordance spaces, higher simple-homotopy theory, and applications}, Proc. Sympos. Pure Math., {\bf 32} Part 1 (1978), pp. 3--21, Amer. Math. Soc., Providence, R.I.

\bibitem[Hat2]{Hat2}
A.~Hatcher, Algebraic Topology, Cambridge University Press, Cambridge, 2002. xii+544 pp.

\bibitem[Hat3]{Hat3}
A.~Hatcher, \emph{Spectral Sequences}, Chapter 5 of \emph{Algebraic Topology}. 

\bibitem[Hatt]{Hatt}
A.~Hattori, Iso kikagaku. I-III. (Japanese) [Topology. I-III] Second edition. Iwanami Shoten Kiso S\={u}gaku [Iwanami Lectures on Fundamental Mathematics], 19. Dais\={u} [Algebra], ii. Iwanami Shoten, Tokyo, 1982. Vol. I: vi+132 pp.; Vol. II: pp. i--vi and 133--308; Vol. III: pp. i--viii and 309--495. 

\bibitem[HW]{HW}
A.~Hatcher, J.~Wagoner, \emph{Pseudo-isotopies of compact manifolds} Ast\'{e}risque {\bf 6} (1973).

\bibitem[HS]{HS}
P.J.~Hilton, U.~Stammbach, A course in homological algebra. Second edition, Graduate Texts in Mathematics, {\bf 4}, Springer-Verlag, New York, 1997. xii+364 pp.

\bibitem[Igu]{Igu}
K.~Igusa, \emph{Second obstruction to pseudoisotopy I}, arXiv:2110.09659.


\bibitem[Igu2]{Igu2}
K.~Igusa, \emph{Second obstruction to pseudoisotopy in dimension 3}, arXiv:2112.08293.

\bibitem[Kas]{Kas}
C.~Kassel, Quantum groups, Graduate Texts in Mathematics, {\bf 155}, Springer-Verlag, New York, 1995. xii+531 pp.

\bibitem[Kna]{Kna}
A.W.~Knapp, Lie groups beyond an introduction. Second edition. Progress in Mathematics, {\bf 140}. Birkhäuser Boston, Inc., Boston, MA, 2002. xviii+812 pp. 

\bibitem[Kon]{Kon} M.~Kontsevich, \emph{Feynman diagrams and low-dimensional topology}, First European Congress of Mathematics, Vol. II (Paris, 1992), Progr. Math. {\bf 120} (Birkhauser, Basel, 1994), 97--121.

\bibitem[Kre]{Kre} M.~Kreck, Differential Algebraic Topology: From Stratifolds to Exotic Spheres, Grad. Stud. in Math. {\bf 110}, Amer. Math. Soc., Providence, RI, 2010. xii+218 pp. 

\bibitem[KM]{KM}
P.~Kronheimer, T.~Mrowka, \emph{The Dehn twist on a sum of two K3 surfaces}, 
Math. Res. Lett. {\bf 27} (6) (2020), 1767--1783.

\bibitem[KT]{KT} G.~Kuperberg, D.~P.~Thurston, \emph{Perturbative 3-manifold invariants by cut-and-paste topology}, arXiv:math/9912167.

\bibitem[LM]{LM}
E.~Laitinen, I.~Madsen, \emph{Topological classifications of $\mathrm{SL}_2(\mathbb{F}_p)$ space forms}, Algebraic topology, Aarhus 1978 (Proc. Sympos., Univ. Aarhus, Aarhus, 1978), 235--261, Lecture Notes in Math., {\bf 763}, Springer, Berlin, 1979. 

\bibitem[Les1]{Les1}
C.~Lescop, \emph{On the cube of the equivariant linking pairing for knots and 3-manifolds of rank one}, arXiv:1008.5026.

\bibitem[Les2]{Les2}
C.~Lescop, \emph{Invariants of knots and 3-manifolds derived from the equivariant linking pairing}, Chern--Simons gauge theory: 20 years after, AMS/IP Stud. Adv. Math. {\bf 50} (2011), Amer. Math. Soc., Providence, RI, 217--242.

\bibitem[Mar]{Mar}
J.~March\'{e}, \emph{An equivariant Casson invariant of knots in homology spheres}, preprint, 2005.

\bibitem[Pa]{Pa} R.~Palais, \emph{Local triviality of the restriction map for embeddings}, Comment. Math. Helv. {\bf 34} (1960), 305--312.

\bibitem[Ru]{Ru}
D.~Ruberman, \emph{An obstruction to smooth isotopy in dimension 4}, 
Math. Res. Lett. {\bf 5}-6 (1998), 743--758.

\bibitem[Sa1]{Sa1}
N.~Saveliev, Lectures on the Topology of 3-Manifolds: An Introduction To The Casson Invariant, Walter de Gruyter, 1999.

\bibitem[Sa2]{Sa2}
N.~Saveliev, Invariants for homology 3-spheres. Encyclopaedia of Mathematical Sciences, {\bf 140}, Low-Dimensional Topology, I. Springer-Verlag, Berlin, 2002. xii+223 pp. 

\bibitem[Sin]{Sin}
O.~Singh, \emph{Pseudo-isotopies and diffeomorphisms of 4-manifolds}, arXiv:2111.15658.

\bibitem[Tei]{Tei}
P.~Teichner, Private communication.

\bibitem[Wa1]{Wa1} T.~Watanabe, \emph{Some exotic nontrivial elements of the rational homotopy groups of $\Diff(S^4)$}, arXiv:1812.02448.

\bibitem[Wa2]{Wa2} T.~Watanabe, \emph{Addendum to the paper: Some
  exotic nontrivial elements of the rational homotopy groups of
  $\Diff(S^4)$}, arXiv:2109.01609. 

\bibitem[Wh]{Wh} G.W.~Whitehead, \emph{Elements of Homotopy Theory}, Graduate Texts in Mathematics {\bf 61}, Springer-Verlag, New York-Berlin, 1978. xxi+744 pp.

\end{thebibliography}
\end{document}